\documentclass[12pt,psamsfonts]{amsart}
\usepackage{amssymb,amsmath,amscd,eucal,epsfig}
\usepackage{tikz}

 \def\Dj{\hbox{D\kern-.73em\raise.30ex\hbox{-} \raise-.30ex\hbox{}}}
 \def\dj{\hbox{d\kern-.33em\raise.80ex\hbox{-} \raise-.80ex\hbox{\kern-.40em}}}

\def\id{\operatorname{id}}
\def\im{\operatorname{im}}
\def\ad{\operatorname{ad}}
\def\tr{\operatorname{tr}}
\def\Ric{\operatorname{Ric}}
\def\ric{\operatorname{ric}}        %added
\def\rk{\operatorname{rk}}          %added
\def\Span{\operatorname{Span}}      %added
\def\es{\emptyset}
\def\sq{\subseteq}
\def\<{\langle}                     %added
\def\>{\rangle}                     %added
\def\ip{\<\cdot,\cdot\>}            %added

\def\N{\mathbb N}
\def\Z{\mathbb Z}
\def\R{\mathbb R}
\def\Q{\mathbb Q}
\newtheorem{thm}{Theorem}[section]
\newtheorem{prop}[thm]{Proposition}
\newtheorem{lem}[thm]{Lemma}
\newtheorem{cor}[thm]{Corollary}
\newtheorem{rem}[thm]{Remark}
\newtheorem{conj}[thm]{Conjecture}
\newcommand{\ben}{\begin{enumerate}}
\newcommand{\een}{\end{enumerate}}

\renewcommand{\theenumi}{\alph{enumi}}
\renewcommand{\labelenumi}{(\theenumi)}
\renewcommand{\theenumii}{\roman{enumii}}

\theoremstyle{plain}

\theoremstyle{definition}

\numberwithin{equation}{section}

\begin{document}
\setcounter{page}{1}

\title[Labeling Products of Cycles]{$L(1,1)-$ Labeling of Direct Product of Cycles}

\author[T.C Adefokun]{Tayo Charles Adefokun$^1$ }
\address{$^1$Department of Computer and Mathematical Sciences,
\newline \indent Crawford University,
\newline \indent Nigeria}
\email{tayoadefokun@crawforduniversity.edu.ng; tayo.adefokun@gmail.com}

\author[D.O. Ajayi]{Deborah Olayide Ajayi$^2$}
\address{$^2$Department of Mathematics,
\newline \indent University of Ibadan,
\newline \indent Ibadan,
\newline \indent Nigeria}
\email{olayide.ajayi@mail.ui.edu.ng; adelaideajayi@yahoo.com}

\keywords{$L(1,1)$-labeling, D-2 Coloring, Direct Product of Graphs, Cross Product of Graphs, cycle graphs \\
\indent 2010 {\it Mathematics Subject Classification}. Primary: 05C35}

\begin{abstract} An $L(1,1)$-labeling of a graph $G$ is an assignment of labels from $\{0,1 \cdots, k \}$ to the vertices of $G$ such that two vertices that are adjacent or have a common neighbor receive distinct labels. The $\lambda_1^1-$ number,  $\lambda_1^1(G)$ of $G$ is the minimum value $k$ such that $G$ admits an $L(1,1)$ labeling. We establish the $\lambda_1^1-$ numbers for direct product of cycles $C_m \times C_n$ for all positive $m, n \geq 3$, where both $m,n$ are even or when one of them is even and the other odd.
\end{abstract}

\maketitle

%%%%%%%%%%%%%%%%%%%%%%%%%%%%%%%%%%%%%%%%%%%%%%%%%%%%%%%%%%%%%%%%%%%%%%
\section{Introduction}
%%%%%%%%%%%%%%%%%%%%%%%%%%%%%%%%%%%%%%%%%%%%%%%%%%%%%%%%%%%%%%%%%%%%%%
The $L(h,k)$-labeling problem (or $(h,k)$-coloring problem) is that of vertex labeling of an undirected graph $G$ with non-negative integers such that for every $u,v \in V(G)$, $uv \in E(G)$, $\left|l(u)-l(v)\right|\geq h$ and for all $u,v\in V(G), d(u,v)=2$, $\left|l(u)-l(v)\right|\geq k$.
%The $L(h,k)$ problem is that of vertex labeling of an undirected graph $G$ with non-negative integers such that adjacent vertices have labels which differ by at least $h$, and vertices at distance two have labels which differ by at least $k$.
The difference between the largest label and the smallest label assigned is called the span. The aim of $L(h,k)-$labeling is to obtain the smallest non negative integer $\lambda_h^k(G),$ such that there exists an $L(h,k)$-labeling of $G$ with no label on $V(G)$ greater than $\lambda_h^k(G)$.

 Motivated by Hales' 1980 paper \cite{H11}, which provided a new model for frequency assignment problems as a graph coloring problem, Griggs and Yeh \cite{GY10} formulated the $L(2,1)$ problem to model the channel assignment problem. The general notion of $L(h,k)$- labeling was first presented by Georges and Mauro \cite{GM7} in 1995.  The topic has since then been an object of extensive research for various graphs. Calamonerri's survey paper \cite{C2} contains known results on $L(h,k)$-labeling of graphs.

$L(1,1)$-labeling (or strong labeling condition) of a graph is a labeling of $G$ such that vertices with a common neighbor are assigned distinct labels. The usual labeling (or proper vertex coloring) condition is that adjacent vertices have different colors, but for $L(1,1)$, also all neighbors of any vertex are colored differently. This is equivalent to a proper vertex-coloring  of the square of a graph $G$. Note that a proper $k$- coloring of a graph is a mapping $\alpha: V(G) \rightarrow \{1, \cdots, k \}$ such that for all $uv \in E(G)$  $\alpha(u) \neq \alpha(v) $ and the square $G^2$ of $G$ has vertex $V(G)$ with an edge between two vertices which are adjacent in $G$ or have a common neighbor in $G$. The chromatic number $\chi(G)$ of $G$ is the smallest $k$ for which $G$ admits a $k$-coloring. Therefore, $\chi(G^2)=\lambda_1^1(G)+1$ for a graph $G$.

Labeling of graph powers is often motivated by applications in frequency assignment and has attracted much attention [See for example, \cite{AH}].  $L(1,1)$-labeling has applications in computing approximation to sparse Hessian matrices, design of collision-free multi-hop channel access protocols in radio networks segmentation problem for files in a network and drawings of graphs in the plane \cite{{B88},{KL},{M1},{MMZ}} to mention a few.

For graphs $G$ and $H$, the direct product $G \times H$ have vertex set $V(G) \times V(H)$ where two vertices $(x_1, x_2)$ and $(y_1, y_2)$ are adjacent if and only if $(x_1, y_1) \in E(G)$ and $(x_2, y_2) \in E(H)$. This product is one of the most important graph products with potential applications in engineering, computer science and related disciplines. \cite{JKV13}.  The $L(h,k)$- labeling of direct product of graphs was investigated in \cite{{AA1},{CPP3},{HJ},{JKV},{KRS1},{SW1},{SW18},{WP19}}.

In particular, Jha et al \cite{JKV} gave upper bounds for $\lambda_1^k$-labeling of the multiple direct product and cartesian product of cycles with some conditions on $k$ and the length of the cycles. They also presented some cases where we have exact values.  In addition by using backtracking algorithm, they computed $\lambda_1^d (C_m \times C_n)$ for $2\leq d \leq 4 $ and $4 \leq m,n\leq 10$. Since every $L(2,1)$-labeling is an $L(1,1)$-labeling, then $\lambda_1^1(G) \leq \lambda_1^2(G)$. Therefore, their results for $d=2$ provided upper bounds for $L(1,1)$-labeling of $C_m \times C_n$ for $4 \leq m,n\leq 10$.    The only result for $\lambda_1^1$-labeling for direct product of two cycles in the paper is that if $m, n \equiv 0 mod 5$ then  $\lambda_1^1(C_m \times C_n) = 4$.

In this paper, we solve the $L(1,1)$-labeling problem for direct product of cycles $C_m$, $C_n$, $m, n  \geq 3 $, except for $m \in \left\{16,18,22,26,32,36,46\right\}$, $n \in \left\{14,16,18,26,28,34\right\}$ and for these outstanding cases we conjecture that $\lambda_1^1(C_m \times C_n) =5$.

 The paper is organized as follows: We give some preliminaries in Section 2 and obtain the $\lambda_1^1$ labeling numbers for $C_m \times C_n$ for $m\geq 3$ and $n =4 \ {\rm and}\ 6$ and some of their multiples in Section 3. Section 4 deals with labeling of  direct product of bigger cycles.

\section{Preliminaries}
Let $G$ be a finite simple undirected graph with at least two vertices.  For subgraph $V' \subseteq V(G)$, we denote by $L(V')$ the set of $L(1,1)-$labeling on $V'$ and for a non-negative integer, say, $k$, we take $[k(\epsilon)]$ as the set of even integers and zero in $[k]$ while $[k(o)]$ is the set of odd integers in $[k]$. Suppose further that $v \in V(G)$, we denote $d_v$ as the degree of $v$.

The following results, remarks and definitions are needed in the work.
\begin{thm} \label{a}\cite{IK12} Graph $G\times H$ is connected if and only if $G$ and $H$ are connected and at least one of $G$ and $H$ is non-bipartite.
\end{thm}

\begin{rem}\begin{itemize}
\item[(i)] \label{b} Let $G=C_m \times C_n$, where $m,n$ are even positive integers. Then, $G$= $G_1 \cup G_2$, where $G_1$ and $G_2$ are the connected components of $C_m \times C_n$, where
\begin{center}
$\left\{V(G_1)=u_iv_j: i\in [(m-1)(\epsilon)], j\in [(n-1)(\epsilon)] \; {\rm or} \; \right.$ \ \\ $\left. i \in [(m-1)(o)]; j\in [(n-1)(0)]\right\}$ and
$\left\{V(G_2)=u_iv_j: i\in [(m-1)(\epsilon)], j\in [(n-1)(o)] \;{\rm or} \; \right.$ \ \\ $\left. i \in [(m-1)(o)]; j\in [(n-1)(\epsilon)]\right\}.$
\end{center}
%\end{rem}
Note that $G_1$ and $G_2$ are isomorphic and it is demonstrated in the graph $C_4 \times C_6$ below\\

{\tiny{
\begin{center}
\pgfdeclarelayer{nodelayer}
\pgfdeclarelayer{edgelayer}
\pgfsetlayers{nodelayer,edgelayer}
\begin{tikzpicture}
%\centering
	\begin{pgfonlayer}{nodelayer}
	
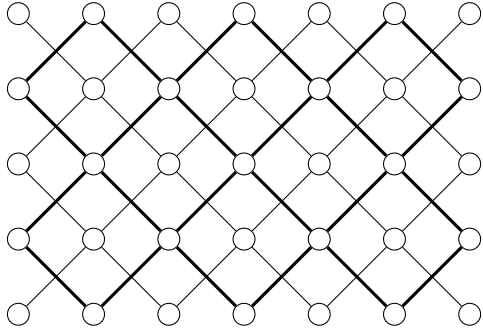
	\node [minimum size=0cm,]  at (-13,6.5) {Fig. 1: The components of $C_4 \times C_6$};

		\node [minimum size=0cm,draw,circle] (0) at (-16,7) {};
		\node [minimum size=0cm,draw,circle] (1) at (-14,7) {};
		\node [minimum size=0cm,draw,circle] (2) at (-12,7) {};
		\node [minimum size=0cm,draw,circle] (3) at (-10,7) {};
		\node [minimum size=0cm,draw,circle] (4) at (-15,8) {};
		\node [minimum size=0cm,draw,circle] (5) at (-13,8) {};
		\node [minimum size=0cm,draw,circle] (6) at (-11,8){};
		\node [minimum size=0cm,draw,circle] (7) at (-16,9) {};
		\node [minimum size=0cm,draw,circle] (8) at (-14,9) {};1
		\node [minimum size=0cm,draw,circle] (9) at (-12,9) {};
		\node [minimum size=0cm,draw,circle] (10) at (-10,9) {};
		\node [minimum size=0cm,draw,circle] (11) at (-15,10) {};
		\node [minimum size=0cm,draw,circle] (12) at (-13,10) {};
		\node [minimum size=0cm,draw,circle] (13) at (-11,10) {};
		\node [minimum size=0cm,draw,circle] (14) at (-16,11) {};
		\node [minimum size=0cm,draw,circle] (15) at (-14,11) {};
		\node [minimum size=0cm,draw,circle] (16) at (-12,11) {};
		\node [minimum size=0cm,draw,circle] (17) at (-10,11) {};
		
		\node [minimum size=0cm,draw,circle] (0a) at (-15,7) {};
		\node [minimum size=0cm,draw,circle] (1a) at (-13,7) {};
		\node [minimum size=0cm,draw,circle] (2a) at (-11,7) {};
		
		\node [minimum size=0cm,draw,circle] (3a) at (-16,8) {};
		\node [minimum size=0cm,draw,circle] (4a) at (-14,8) {};
		\node [minimum size=0cm,draw,circle] (5a) at (-12,8) {};
		\node [minimum size=0cm,draw,circle] (6a) at (-10,8) {};
		
		\node [minimum size=0cm,draw,circle] (7a) at (-15,9) {};
		\node [minimum size=0cm,draw,circle] (8a) at (-13,9) {};
		\node [minimum size=0cm,draw,circle] (9a) at (-11,9) {};

		\node [minimum size=0cm,draw,circle] (10a) at (-16,10) {};
		\node [minimum size=0cm,draw,circle] (11a) at (-14,10) {};
		\node [minimum size=0cm,draw,circle] (12a) at (-12,10) {};
		\node [minimum size=0cm,draw,circle] (13a) at (-10,10) {};
		
		\node [minimum size=0cm,draw,circle] (14a) at (-15,11) {};
		\node [minimum size=0cm,draw,circle] (15a) at (-13,11) {};
		\node [minimum size=0cm,draw,circle] (16a) at (-11,11) {};

		\end{pgfonlayer}
	\begin{pgfonlayer}{edgelayer}
		\draw [thin=1.00] (0) to (4);
		\draw [thin=1.00] (1) to (4);
		\draw [thin=1.00] (1) to (5);
		\draw [thin=1.00] (2) to (5);
		\draw [thin=1.00] (2) to (6);
		\draw [thin=1.00] (3) to (6);
		\draw [thin=1.00] (4) to (7);
		\draw [thin=1.00] (4) to (8);
		\draw [thin=1.00] (5) to (8);
		\draw [thin=1.00] (5) to (9);
		\draw [thin=1.00] (6) to (9);
		\draw [thin=1.00] (6) to (10);
		\draw [thin=1.00] (7) to (11);
		\draw [thin=1.00] (8) to (11);
		\draw [thin=1.00] (8) to (12);
		\draw [thin=1.00] (9) to (12);
		\draw [thin=1.00] (9) to (13);
		\draw [thin=1.00] (10) to (13);
		\draw [thin=1.00] (11) to (14);
		\draw [thin=1.00] (11) to (15);
		\draw [thin=1.00] (12) to (15);
		\draw [thin=1.00] (12) to (16);
		\draw [thin=1.00] (13) to (16);
		\draw [thin=1.00] (13) to (17);

		\draw [very thick=3.00] (0a) to (3a);
		\draw [very thick=3.00] (0a) to (4a);
		\draw [very thick=3.00] (1a) to (4a);
		\draw [very thick=3.00] (1a) to (5a);
		\draw [very thick=3.00] (2a) to (5a);
		\draw [very thick=3.00] (2a) to (6a);
		\draw [very thick=15.00] (3a) to (7a);
		\draw [very thick=1.00] (4a) to (7a);
		\draw [very thick=1.00] (4a) to (8a);
		\draw [very thick=1.00] (5a) to (8a);
		\draw [very thick=1.00] (5a) to (9a);
		\draw [very thick=1.00] (6a) to (9a);
		\draw [very thick=1.00] (7a) to (10a);
		\draw [very thick=1.00] (7a) to (11a);
		\draw [very thick=1.00] (8a) to (11a);
		\draw [very thick=1.00] (8a) to (12a);
		\draw [very thick=1.00] (9a) to (12a);
		\draw [very thick=1.00] (9a) to (13a);
		\draw [very thick=1.00] (10a) to (14a);
		\draw [very thick=1.00] (11a) to (14a);
		\draw [very thick=1.00] (11a) to (15a);
		\draw [very thick=1.00] (12a) to (15a);
		\draw [very thick=1.00] (12a) to (16a);
		\draw [very thick=1.00] (13a) to (16a);

	\end{pgfonlayer}
\end{tikzpicture}
\end{center}
}}
%\begin{rem}\label{c}
\item[(ii)] \label{c} Suppose $G=C_m \times C_n$ such that $G=G'\cup G''$, where $G', G''$ are components of $G$, then, $\lambda_1^1(G)=max\left\{\lambda_1^1(G'),\lambda_1^1(G'')\right\}.$
%\end{rem}
%\begin{rem}
\item[(iii)] \label{d} Let $G=C_m \times C_n$, where $m$ is even and $n$ odd positive integers. Then, $G \equiv G_1 $, where $G_1$ is any of the two connected components of $C_m \times C_{2n}$.
{\tiny{
\begin{center}
\pgfdeclarelayer{nodelayer}
\pgfdeclarelayer{edgelayer}
\pgfsetlayers{nodelayer,edgelayer}
\begin{tikzpicture}
%\centering
	\begin{pgfonlayer}{nodelayer}
	
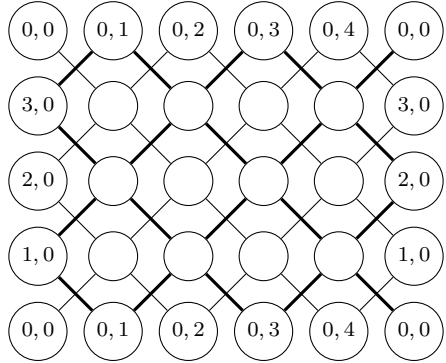
	\node [minimum size=0cm,]  at (-13.5,6) {Fig. 2: $C_4 \times C_5$ is isomorphic to a component of $C_4 \times {10}$};

		\node [minimum size=0cm,draw,circle] (0) at (-16,7) {${0,0}$};
		\node [minimum size=0cm,draw,circle] (1) at (-14,7) {$  {0,2}$};
		\node [minimum size=0cm,draw,circle] (2) at (-12,7) {$  {0,4}$};

		\node [minimum size=0.65cm,draw,circle] (4) at (-15,8) {};
		\node [minimum size=0.65cm,draw,circle] (5) at (-13,8) {};
		\node [minimum size=0cm,draw,circle] (6) at (-11,8) {$  {1,0}$};

		\node [minimum size=0cm,draw,circle] (7) at (-16,9) {$  {2,0}$};
		\node [minimum size=0.65cm,draw,circle] (8) at (-14,9) {};
		\node [minimum size=0.65cm,draw,circle] (9) at (-12,9) {};

		\node [minimum size=0.65cm,draw,circle] (11) at (-15,10) {};
		\node [minimum size=0.65cm,draw,circle] (12) at (-13,10) {};
		\node [minimum size=0cm,draw,circle] (13) at (-11,10) {$  {3,0}$};

		\node [minimum size=0cm,draw,circle] (14) at (-16,11) {$  {0,0}$};
		\node [minimum size=0cm,draw,circle] (15) at (-14,11) {$  {0,2}$};
		\node [minimum size=0cm,draw,circle] (16) at (-12,11) {$  {0,4}$};

		\node [minimum size=0cm,draw,circle] (0a) at (-15,7) {$  {0,1}$};
		\node [minimum size=0cm,draw,circle] (1a) at (-13,7) {$  {0,3}$};
		\node [minimum size=0cm,draw,circle] (2a) at (-11,7) {$  {0,0}$};

		\node [minimum size=0cm,draw,circle] (3a) at (-16,8) {$  {1,0}$};
		\node [minimum size=0.65cm,draw,circle] (4a) at (-14,8) {};
		\node [minimum size=0.65cm,draw,circle] (5a) at (-12,8) {};

		\node [minimum size=0.65cm,draw,circle] (7a) at (-15,9) {};
		\node [minimum size=0.65cm,draw,circle] (8a) at (-13,9) {};
		\node [minimum size=0cm,draw,circle] (9a) at (-11,9) {$  {2,0}$};

		\node [minimum size=0cm,draw,circle] (10a) at (-16,10) {$  {3,0}$};
		\node [minimum size=0.65cm,draw,circle] (11a) at (-14,10) {};
		\node [minimum size=0.65cm,draw,circle] (12a) at (-12,10) {};

		\node [minimum size=0cm,draw,circle] (14a) at (-15,11) {$  {0,1}$};
		\node [minimum size=0cm,draw,circle] (15a) at (-13,11) {$  {0,3}$};
		\node [minimum size=0cm,draw,circle] (16a) at (-11,11) {$  {0,0}$};

		\end{pgfonlayer}
	\begin{pgfonlayer}{edgelayer}
		\draw [thin=1.00] (0) to (4);
		\draw [thin=1.00] (1) to (4);
		\draw [thin=1.00] (1) to (5);
		\draw [thin=1.00] (2) to (5);
		\draw [thin=1.00] (2) to (6);
	%	\draw [thin=1.00] (3) to (6);
		\draw [thin=1.00] (4) to (7);
		\draw [thin=1.00] (4) to (8);
		\draw [thin=1.00] (5) to (8);
		\draw [thin=1.00] (5) to (9);
		\draw [thin=1.00] (6) to (9);
		%\draw [thin=1.00] (6) to (10);
		\draw [thin=1.00] (7) to (11);
		\draw [thin=1.00] (8) to (11);
		\draw [thin=1.00] (8) to (12);
		\draw [thin=1.00] (9) to (12);
		\draw [thin=1.00] (9) to (13);
	%	\draw [thin=1.00] (10) to (13);
		\draw [thin=1.00] (11) to (14);
		\draw [thin=1.00] (11) to (15);
		\draw [thin=1.00] (12) to (15);
		\draw [thin=1.00] (12) to (16);
		\draw [thin=1.00] (13) to (16);
		%\draw [thin=1.00] (13) to (17);

		\draw [very thick=3.00] (0a) to (3a);
		\draw [very thick=3.00] (0a) to (4a);
		\draw [very thick=3.00] (1a) to (4a);
		\draw [very thick=3.00] (1a) to (5a);
		\draw [very thick=3.00] (2a) to (5a);
		%\draw [very thick=3.00] (2a) to (6a);
		\draw [very thick=15.00] (3a) to (7a);
		\draw [very thick=1.00] (4a) to (7a);
		\draw [very thick=1.00] (4a) to (8a);
		\draw [very thick=1.00] (5a) to (8a);
		\draw [very thick=1.00] (5a) to (9a);
		%\draw [very thick=1.00] (6a) to (9a);
		\draw [very thick=1.00] (7a) to (10a);
		\draw [very thick=1.00] (7a) to (11a);
		\draw [very thick=1.00] (8a) to (11a);
		\draw [very thick=1.00] (8a) to (12a);
		\draw [very thick=1.00] (9a) to (12a);
		%\draw [very thick=1.00] (9a) to (13a);
		\draw [very thick=1.00] (10a) to (14a);
		\draw [very thick=1.00] (11a) to (14a);
		\draw [very thick=1.00] (11a) to (15a);
		\draw [very thick=1.00] (12a) to (15a);
		\draw [very thick=1.00] (12a) to (16a);
		%\draw [very thick=1.00] (13a) to (16a);

	\end{pgfonlayer}
\end{tikzpicture}
\end{center}
}}

\end{itemize}
\end{rem}

Let $P_m$ be a path of length $m-1$. The following results are from \cite{AA1}:

\begin{cor}\label{e} For $m \geq 3$, $\lambda_1^1(P_m \times C_6)=5$
\end{cor}

\begin{lem} \label{f} For $m \geq 5$, $n \geq 9$, $n \not\equiv 0 \mod 5$ $\lambda_1^1(P_m \times C_n) \geq 5$
\end{lem}

A useful lower bound on $L(1,1)$-labeling for any graph $G$ is contained in the following Lemma:
\begin{lem}\cite{GM7} If $G$ is a graph with maximum degree $\triangle$, and $G$ includes a vertex with $\triangle$ neighbors, each of which is of degree $\triangle$, then $\lambda_1^1(G) \geq \triangle$
\end{lem}
From the lemma, we have for $m, n \geq 3$ $\lambda_1^1(C_m \times C_n) = 4 $

\section{Labeling of $C_m \times C_n, n= 4,6$}
In this section, we investigate the $\lambda-$ numbers of graph product $C_m \times C_4$ and $C_m \times C_6$, where $m \geq 3$.

Let $G'$ be the connected component of the product graph under consideration.

\begin{lem} \label{cc1} For $m \geq 4,$ and even, $\lambda_1^1(C_m \times C_4) \geq 5.$
\end{lem}
\begin{proof} Let $G' \subset C_m \times C_4$, where $m \geq 4$ and even.  Suppose $V_i, V_{i+1}, V_{i+2} \subset V(G')$. Let $G'_1$ be the subgraph of $G'$ induced by $V_{i+j}$, for all $j \in [2].$ Then, $V(G'_1)$=$\{ u_iv_0, u_iv_2,   u_{i+1}v_1, \ \\ u_{i+1}v_2, u_{i+2}v_0, u_{i+2}v_2 \}$. Now it is clear that the diameter of $G'_1 $ is 2. Thus for every pair $v_1,v_2 \in V(G'_1),$ $d(v_1,v_2) \leq 2$. Thus, $l(v_1) \neq l(v_2)$ for all $v_1,v_2 \in V(G'_1).$ Now, $\left|V(G'_1)\right|=6$. Therefore $\lambda_1^1(C_m \times C_4) \geq 5$.
\end{proof}
\begin{rem} \label{cc2} Note that if $G' \subset C_m \times C_4$, $m \geq 4$ with $m$ even and $v_i \in V_i,$ for some $i, \; V_i \subset V(G'),$ such that $l(v_i)=\alpha_i \in [m],$ with $\lambda_1^1(G')=m,$ then $\alpha_1 \notin L\left\{V_{i-2}V_{i-1}V_{i+1}V_{i+2}\right\}.$
\end{rem}

\begin{thm} \label{cc3} $\lambda_1^1(C_4 \times C_4)=7$

\end{thm}
\begin{proof} Let $G' \subset C_4 \times C_4$ and $V_i \subset V(G')$ for each $i \in [3]$. Clearly, $V(G')= \cup^3_{i=0}V_i$. Let $G'_1$ be a subgraph of $G'$ such that $G'_1$ is induced by $V_0,V_1,V_2.$ By the proof of Lemma \ref{cc1}, $\left|V(G'_1)\right|=6$ and suppose $\alpha'_k,\alpha''_k, \in L(V_3)$, then by remark \ref{cc2},$\alpha'_k,\alpha''_k, \notin L(V(G'_1)).$ Thus there exists $\alpha'_k,\alpha''_k, \notin [5]$ such that $\left\{\alpha'_k,\alpha''_k\right\}=L(V_3)$, and $\alpha'_k \neq \alpha''_k$ since $d(v'_3,v''_3)=2$ for $v'_3,v''_3 \in V_3.$ Thus, $\left|L(\cup^3_{i=0}V_i)\right|$=$\left|L(V(G'))\right|=6+2$. Therefore, $\lambda_1^1(C_4 \times C_4)=\lambda_1^1(G')=7.$
\end{proof}

Next we present the necessary and sufficient condition under which $\lambda_1^1(C_m \times C_4)$ is 5.

\begin{thm} \label{cc4} For $ m \geq 4, m$ even, $\lambda_1^1(C_m \times C_4)=5$ if and only if $m \equiv 0 \; mod\;6 $.
\end{thm}
\begin{proof} Let $m=6n,\ n \in \mathbb N$. By Lemma \ref{cc1}, $\lambda_1^1(G) \geq 5$. Therefore, $\lambda_1^1(G')\geq 5$, where $G' \subset C_m \times C_4$. Let $G''$ be the connected component of $C_6 \times C_4$. By Corollary \ref{cc2}, $L(V_0) \cap L(V_1)= \emptyset$, $L(V_1) \cap L(V_2)= \emptyset$, $L(V_2) \cap L(V_0)= \emptyset$. Now, set $L(V_0)=L(V_3)$, $L(V_1)=L(V_4)$,$L(V_2)=L(V_5)$. But $L(V_5) \cap L(V_0)= \emptyset$, $L(V_5) \cap L(V_1)= \emptyset$. Thus, $\lambda_1^1(G'') \leq 5$ and $\lambda_1^1(C_6 \times C_4) = 5$. Thus by re-occurrence along $C_n$ and $C_m$, $m=0\;\mod\;6$ implies $\lambda^1_1(G)=5$.

 Conversely, suppose $\lambda^1_1(G)=5.$ Let $G'$ be a connected component of $G=C_m \times C_4, \; m\geq 4, m$ even. Then, $\lambda_1^1(G')=5.$ Now, assume that $m \not\equiv 0 \mod 6 $ , then $m=6n'+2$ or $m=6n'+4$ where $n' \in  \mathbb{N} \cup 0$. For $n'=0$, $G=C_4 \times C_4$, for which $\lambda_1^1(G')=7$ by Theorem \ref{cc3}. \\
Case i: For $m=6n'+2, \; n'\in \mathbb{N},$ let $V_0,V_1,V_2$ be subsets of $V(G')$. By Corollary \ref{cc2}, $L(V_0) \cap L(V_1)= \emptyset$, $L(V_1) \cap L(V_2)= \emptyset$ and $L(V_2) \cap L(V_0)= \emptyset$. Now let $G'_1$ be the subgraph of $G'$ induced by $V_0,V_1,V_2$. Since $L(V_3) \cap L(V_2)= \emptyset$ for $V_3 \subset V(G')$ and $\lambda_1^1(G')=5,$ then $L(V_3)=L(V_0).$ Let $V_4 \subset V(G').$ Then $L(V_4) \cap L(V_3)=\emptyset$ and $L(V_4) \cap L(V_2)=\emptyset$. Thus $L(V_4)=L(V_1)$. Let $V_5 \subset V(G').$ Then $L(V_5) \cap L(V_4)=\emptyset$ and $L(V_5) \cap L(V_3)=\emptyset$ and therefore $L(V_4)=L(V_1)$. The scheme continues in such a way that $L(V_i)=L(V_i+3)$ for all $i \in [m-1],$ that is $L(V_0)=L(V_3)=L(V_6)= \cdots =L(V_{6n'})$, $L(V_1)=L(V_4)=L(V_{7})= \cdots =L(V_{6n'+1})$, $L(V_2)=L(V_5)=L(V_8)= \cdots =L(V_0)$. Now, for all $v_a \in V_0$ and $v_b \in V_{6n'},$ $d(v_a,v_b)=2$. For all $v_c \in V_1,$ $v_d \in V_{6n'+1}$, $d(v_c,v_d)=2$ and finally, $L(V_0) \cap L(V_2)=\emptyset.$ Thus a contradiction.\\
Case ii: For $m=6n'+4$, $n' \in \mathbb N$, similar argument as in $m=6n+2$ applies. Thus, $\lambda_1^1(C_m \times C_4)=5$ if and only $m\equiv 0 \;\mod\; 6$
\end{proof}
\begin{cor} \label{cc5} \label{h} Let $m \equiv 0\; mod \;6 $ and $n \equiv 0\; mod \;4$. Then, $\lambda_1^1(C_m \times C_n)=5$
\end{cor}
\begin{proof}The claim follows from Theorem \ref{cc4} and the re-occurrence of the optimal labeling of $C_{6n'} \times C_4$, $n'\in \mathbb N$.
\end{proof}
\begin{cor} \label{cc6}For all $m\not\equiv 0 \mod 6$, $\lambda_1^1(C_m \times C_4) \geq 6.$
\end{cor}
\begin{thm} \label{cc7} For $C_8 \times C_4,$ $\lambda_1^1(C_8 \times C_4)=7$
\end{thm}
\begin{proof} Suppose $G'$ is a connected component of $C_8 \times C_4$ and suppose that $\lambda_1^1(G')=6.$ By Corollary \ref{cc2}, $\left|L\left\{V_0,V_1,V_2\right\}\right|=6$. Likewise, for all $\alpha_k \in L(V_0),$ $\alpha_k \notin L(V_7)$ and $\alpha_k \notin L(V_6)$. Also, for all $\alpha_j \in L(V_1)$, $\alpha_j \notin L(V_7)$. Suppose $L(V_2)=L(V_7)$, then by Corollary \ref{cc2}, if $\alpha_a, \alpha_b \in L(V_2),$ then $\alpha_a,\alpha_b \notin L \left\{V_3,V_4,V_5,V_6\right\}$. Since $\lambda_1^1(G')=6$, then there exists only five members of [6] that labels $V_3,V_4,V_5,V_6$. However, this contradicts Lemma \ref{cc1}. Thus, $L(V_2) \neq L(V_7)$. Now suppose one of $\alpha_a,\alpha_b \in L(V_2),$ say $\alpha_a$, labels some vertex $v_1 \in L(V_7)$, then there exists some $\alpha'_a \in [6]$ such that $\alpha'_a \notin L\left\{V_0,V_1,V_2\right\}$ such that $\alpha'_a = l(v_2) \in V_7$, with $v_1 \neq V_2.$ Now let $\alpha_c, \alpha_d \in L(V_0).$ Suppose $L(V_3)=L(V_0)$. Then by Corollary \ref{cc2}, $\alpha_a,\alpha_b \notin L(V_4,V_5,V_6)$, which contradicts Lemma \ref{cc1} since $\left|L(V_4,V_5,V_6)\right|=6$ and $[6]\backslash 2=5.$ Then, $\alpha'_a \in L(V_3)$ and also also one of $\alpha_a, \alpha'_b \in L(V_3)$. Further, by Corollary \ref{cc2}, $\alpha_a, \alpha'_b \notin L(V_4,V_5,V_6)$. Thus, $\lambda^1_1(G')\geq 7.$ Conversely, $\lambda^1_1(G')\leq 7$ follows directly from re-occurrence of the labeling of $C_4 \times C_4$. Thus, $\lambda^1_1(G')=\lambda^1_1(C_8 \times C_4)\geq 7$.
\end{proof}

The next result focuses on the $\lambda_1^1-$number of $C_m \times C_4$, for $m \geq 9$. Theorem \ref{cc6} have already established the lower bound for $\lambda_1^1-$number of $C_m \times C_4$ to be $6$ if $m$ is not a multiple of $6$. So we only need to label $C_{10} \times C_4$ with [6] such that it combines perfectly with the labeling of $C_6 \times C_4$ with [5] to establish general bound for all cases except when $m=14$ which is dealt with separately.

{\tiny{
\begin{center}
\pgfdeclarelayer{nodelayer}
\pgfdeclarelayer{edgelayer}
\pgfsetlayers{nodelayer,edgelayer}
\begin{tikzpicture}
%\centering
	\begin{pgfonlayer}{nodelayer}
	
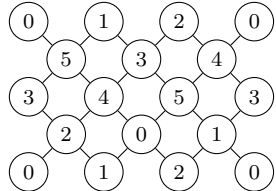
	\node [minimum size=0cm,]  at (-6.5,3.5) {$Fig. \; 3$: $5-L(1,1)$- labeling of $C_4 \times C_6$};

		\node [minimum size=0cm,draw,circle] (0) at (-8,4) {$0$};
		\node [minimum size=0cm,draw,circle] (1) at (-7,4) {$1$};
		\node [minimum size=0cm,draw,circle] (2) at (-6,4) {$2$};
		\node [minimum size=0cm,draw,circle] (3) at (-5,4) {$0$};
		\node [minimum size=0cm,draw,circle] (4) at (-7.5,4.5) {$2$};
		\node [minimum size=0cm,draw,circle] (5) at (-6.5,4.5) {$0$};
		\node [minimum size=0cm,draw,circle] (6) at (-5.5,4.5) {$1$};
		\node [minimum size=0cm,draw,circle] (7) at (-8,5) {$3$};
		\node [minimum size=0cm,draw,circle] (8) at (-7,5) {$4$};1
		\node [minimum size=0cm,draw,circle] (9) at (-6,5) {$5$};
		\node [minimum size=0cm,draw,circle] (10) at (-5,5) {$3$};
		\node [minimum size=0cm,draw,circle] (11) at (-7.5,5.5) {$5$};
		\node [minimum size=0cm,draw,circle] (12) at (-6.5,5.5) {$3$};
		\node [minimum size=0cm,draw,circle] (13) at (-5.5,5.5) {$4$};
		\node [minimum size=0cm,draw,circle] (14) at (-8,6) {$0$};
		\node [minimum size=0cm,draw,circle] (15) at (-7,6) {$1$};
		\node [minimum size=0cm,draw,circle] (16) at (-6,6) {$2$};
		\node [minimum size=0cm,draw,circle] (17) at (-5,6) {$0$};

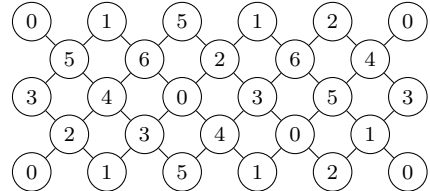
		\node [minimum size=0cm,]  at (-0.5,3.5) {$Fig.\;4$: $6-L(1,1)$- labeling of $C_4 \times C_{10}$};
		\node [minimum size=0cm,draw,circle] (18) at (-3,4) {$0$};
		\node [minimum size=0cm,draw,circle] (19) at (-2,4) {$1$};
		\node [minimum size=0cm,draw,circle] (20) at (-1,4) {$5$};
		\node [minimum size=0cm,draw,circle] (21) at (0,4) {$1$};
		\node [minimum size=0cm,draw,circle] (22) at (1,4) {$2$};
		\node [minimum size=0cm,draw,circle] (23) at (2,4) {$0$};
		\node [minimum size=0cm,draw,circle] (24) at (-2.5,4.5) {$2$};
		\node [minimum size=0cm,draw,circle] (25) at (-1.5,4.5) {$3$};
		\node [minimum size=0cm,draw,circle] (26) at (-0.5,4.5) {$4$};
		\node [minimum size=0cm,draw,circle] (27) at (0.5,4.5) {$0$};
		\node [minimum size=0cm,draw,circle] (28) at (1.5,4.5) {$1$};
		\node [minimum size=0cm,draw,circle] (29) at (-3,5) {$3$};
		\node [minimum size=0cm,draw,circle] (30) at (-2,5) {$4$};
		\node [minimum size=0cm,draw,circle] (31) at (-1,5) {$0$};
		\node [minimum size=0cm,draw,circle] (32) at (0,5) {$3$};
		\node [minimum size=0cm,draw,circle] (33) at (1,5) {$5$};
		\node [minimum size=0cm,draw,circle] (34) at (2,5) {$3$};
		\node [minimum size=0cm,draw,circle] (35) at (-2.5,5.5) {$5$};
		\node [minimum size=0cm,draw,circle] (36) at (-1.5,5.5) {$6$};
		\node [minimum size=0cm,draw,circle] (37) at (-0.5,5.5) {$2$};
		\node [minimum size=0cm,draw,circle] (38) at (0.5,5.5) {$6$};
		\node [minimum size=0cm,draw,circle] (39) at (1.5,5.5) {$4$};
		\node [minimum size=0cm,draw,circle] (40) at (-3,6) {$0$};
		\node [minimum size=0cm,draw,circle] (41) at (-2,6) {$1$};
		\node [minimum size=0cm,draw,circle] (42) at (-1,6) {$5$};
		\node [minimum size=0cm,draw,circle] (43) at (0,6) {$1$};
		\node [minimum size=0cm,draw,circle] (44) at (1,6) {$2$};
		\node [minimum size=0cm,draw,circle] (45) at (2,6) {$0$};
		
	\end{pgfonlayer}
	\begin{pgfonlayer}{edgelayer}
		\draw [thin=1.00] (0) to (4);
		\draw [thin=1.00] (1) to (4);
		\draw [thin=1.00] (1) to (5);
		\draw [thin=1.00] (2) to (5);
		\draw [thin=1.00] (2) to (6);
		\draw [thin=1.00] (3) to (6);
		\draw [thin=1.00] (4) to (7);
		\draw [thin=1.00] (4) to (8);
		\draw [thin=1.00] (5) to (8);
		\draw [thin=1.00] (5) to (9);
		\draw [thin=1.00] (6) to (9);
		\draw [thin=1.00] (6) to (10);
		\draw [thin=1.00] (7) to (11);
		\draw [thin=1.00] (8) to (11);
		\draw [thin=1.00] (8) to (12);
		\draw [thin=1.00] (9) to (12);
		\draw [thin=1.00] (9) to (13);
		\draw [thin=1.00] (10) to (13);
		\draw [thin=1.00] (11) to (14);
		\draw [thin=1.00] (11) to (15);
		\draw [thin=1.00] (12) to (15);
		\draw [thin=1.00] (12) to (16);
		\draw [thin=1.00] (13) to (16);
		\draw [thin=1.00] (13) to (17);

    \draw [thin=1.00] (18) to (24);
		\draw [thin=1.00] (19) to (24);
		\draw [thin=1.00] (19) to (25);
		\draw [thin=1.00] (20) to (25);
		\draw [thin=1.00] (20) to (26);
		\draw [thin=1.00] (21) to (26);
		\draw [thin=1.00] (21) to (27);
		\draw [thin=1.00] (22) to (27);
		\draw [thin=1.00] (22) to (28);
		\draw [thin=1.00] (23) to (28);
		
		\draw [thin=1.00] (24) to (29);
		\draw [thin=1.00] (24) to (30);
		\draw [thin=1.00] (25) to (30);
		\draw [thin=1.00] (25) to (31);
		\draw [thin=1.00] (26) to (31);
		\draw [thin=1.00] (26) to (32);
		\draw [thin=1.00] (27) to (32);
		\draw [thin=1.00] (27) to (33);
		\draw [thin=1.00] (28) to (33);
		\draw [thin=1.00] (28) to (34);

  	\draw [thin=1.00] (29) to (35);
		\draw [thin=1.00] (30) to (35);
		\draw [thin=1.00] (30) to (36);
		\draw [thin=1.00] (31) to (36);
	  \draw [thin=1.00] (31) to (37);
		\draw [thin=1.00] (32) to (37);		
		\draw [thin=1.00] (32) to (38);
		\draw [thin=1.00] (33) to (38);
		\draw [thin=1.00] (33) to (39);
		\draw [thin=1.00] (34) to (39);
		
		\draw [thin=1.00] (35) to (40);
		\draw [thin=1.00] (35) to (41);
		\draw [thin=1.00] (36) to (41);
		\draw [thin=1.00] (36) to (42);
		\draw [thin=1.00] (37) to (42);
		\draw [thin=1.00] (37) to (43);
		\draw [thin=1.00] (38) to (43);
		\draw [thin=1.00] (38) to (44);		
		\draw [thin=1.00] (39) to (44);
		\draw [thin=1.00] (39) to (45);

	\end{pgfonlayer}
\end{tikzpicture}
\end{center}
}}
\begin{thm} \label{cc8} Let $m',m'' \in \mathbb N \cup 0$, with $10m'+6n''$ not a multiple of 6. Then $\lambda_1^1(C_{10m'+6m''} \times C_4)=6.$
\end{thm}
\begin{proof} By Corollary \ref{cc6}, $\lambda_1^1(C_m \times C_n) \geq 6$ for all $m$ not multiple of 6. The claim follows required combinations of Figures $3$ and $4$ above which shows that $\lambda_1^1(C_{10m'+6m''} \times C_4)\leq 6$ for $10m'+6m''$ not a multiple of 6.
\end{proof}

Clearly, every even number $m \geq 10, \; m\neq 14$ can be obtained from $10m'+6m''$ defined above. Therefore, we can conclude that for all $m \geq 9,$ $\lambda_1^1(C_m \times C_n)=6$ for all $m$ that is not a multiple of 6 if we can establish that the $\lambda_1^1-$number of $C_{14} \times C_4$ is $6$. We show this in the next result.

\begin{thm} \label{cc9} $\lambda_1^1(C_{14} \times C_4)=6$
\end{thm}
\begin{proof} $\lambda_1^1(C_{14} \times C_4)=6$
\end{proof}
We have now completely determined the $\lambda_1^1-$numbers of $(C_m \times C_4)$ for all $m \geq 3$. In what follows, we investigate the values of $\lambda_1^1(C_m \times C_6)$.

\begin{prop} \label{cc10} Let $G'$ be a connected component of $C_m \times C_6,$ $m \in \mathbb N, m \geq 3.$ Let $V_i \subseteq V(G'), \;i \in [m-1].$ Then,
	
$(i)$\;$\lambda_1^1(C_m \times C_6) \geq 5.$

$(ii)$ \;Given $v_a, v_b \in V_i, $ $d(v_a,v_b) \leq 2.$

$(iii)$\;For all $V_i \subseteq V(G'), \; \left|L(V_i)\right|=3$

$(iv)$\;suppose $\alpha_k \in L(V_i),$ then $\alpha_k \notin L(V_{i+2}).$
%\end{itemize}
\end{prop}
\begin{proof}The proof of the claims above are as follows:\\
 $(i)$ $C_m \times C_6$ contains $P_m \times C_6$. Now from Corollary \ref{e}, $\lambda_1^1(P_m \times C_6)=5.$ Therefore $\lambda_1^1(C_m \times C_6) \geq 5$.\\
$(ii)$ Let $V_i \subseteq V(G').$ $V_i=\left\{u_iv_j,u_iv_{j+2},u_iv_{j+4}\right\}$, where $j \in \left\{0,1\right\}$. Now, since $C_m$ is a cycle,then $d(u_iv_j, u_iv_{j+4})=2$. Clearly, $d(u_iv_j, u_iv_{j+2})=2$, $d(u_iv_{j+2}, u_iv_{j+4})=2$ and thus the claim.\\
	  $(iii)$ This is quite obvious.\\
	 $(iv)$ It is obvious that for all $u_iv_j \in V_i$ and $u_{i+2}v_k \in V_{i+2}$, $d(u_iv_j,u_{i+2}v_k)=2$. Therefore, $L(V_i) \cap L(V_{i+2})= \emptyset.$

\end{proof}

%\begin{prop} \label{cc11}$\lambda^1_1(C_4 \times C_6)=5$
%\end{prop}
%\begin{proof} This follows from Corollary \ref{cc6} and the commutative property of the direct %product of graphs. \end{proof}

The next result describes a property of $L(1,1)$-labeling of $C_m \times C_6$
\begin{lem} \label{cc12} Let $\alpha_k \in L(V_i)$, $i \in [m-1]$, $V_i \subseteq V(G')$, then $\alpha_k$ labels some vertex $v_{i+1} \in V_{i+1}$. In other words, $L(V_i)$ labels $V_{i+1}$.
\end{lem}
\begin{proof} Let $\left\{u_iv_j,u_iv_{j+2},u_iv_{j+4}\right\}=V_i$ and $\left\{u_{i+1}v_{j+1},u_{i+1}v_{j+3},u_{i+1}v_{j+5}\right\}=V_{i+1}$ Clearly $d(u_iv_j,u_{i+1}v_{j+3})=d(u_iv_{j+2},u_{i+1}v_{j+5})=d(u_iv_{j+4},u_{i+1}v_{j+1})=3$. Therefore, suppose $\alpha_k = l(v_i)$, for some $v_i \in V_i,$ then, there exists some unique $v_{i+k} \in V_{i+1}$ such that $l(v_i)=l(v_{i+k})$,with $\left|(i-(i+k))\right|=3$. (The uniqueness of $v_{i+k}$ results from Proposition \ref{cc10}(b).)
\end{proof}

\begin{cor}\label{cc13} If $L(V_i)=L(V_{i+1})$, then, $L(V_{i+2}) \cap L(V_i) = \emptyset$ and $L(V_i) \cap L(V_{i+3})= \emptyset$.
\end{cor} It is obvious from Proposition \ref{cc10}(d).

\begin{cor}\label{cc14} $\lambda_1^1(C_m \times C_6)=5$ if and only if $m \equiv 0\; mod \;4$
\end{cor}
\begin{proof} Let $m \equiv 0\mod4.$ For $m=4,$ clearly $\lambda_1^1(C_4 \times C_6)=5$, which is obtained from Corollary \ref{cc6}. Now in the case of the general $m \equiv 0\mod 4$, by re-occurrence of the labeling of $C_4 \times C_6$ along $C_m$, it follows that $\lambda_1^1(C_m \times C_6)=5$. Conversely, suppose that $\lambda_1^1(C_m \times C_6)=5$. We show that $m \equiv 0 \mod 4$. Let $L(V_i)=\left\{\alpha_i,\alpha_j,\alpha_k\right\}$. By Proposition \ref{cc10} $(c)$,$\alpha_i \neq \alpha_j \neq \alpha_k \neq \alpha_i,$ that is, $\left|L(V_i)\right|=3$. Suppose $L(V_0)=L(V_1)$ by Lemma \ref{cc12}, then by Corollary \ref{cc13}, $L(V_2)\cap L(V_1) = \emptyset$ and $L(V_3)\cap L(V_1) = \emptyset$. Since $\lambda_1^1(C_m \times C_6)=5,$ Then $L(V_2)=L(V_3)=[5] \backslash L(V_0)$. This scheme continues such that $L(V_0)=L(V_4)=L(V_5)$; $L(V_2)=L(V_6)=L(V_7) \cdots= L(V_0)=L(V_{m-4})=L(V_{m-3})$; and $L(V_0)=L(V_4)=L(V_8)=\cdots =L(V_{4(n)})$, $n \in \mathbb N$, where $4n=(m-1)+1=m$ since $C_m$ is a cycle. Thus $m \equiv 0\mod4.$
\end{proof}

The implication of the last result is that the lower bound for the $\lambda_1^1-$number of graph product $C_m \times C_6$, $m \geq 3$ is 6 except for when $m \equiv 0 \mod 4$, in which case the optimal $ \lambda_1^1-$number reduces by $1$.

Now we consider particular cases where the lower bound is strictly greater than 6.

\begin{thm} \label{cc15}$\lambda_1^1(C_6 \times C_6)=8$
\end{thm}
\begin{proof} Suppose that $\lambda_1^1(C_6 \times C_6)=7$. Let $\left\{V_i\right\} \subseteq V(G'), $ for all $i \in [5]$. By proposition \ref{cc10} (d), $L(V_0) \cap L(V_2)= \emptyset$; $L(V_0) \cap L(V_4)= \emptyset$ and $L(V_4) \cap L(V_2)= \emptyset$. Now, $L(V_2) \subseteq[7] \backslash L(V_0)$ and $L(V_4) \subseteq [7] \backslash L(V_0)$. Note that $\left|[7] \backslash L(V_0)\right|=5$. Now set $[7] \backslash L(V_0)=[A']$. $L(V_4) \subseteq [A'] \backslash L(V_2)$ since $L(V_4) \cap L(V_2)= \emptyset$. Now, $\left|[A'] \backslash L(V_2)\right|=2$. However, by Proposition \ref{cc10} (c), $\left|L(V_4)\right|=3$. Therefore a contradiction and hence $\lambda(C_6 \times C_6) \geq 8.$  The labeling in Figure $6$ confirms that $\lambda_1^1(C_6 \times C_6)\leq 8$, and thus, $\lambda_1^1(C_6 \times C_6)=8$.

\begin{center}
\pgfdeclarelayer{nodelayer}
\pgfdeclarelayer{edgelayer}
\pgfsetlayers{nodelayer,edgelayer}
\begin{tikzpicture}
%\centering
%\pgfdeclarelayer{nodelayer}
%\pgfdeclarelayer{edgelayer}
%\pgfsetlayers{nodelayer,edgelayer}

	\begin{pgfonlayer}{nodelayer}
{\tiny{	
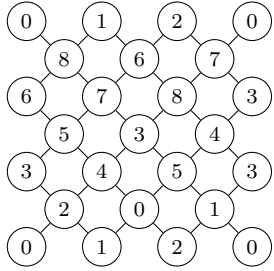
	\node [minimum size=0cm,]  at (-6.5,4) {$Fig. \; 6$: $8-L(1,1)$- labeling of $C_6 \times C_6$};

		\node [minimum size=0cm,draw,circle] (0) at (-8,5) {$ 0$};
		\node [minimum size=0cm,draw,circle] (1) at (-7,5) {$ 1$};
		\node [minimum size=0cm,draw,circle] (2) at (-6,5) {$ 2$};
		\node [minimum size=0cm,draw,circle] (3) at (-5,5) {$ 0$};
		\node [minimum size=0cm,draw,circle] (4) at (-7.5,5.5) {$ 2$};
		\node [minimum size=0cm,draw,circle] (5) at (-6.5,5.5) {$ 0$};
		\node [minimum size=0cm,draw,circle] (6) at (-5.5,5.5) {$ 1$};
		\node [minimum size=0cm,draw,circle] (7) at (-8,6) {$ 3$};
		\node [minimum size=0cm,draw,circle] (8) at (-7,6) {$ 4$};
		\node [minimum size=0cm,draw,circle] (9) at (-6,6) {$ 5$};
		\node [minimum size=0cm,draw,circle] (10) at (-5,6) {$ 3$};
		\node [minimum size=0cm,draw,circle] (11) at (-7.5,6.5) {$ 5$};
		\node [minimum size=0cm,draw,circle] (12) at (-6.5,6.5) {$ 3$};
		\node [minimum size=0cm,draw,circle] (13) at (-5.5,6.5) {$ 4$};
		\node [minimum size=0cm,draw,circle] (14) at (-8,7) {$ 6$};
		\node [minimum size=0cm,draw,circle] (15) at (-7,7) {$ 7$};
		\node [minimum size=0cm,draw,circle] (16) at (-6,7) {$ 8$};
		\node [minimum size=0cm,draw,circle] (17) at (-5,7) {$ 3$};
		\node [minimum size=0cm,draw,circle] (18) at (-7.5,7.5) {$ 8$};
		\node [minimum size=0cm,draw,circle] (19) at (-6.5,7.5) {$ 6$};
		\node [minimum size=0cm,draw,circle] (20) at (-5.5,7.5) {$ 7$};
		\node [minimum size=0cm,draw,circle] (21) at (-8,8) {$ 0$};
		\node [minimum size=0cm,draw,circle] (22) at (-7,8) {$ 1$};
		\node [minimum size=0cm,draw,circle] (23) at (-6,8) {$ 2$};
		\node [minimum size=0cm,draw,circle] (24) at (-5,8) {$ 0$};
		}}\end{pgfonlayer}
	\begin{pgfonlayer}{edgelayer}
		\draw [thin=1.00] (0) to (4);
		\draw [thin=1.00] (1) to (4);
		\draw [thin=1.00] (1) to (5);
		\draw [thin=1.00] (2) to (5);
		\draw [thin=1.00] (2) to (6);
		\draw [thin=1.00] (3) to (6);
		\draw [thin=1.00] (4) to (7);
		\draw [thin=1.00] (4) to (8);
		\draw [thin=1.00] (5) to (8);
		\draw [thin=1.00] (5) to (9);
		\draw [thin=1.00] (6) to (9);
		\draw [thin=1.00] (6) to (10);
		\draw [thin=1.00] (7) to (11);
		\draw [thin=1.00] (8) to (11);
		\draw [thin=1.00] (8) to (12);
		\draw [thin=1.00] (9) to (12);
		\draw [thin=1.00] (9) to (13);
		\draw [thin=1.00] (10) to (13);
		\draw [thin=1.00] (11) to (14);
		\draw [thin=1.00] (11) to (15);
		\draw [thin=1.00] (12) to (15);
		\draw [thin=1.00] (12) to (16);
		\draw [thin=1.00] (13) to (16);
		\draw [thin=1.00] (13) to (17);
		\draw [thin=1.00] (14) to (18);
		\draw [thin=1.00] (15) to (18);
		\draw [thin=1.00] (15) to (19);
		\draw [thin=1.00] (16) to (19);
		\draw [thin=1.00] (16) to (20);
		\draw [thin=1.00] (17) to (20);
		\draw [thin=1.00] (18) to (21);
		\draw [thin=1.00] (18) to (22);
		\draw [thin=1.00] (19) to (22);
		\draw [thin=1.00] (19) to (23);
		\draw [thin=1.00] (20) to (23);
		\draw [thin=1.00] (20) to (24);		
	\end{pgfonlayer}
\end{tikzpicture}
\end{center}

\end{proof}
\begin{thm} \label{cc16} $\lambda_1^1(C_{10} \times C_6)=7$
\end{thm}
\begin{proof}Let $G'$ be a connected component of $C_{10} \times C_6.$ and suppose that $\lambda_1^1(G')=6$. Let $V_0 \subset V(G')$ such that $L(V_0) \subset [6].$ By Proposition \ref{cc10} (c), $L(V_0) \cap L(V_2)= \emptyset$. Therefore $L(V_0) \subset [6] \backslash L(V_0),$ where $\left|L(V_0)\right|=3$. Thus, $\left|[6]\backslash L(V_0)\right|=4$. Now, for all $v_0 \in V_0$ and $v_8 \in V_8,$ $V_0,V_8 \subset V(G')$, $d(v_0,v_8)=2,$ since $C_{10}$ is a cycle of length $10$. Therefore, $L(V_8) \subset [6] \backslash L(V_0).$ Now, suppose that, $L(V_8)=L(V_2),$ by Proposition \ref{cc10}, then there exists $\alpha_k \in [6]$ such that $\alpha_k \notin L(V_0)$, and$\alpha_k \in L(V_2).$ Thus $L(V_4) \subset L(V_0) \cup \alpha_k$ and $L(V_6) \subset L(V_0) \cup \alpha_k$. Now, $\left|L(V_0) \cup \alpha_k \right|=4$. By Proposition \ref{cc10}, (d), $L(V_6) \cap L(V_4)= \emptyset$. Thus $\left|L(V_6) \cup L(V_4)\right|=6,$ which is a contradiction.

Now, suppose $L(V_8) \neq L(V_0)$ then it is not difficult to see that there exists $\alpha_a, \alpha_b \in [6] \backslash L(V_0)$ such that $L(V_8) \cap L(V_2)=\left\{\alpha_a, \alpha_b\right\}$. Thus $L(V_8)=\left\{\alpha_a,\alpha_b,\alpha_c\right\}$ and $L(V_2)=\left\{\alpha_a,\alpha_b,\alpha_d\right\}$ such that $L(V_2) \cup L(V_8)=[6]\backslash L[V_2]$.
Now by Proposition \ref{cc10} (d) still, $L(V_4) \subseteq [6]\backslash L(V_2)=L(V_0) \cup \alpha_k,$ such that $\alpha_k \notin L(V_0)$, $\alpha_k \in L(V_2).$ $L(V_0) \subseteq [6]\backslash L(V_8)=L(V_0) \cup \alpha_j$ for $\alpha_j \notin L(V_2),$ $\alpha_k \notin L(V_0), \; \alpha_j \neq \alpha_k$. Thus, $\left|L(V_0) \cup \left\{\alpha_j\cup \alpha_k\right\}\right|=5$. By $\left|L(V_6 \cup L(V_4))\right|=6$, and for all $v_6 \in V_6,$ and $v_4 \in V_4$, $d(v_4,v_6)=2.$ Thus a contradiction and hence $\lambda(G') \geq 7.$ \\Conversely, we consider the $7-L(1,1)$-labeling of $C_{10} \times C_6$ in Figure $5$ below.

\begin{center}
\pgfdeclarelayer{nodelayer}
\pgfdeclarelayer{edgelayer}
\pgfsetlayers{nodelayer,edgelayer}
\begin{tikzpicture}
	\begin{pgfonlayer}{nodelayer}
{\tiny{	
	\node [minimum size=0cm,]  at (-5.5,-9) {$Fig. \; 7 $: $7-L(1,1)$- labeling of $C_{10} \times C_6$};

		\node [minimum size=0cm,draw,circle] (0) at (-8,-8) {$ 0$};
		\node [minimum size=0cm,draw,circle] (1) at (-8,-7) {$ 1$};
		\node [minimum size=0cm,draw,circle] (2) at (-8,-6) {$ 2$};
		\node [minimum size=0cm,draw,circle] (3) at (-8,-5) {$ 0$};
		\node [minimum size=0cm,draw,circle] (4) at (-7.5,-7.5) {$ 5$};
		\node [minimum size=0cm,draw,circle] (5) at (-7.5,-6.5) {$ 4$};
		\node [minimum size=0cm,draw,circle] (6) at (-7.5,-5.5) {$ 1$};
		\node [minimum size=0cm,draw,circle] (7) at (-7,-8) {$ 6$};
		\node [minimum size=0cm,draw,circle] (8) at (-7,-7) {$ 7$};1
		\node [minimum size=0cm,draw,circle] (9) at (-7,-6) {$ 3$};
		\node [minimum size=0cm,draw,circle] (10) at (-7,-5) {$ 6$};
		\node [minimum size=0cm,draw,circle] (11) at (-6.5,-7.5) {$ 3$};
		\node [minimum size=0cm,draw,circle] (12) at (-6.5,-6.5) {$ 0$};
		\node [minimum size=0cm,draw,circle] (13) at (-6.5,-5.5) {$ 2$};
		\node [minimum size=0cm,draw,circle] (14) at (-6,-8) {$ 1$};
		\node [minimum size=0cm,draw,circle] (15) at (-6,-7) {$ 4$};
		\node [minimum size=0cm,draw,circle] (16) at (-6,-6) {$ 5$};
		\node [minimum size=0cm,draw,circle] (17) at (-6,-5) {$ 1$};
		\node [minimum size=0cm,draw,circle] (18) at (-5.5,-7.5) {$ 6$};
		\node [minimum size=0cm,draw,circle] (19) at (-5.5,-6.5) {$ 1$};
		\node [minimum size=0cm,draw,circle] (20) at (-5.5,-5.5) {$ 7$};
		\node [minimum size=0cm,draw,circle] (21) at (-5,-8) {$ 2$};
		\node [minimum size=0cm,draw,circle] (22) at (-5,-7) {$ 7$};
		\node [minimum size=0cm,draw,circle] (23) at (-5,-6) {$ 0$};
		\node [minimum size=0cm,draw,circle] (24) at (-5,-5) {$ 2$};
		\node [minimum size=0cm,draw,circle] (25) at (-4.5,-7.5) {$ 3$};
		\node [minimum size=0cm,draw,circle] (26) at (-4.5,-6.5) {$ 4$};
		\node [minimum size=0cm,draw,circle] (27) at (-4.5,-5.5) {$ 5$};
		\node [minimum size=0cm,draw,circle] (28) at (-4,-8) {$ 6$};
		\node [minimum size=0cm,draw,circle] (29) at (-4,-7) {$ 5$};
		\node [minimum size=0cm,draw,circle] (30) at (-4,-6) {$ 3$};
		\node [minimum size=0cm,draw,circle] (31) at (-4,-5) {$ 6$};
		\node [minimum size=0cm,draw,circle] (32) at (-3.5,-7.5) {$ 2$};
		\node [minimum size=0cm,draw,circle] (33) at (-3.5,-6.5) {$ 0$};
		\node [minimum size=0cm,draw,circle] (34) at (-3.5,-5.5) {$ 7$};
		\node [minimum size=0cm,draw,circle] (35) at (-3,-8) {$ 0$};
		\node [minimum size=0cm,draw,circle] (36) at (-3,-7) {$ 1$};
		\node [minimum size=0cm,draw,circle] (37) at (-3,-6) {$ 2$};
		\node [minimum size=0cm,draw,circle] (38) at (-3,-5) {$ 0$};
	}}\end{pgfonlayer}
	\begin{pgfonlayer}{edgelayer}
		\draw [thin=1.00] (0) to (4);
		\draw [thin=1.00] (1) to (4);
		\draw [thin=1.00] (1) to (5);
		\draw [thin=1.00] (2) to (5);
		\draw [thin=1.00] (2) to (6);
		\draw [thin=1.00] (3) to (6);
		\draw [thin=1.00] (4) to (7);
		\draw [thin=1.00] (4) to (8);
		\draw [thin=1.00] (5) to (8);
		\draw [thin=1.00] (5) to (9);
		\draw [thin=1.00] (6) to (9);
		\draw [thin=1.00] (6) to (10);
		\draw [thin=1.00] (7) to (11);
		\draw [thin=1.00] (8) to (11);
		\draw [thin=1.00] (8) to (12);
		\draw [thin=1.00] (9) to (12);
		\draw [thin=1.00] (9) to (13);
		\draw [thin=1.00] (10) to (13);
		\draw [thin=1.00] (11) to (14);
		\draw [thin=1.00] (11) to (15);
		\draw [thin=1.00] (12) to (15);
		\draw [thin=1.00] (12) to (16);
		\draw [thin=1.00] (13) to (16);
		\draw [thin=1.00] (13) to (17);
		\draw [thin=1.00] (14) to (18);
		\draw [thin=1.00] (15) to (18);
		\draw [thin=1.00] (15) to (19);
		\draw [thin=1.00] (16) to (19);
		\draw [thin=1.00] (16) to (20);
		\draw [thin=1.00] (17) to (20);
		\draw [thin=1.00] (18) to (21);
		\draw [thin=1.00] (18) to (22);
		\draw [thin=1.00] (19) to (22);
		\draw [thin=1.00] (19) to (23);
		\draw [thin=1.00] (20) to (23);
		\draw [thin=1.00] (20) to (24);
		\draw [thin=1.00] (21) to (25);
		\draw [thin=1.00] (22) to (25);
		\draw [thin=1.00] (22) to (26);
		\draw [thin=1.00] (23) to (26);
		\draw [thin=1.00] (23) to (27);
		\draw [thin=1.00] (24) to (27);
		\draw [thin=1.00] (25) to (28);
		\draw [thin=1.00] (25) to (29);
		\draw [thin=1.00] (26) to (29);
		\draw [thin=1.00] (26) to (30);
		\draw [thin=1.00] (27) to (30);
		\draw [thin=1.00] (27) to (31);
    \draw [thin=1.00] (28) to (32);
		\draw [thin=1.00] (29) to (32);
		\draw [thin=1.00] (29) to (33);
		\draw [thin=1.00] (30) to (33);
		\draw [thin=1.00] (30) to (34);
		\draw [thin=1.00] (31) to (34);
		\draw [thin=1.00] (32) to (35);
		\draw [thin=1.00] (32) to (36);
		\draw [thin=1.00] (33) to (36);
		\draw [thin=1.00] (33) to (37);
		\draw [thin=1.00] (34) to (37);
		\draw [thin=1.00] (34) to (38);
\end{pgfonlayer}
\end{tikzpicture}
\end{center}

\end{proof}

\begin{thm} \label{cc18}  $\lambda_1^1(C_{14} \times C_6)=6$
\end{thm}
\begin{proof} Since $14$ is not a multiple of $4$ and by Corollary \ref{cc14}, $\lambda_1^1(C_{14} \times C_6) \geq 6.$

\begin{center}
\pgfdeclarelayer{nodelayer}
\pgfdeclarelayer{edgelayer}
\pgfsetlayers{nodelayer,edgelayer}
\begin{tikzpicture}
%\centering
	\begin{pgfonlayer}{nodelayer}
{\tiny{	
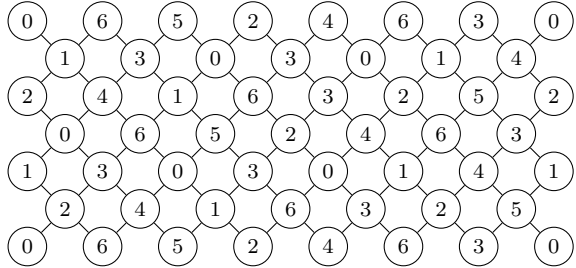
	\node [minimum size=0cm,]  at (5.5,-9) {$Fig. \; 8$: $5-L(1,1)$- labeling of $C_{14} \times C_6$};

		\node [minimum size=0cm,draw,circle] (0) at (2,-8) {$ 0$};
		\node [minimum size=0cm,draw,circle] (1) at (2,-7) {$ 1$};
		\node [minimum size=0cm,draw,circle] (2) at (2,-6) {$ 2$};
		\node [minimum size=0cm,draw,circle] (3) at (2,-5) {$ 0$};
		\node [minimum size=0cm,draw,circle] (4) at (2.5,-7.5) {$ 2$};
		\node [minimum size=0cm,draw,circle] (5) at (2.5,-6.5) {$ 0$};
		\node [minimum size=0cm,draw,circle] (6) at (2.5,-5.5) {$ 1$};
		\node [minimum size=0cm,draw,circle] (7) at (3,-8) {$ 6$};
		\node [minimum size=0cm,draw,circle] (8) at (3,-7) {$ 3$};1
		\node [minimum size=0cm,draw,circle] (9) at (3,-6) {$ 4$};
		\node [minimum size=0cm,draw,circle] (10) at (3,-5) {$ 6$};
		\node [minimum size=0cm,draw,circle] (11) at (3.5,-7.5) {$ 4$};
		\node [minimum size=0cm,draw,circle] (12) at (3.5,-6.5) {$ 6$};
		\node [minimum size=0cm,draw,circle] (13) at (3.5,-5.5) {$ 3$};
		\node [minimum size=0cm,draw,circle] (14) at (4,-8) {$ 5$};
		\node [minimum size=0cm,draw,circle] (15) at (4,-7) {$ 0$};
		\node [minimum size=0cm,draw,circle] (16) at (4,-6) {$ 1$};
		\node [minimum size=0cm,draw,circle] (17) at (4,-5) {$ 5$};
		\node [minimum size=0cm,draw,circle] (18) at (4.5,-7.5) {$ 1$};
		\node [minimum size=0cm,draw,circle] (19) at (4.5,-6.5) {$ 5$};
		\node [minimum size=0cm,draw,circle] (20) at (4.5,-5.5) {$ 0$};
		\node [minimum size=0cm,draw,circle] (21) at (5,-8) {$ 2$};
		\node [minimum size=0cm,draw,circle] (22) at (5,-7) {$ 3$};
		\node [minimum size=0cm,draw,circle] (23) at (5,-6) {$ 6$};
		\node [minimum size=0cm,draw,circle] (24) at (5,-5) {$ 2$};
		\node [minimum size=0cm,draw,circle] (25) at (5.5,-7.5) {$ 6$};
		\node [minimum size=0cm,draw,circle] (26) at (5.5,-6.5) {$ 2$};
		\node [minimum size=0cm,draw,circle] (27) at (5.5,-5.5) {$ 3$};
		\node [minimum size=0cm,draw,circle] (28) at (6,-8) {$ 4$};
		\node [minimum size=0cm,draw,circle] (29) at (6,-7) {$ 0$};
		\node [minimum size=0cm,draw,circle] (30) at (6,-6) {$ 3$};
		\node [minimum size=0cm,draw,circle] (31) at (6,-5) {$ 4$};
		\node [minimum size=0cm,draw,circle] (32) at (6.5,-7.5) {$ 3$};
		\node [minimum size=0cm,draw,circle] (33) at (6.5,-6.5) {$ 4$};
		\node [minimum size=0cm,draw,circle] (34) at (6.5,-5.5) {$ 0$};
		\node [minimum size=0cm,draw,circle] (35) at (7,-8) {$ 6$};
		\node [minimum size=0cm,draw,circle] (36) at (7,-7) {$ 1$};
		\node [minimum size=0cm,draw,circle] (37) at (7,-6) {$ 2$};
		\node [minimum size=0cm,draw,circle] (38) at (7,-5) {$ 6$};
		\node [minimum size=0cm,draw,circle] (39) at (7.5,-7.5) {$ 2$};
		\node [minimum size=0cm,draw,circle] (40) at (7.5,-6.5) {$ 6$};
		\node [minimum size=0cm,draw,circle] (41) at (7.5,-5.5) {$ 1$};
		\node [minimum size=0cm,draw,circle] (42) at (8,-8) {$ 3$};
		\node [minimum size=0cm,draw,circle] (43) at (8,-7) {$ 4$};
		\node [minimum size=0cm,draw,circle] (44) at (8,-6) {$ 5$};
		\node [minimum size=0cm,draw,circle] (45) at (8,-5) {$ 3$};
		\node [minimum size=0cm,draw,circle] (46) at (8.5,-7.5) {$ 5$};
		\node [minimum size=0cm,draw,circle] (47) at (8.5,-6.5) {$ 3$};
		\node [minimum size=0cm,draw,circle] (48) at (8.5,-5.5) {$ 4$};
		\node [minimum size=0cm,draw,circle] (49) at (9,-8) {$ 0$};
		\node [minimum size=0cm,draw,circle] (50) at (9,-7) {$ 1$};
		\node [minimum size=0cm,draw,circle] (51) at (9,-6) {$ 2$};
		\node [minimum size=0cm,draw,circle] (52) at (9,-5) {$ 0$};

	}}\end{pgfonlayer}
	\begin{pgfonlayer}{edgelayer}
		\draw [thin=1.00] (0) to (4);
		\draw [thin=1.00] (1) to (4);
		\draw [thin=1.00] (1) to (5);
		\draw [thin=1.00] (2) to (5);
		\draw [thin=1.00] (2) to (6);
		\draw [thin=1.00] (3) to (6);
		\draw [thin=1.00] (4) to (7);
		\draw [thin=1.00] (4) to (8);
		\draw [thin=1.00] (5) to (8);
		\draw [thin=1.00] (5) to (9);
		\draw [thin=1.00] (6) to (9);
		\draw [thin=1.00] (6) to (10);
		\draw [thin=1.00] (7) to (11);
		\draw [thin=1.00] (8) to (11);
		\draw [thin=1.00] (8) to (12);
		\draw [thin=1.00] (9) to (12);
		\draw [thin=1.00] (9) to (13);
		\draw [thin=1.00] (10) to (13);
		\draw [thin=1.00] (11) to (14);
		\draw [thin=1.00] (11) to (15);
		\draw [thin=1.00] (12) to (15);
		\draw [thin=1.00] (12) to (16);
		\draw [thin=1.00] (13) to (16);
		\draw [thin=1.00] (13) to (17);
		\draw [thin=1.00] (14) to (18);
		\draw [thin=1.00] (15) to (18);
		\draw [thin=1.00] (15) to (19);
		\draw [thin=1.00] (16) to (19);
		\draw [thin=1.00] (16) to (20);
		\draw [thin=1.00] (17) to (20);
		\draw [thin=1.00] (18) to (21);
		\draw [thin=1.00] (18) to (22);
		\draw [thin=1.00] (19) to (22);
		\draw [thin=1.00] (19) to (23);
		\draw [thin=1.00] (20) to (23);
		\draw [thin=1.00] (20) to (24);
		\draw [thin=1.00] (21) to (25);
		\draw [thin=1.00] (22) to (25);
		\draw [thin=1.00] (22) to (26);
		\draw [thin=1.00] (23) to (26);
		\draw [thin=1.00] (23) to (27);
		\draw [thin=1.00] (24) to (27);
		\draw [thin=1.00] (25) to (28);
		\draw [thin=1.00] (25) to (29);
		\draw [thin=1.00] (26) to (29);
		\draw [thin=1.00] (26) to (30);
		\draw [thin=1.00] (27) to (30);
		\draw [thin=1.00] (27) to (31);
    \draw [thin=1.00] (28) to (32);
		\draw [thin=1.00] (29) to (32);
		\draw [thin=1.00] (29) to (33);
		\draw [thin=1.00] (30) to (33);
		\draw [thin=1.00] (30) to (34);
		\draw [thin=1.00] (31) to (34);
		\draw [thin=1.00] (32) to (35);
		\draw [thin=1.00] (32) to (36);
		\draw [thin=1.00] (33) to (36);
		\draw [thin=1.00] (33) to (37);
		\draw [thin=1.00] (34) to (37);
		\draw [thin=1.00] (34) to (38);
		\draw [thin=1.00] (35) to (39);
		\draw [thin=1.00] (36) to (39);
		\draw [thin=1.00] (36) to (40);
		\draw [thin=1.00] (37) to (40);
		\draw [thin=1.00] (37) to (41);
		\draw [thin=1.00] (38) to (41);
		\draw [thin=1.00] (39) to (42);
		\draw [thin=1.00] (39) to (43);
		\draw [thin=1.00] (40) to (43);
		\draw [thin=1.00] (40) to (44);
		\draw [thin=1.00] (41) to (44);
		\draw [thin=1.00] (41) to (45);
		\draw [thin=1.00] (42) to (46);
		\draw [thin=1.00] (43) to (46);
		\draw [thin=1.00] (43) to (47);
		\draw [thin=1.00] (44) to (47);
		\draw [thin=1.00] (44) to (48);
		\draw [thin=1.00] (45) to (48);
		\draw [thin=1.00] (46) to (49);
		\draw [thin=1.00] (46) to (50);
		\draw [thin=1.00] (47) to (50);
		\draw [thin=1.00] (47) to (51);
		\draw [thin=1.00] (48) to (51);
		\draw [thin=1.00] (48) to (52);	
\end{pgfonlayer}
\end{tikzpicture}
\end{center}
\end{proof}
We can conclude that for $C_m \times C_6,$ if $m$ is even,  and $m \not\equiv 0 \mod 4$, then $\lambda_1^1(C_m \times C_6) \geq 6$. For $m \geq 14$, this class of direct product graphs can be obtained from $C_{14+4m'}$, where $m'$ is a non-negative integer.

\begin{thm} \label{cc19} For $m=14+4m'$, where $m'$ is a non-negative integer, $\lambda_1^1(C_m \times C_6)=6$
\end{thm}
\begin{proof} Since for any non-negative integer $m'$, $m=14+4m' \not\equiv 0 \mod \; 4,$ then by \ref{cc14}, $\lambda_1^1(C_m \times C_n) \geq 6$. By combining the labeling in Figure 8 and $m'-$multiple of the labeling in Figure $3$, we have that $\lambda_1^1(C_m \times C_n)\leq 6$ and the result follows.
\end{proof}

Note that the following result was established in \cite{AA1}

\begin{thm}\label{ccc20} Let $m', n' \equiv 0 mod 10$ and $A=\left\{12,14,16,18\right\}$. Then, for all $k \in A and m, n$, $\lambda_1^1(C_{m'} \times C_{k+n'})=5$. Also, let $m,n \equiv 0 mod 5$, then $\lambda_1^1(C_m \times C_n)=4.$
\end{thm}

\section{Labeling of $C_m \times C_n,\ n\geq 8$}
In this section, we obtain the $\lambda_1^1-$numbers of graph product $C_m \times C_n$, where $n,m \geq 8$. %We define $G'$ as the connected component of the product graph under consideration.

Now we establish the $\lambda_1^1-$number for $C_m \times C_8$. Since labeling of product graphs is commutative, we restrict our work in this section to $m \geq 8$ since the cases for smaller graphs have been taken care of in the last sections. %As we have the earlier section, $m$ is even since $2m$ conveniently takes care of  odd $m$.

The following result are helpful to reveal some useful properties of $L(1,1)$- labeling of $C_m \times C_8$.

\begin{lem} \label{cc20}  Let $G'$ be a connected component of $C_m \times C_8, \;m \geq 4.$ Suppose there exist $v_a,v_b \in V_i$ such that $\alpha_k= l(v_a)=l(v_b) \in[p]$, $p\in \mathbb N.$ then $\alpha_k \notin L(V_{i+1} \cup V_{i+2})$. Furthermore, $\alpha_k \notin L(V_{i-1} \cup V_{i-2})$.
\end{lem}
\begin{proof} Claim: Let $\alpha_k=l(v_a)=l(v_b),$ $v_a,v_b \in V_i,$ $V_i \subseteq V(G').$ Then, $d(v_a,v_b)=4$.
\\ Reason: Clearly, $\left|V_i\right|=4$ for $i \in [m-1]$. Let $V_i=\left\{u_iv_0,u_iv_2,u_iv_4,u_iv_6\right\}$. So, $d(u_iv_0,u_iv_2)$=$d(u_iv_2,u_iv_4)$=$d(u_iv_4,u_iv_6)$=$2$. also, $d(u_iv_6,u_iv_4)=2$ since $C_8$ is a cycle. However, $d(u_iv_0,u_iv_4)$=$d(u_iv_2,u_iv_6)=4$. Thus $v_a=v_iv_0$ and $v_b=v_iv_4$ or $v_a=v_iv_2$ and $v_b=v_iv_6$.
Now, suppose $v_a=u_iv_0$ and $v_b=u_iv_4.$ Let $V_{i+1}=\left\{u_{i+1}v_1,u_{I+1}v_3,u_{i+1}v_5,u_{i+1}v_7\right\}$. Then that $d(u_iv_0,u_{i+1}v_1)=1=d(u_iv_0,u_{i+1}v_7)$ follows from the definition of $C_m \times C_{n=8}.$ Likewise, $d(u_iv_4,u_{i+1}v_3)=1=d(u_iv_4,u_{i+1}v_5)$. Therefore, $\alpha_k \in L(V_{i+1}).$ Also, let $V_{i+2}=\left\{u_{i+2}v_0,u_{i+2}v_2,u_{i+2}v_4,u_{i+2}v_6\right\}$. Then $d(u_iv_0,u_{i+2}v_0)=2$=$d(u_iv_0,u_{i+2}v_2)$ and $d(u_iv_0,u_{i+2}v_6)$=2=$d(u_iv_0,u_{i+2}v_2)$ since $C_{m=8}$ is a cycle. Now $d(u_iv_4,u_{i+2}v_{2(4,6)})=2$ and therefore $\alpha_k \notin L(V_{i+2})$. This argument is valid for $V_{i-1}$ and $V_{1-2}$.
\end{proof}
The consequence of Lemma \ref{cc20} is that if a label is assigned to two vertices on $V_i \subset V(G')$, then the label could no longer be assigned to another vertex on the vertex sets two step above or below it. The next result is similar.
\begin{prop} \label{cc21} Suppose $v_i \in V_i$ and $v_{i+2} \in V_{i+1}$ such that $\alpha_k =l(v_i)=l(v_{i+1}),$ then, $d(v_i,v_{i+1})=3$.
\end{prop}
\begin{proof} Suppose $v_i=u_iv_0$ without loss of generality, then $d(v_i,u_{i+1}v_{1(7)})=2.$ Now. $d(v_i, u_{i+1}v_3)=3$ and $d(v_i, u_{i+1}v_5)=3$.
\end{proof}
\begin{lem}
\label{cc22} Suppose $V_1,V_{i+1} \subset V(G')$ where $G'$ is a connected component of $C_m \times C_8$. Let $\alpha_k \in L(V_i) \cap L(V_{i+1})$ then $\alpha_k \notin L(V_{i-1}) \cup L(V_{i+2})$.
\end{lem}
\begin{proof}
By Proposition \ref{cc21}, suppose that $\alpha_k=l(u_iv_j)$ and that $\alpha_k \in L(V_{i+1})$ then, $\alpha_k=l(u_{i+1}v_{j+3})$ or $\alpha_k=l(u_{i+1}v_{j+5})$. Without loss of generality, suppose that in fact, $\alpha_k=l(u_{i+1}v_j+5)$. Then $d(u_iv_j,u_{i+2}v_{j(j+2,j+6)})=2$. Meanwhile, $d(u_{i+1}v_{j+3(j+5)},u_{i+2}v_{j+4})=1$. Thus, $\alpha_k \notin L(V_{i+2})$. Similar argument holds for $\alpha_k \notin L(V_{i-1})$.
\end{proof}
By to Lemma \ref{cc22}, it is quite clear that if $\alpha_k$ belongs $L(V_i) \cap L(V_i+1)$, then $\alpha_k$ does not belong to $L(V_{i-1} \cup V{i+2})$. A similar result is as follows:
\begin{lem} \label{cc23} Suppose $\alpha_k \in L(V_i) \cap L(V_{i+2}) \subset V(G')$, where $G'$ is a connected component of $C_m \times C_8$, $m \geq 4$, then $\alpha_k \notin L(V_{i+1})$
\end{lem}
\begin{proof} Let, $v_i \in V_i$ be $u_iv_0$. Note that, $d(u_iv_0,u_{i+2}v_{0(2,6)})=2$ since $C_8$ is a cycle. Then, the remaining vertex $v_{i+2} \in V_{i+2}$ such that $l(v_{i+2})=\alpha_k$ is $u_{i+2}v_4$ and $d(u_iv_0,u_{i+2}v_4)=4$. Now $d(v_i,u_{i+1}v_{1(7)})=1$ and $d(u_{i+2}v_4,u_{i+1}v_{3(5)})=1$. Thus, $\alpha_k \notin L(V_{i+1})$.
\end{proof}
The consequence of Lemma \ref{cc23} is that if two vertices on $V_i$ and $V_{i+2}$ share the same label, then that label can not be shared by another vertex on $V_{i+1}$ given that $V_i,V_{i+1}$ and $V_{i+2}$ are all in $V(G')$.

Next we establish the lower bound of $\lambda^1_1(C_m \times C_8)$ where $m \geq 8$ and $m\equiv 2 \mod 6$. We require the following definition.

Let $G'$ be a connected component of $G$. Then, $V_{\alpha_k} $ is the class of all vertices on $V(G')$ labeled $\alpha_k$.

\begin{lem} \label{cc24}For $m\geq 8,m \equiv 2\; mod \; 6$, $\lambda_1^1(C_m \times C_8) \geq 6$.
\end{lem}
\begin{proof}Case 1: Let $\alpha_k \in L(V(G'))$ such that if $\alpha_k \in L(V_i),\; V_i \subset V(G'),$ $i \in [m-1]$, then there exist $v'_i,v''_i \in V_i$ such that $l(v'_i)= \alpha_k=l(v''_i)$. Let $\bar{V}$ be a class of all $V_i \in V(G')$ such that $\alpha_k \in L(V_i)$. Now suppose, without loss of generality, that $V_0 \in \bar{V}$. By this and Lemma \ref{cc20}, and by assuming that $\alpha_k $ labels $V(G')$ optimally, suppose $V_i \in \bar{V}$, then $i\equiv 0\; \mod\; 3, i \neq m-2.$ Since $m\equiv 2 \; \mod \; 6$, then there exists $n' \in \mathbb N$, such that $m=6n'+2$. Note that $m-5=(6n'+2)-5=3(2n'-1)$. Thus, $V_{m-5} \in \bar{V}.$ By Lemma \ref{cc20} and since $V_0,V_{m-5} \in \bar{V}$, then $\alpha_k \notin L(V_{m-4} \cup V_{m-3}\cup V_{m-2}\cup V_{m-1})$. Thus $\bar{V}=\left\{V_0,\cdots,V_{m-5}\right\}$. Set $\bar{V'}=\bar{V}\backslash\left\{V_0\right\}$. Since $\left|\bar{V'}\right|=2n'-1$, then $\left|V'\right|=2n'$. Now, $\left|V(G')\right|=(6n'+2)4$. Clearly $\left|V_{\alpha_k}\right|=2(2n')=4n'$. Hence, $\frac{\left|V(G')\right|}{\left|V_{\alpha_k}\right|}$=$\frac{6n'+2}{n'} > 6$.

Case 2: Suppose that for all triple $V_i,V_{i+1}V_{i+2} \subset V(G)$, $\alpha_k \in L(V_i \cap V_{i+1})$ and by Lemma \ref{cc22} $\alpha_k \notin V_{i+2}$. Without loss of generality, we select the initial triple to be $V_0,V_1,V_2$, such that $\alpha_k \in L(V_0 \cap V_1)$, $\alpha_k \notin V_2$; (and $\alpha_k \in L(V_3 \cap V_4)$,$\alpha_k \notin L(V_5) \cdots )$. Therefore, $\alpha_k \notin V_i $ for all $i \in [m-1]$ such that $i+1\equiv \;0 \mod \; 3.$ Now, $m \equiv 2\; \mod\; 6$ implies there exists $n' \in \mathbb N$ such that $m=6n'+2$. Thus, $m-2 \equiv 0\; \mod \;3$ and hence $\alpha_k \notin L(V_{m-3}).$ Now, since $\alpha_k \in L(V_0) \cap L(V_1)$, then $\alpha_k \notin L(V_{m-1})$ by Lemma \ref{cc22} and since $C_m$ is a cycle. By Lemma \ref{cc23}, it is possible for $\alpha_k \in L(V_{m-2})$ since $\alpha_k \notin L(V_{m-3})$. Thus we, for maximality, assume that $\alpha_k \in L(V_{m-2})$. Now, Let $\bar{V_{\alpha_k}}$=$\left\{V_0,\cdots,V_{m-3}\right\} \subset V(G')$. Then $\left|\bar{V_{\alpha_k}}\right|=\left[\frac{(m-3)+1}{3}\right]=2(\frac{6n'}{3})=4n'$, where $n' \in \mathbb N$. Thus, for all $V_i \in V(G'),$ $\left|V_{\alpha_k}\right|=4n'+1$ since $\alpha_k \in L(V_{m-2})$. Thus $\frac{\left|V(G')\right|}{\left|V_{\alpha_k}\right|}=\frac{(6n'+2)4}{4n'+1} > 6$.

Case 3. Suppose that, by Lemma \ref{cc23}, $\alpha_k \in L(V_i,V_{i+2}, V_{i+4}, \cdots, V_{i-2})$. Clearly, $\left|V_{\alpha_k}\right|=\frac{m}{2}$, since $m$ is even. Now, $m=2\; \mod \;6$ implies that there exists $n' \in \mathbb N$ such that $m=6n'+2$. Therefore $\left|V_{\alpha_k}\right|=3n'+1$. Now $\frac{\left|V(G')\right|}{\left|V_{\alpha_k}\right|}=\frac{(6n'+2)4}{3n'+1} > 8$.
It is easy therefore to see that combination of the Cases 1-3 will still result in $\frac{\left|V(G')\right|}{V_{\alpha_k}} \geq 7.$ Thus for all $\alpha_k \in [p]$, where $\alpha_k \in [p]$, $\lambda_1^1(G')=p$,$\frac{\left|V(G')\right|}{\left|V_{\alpha_k}\right|}\geq 7$. Suppose $\lambda^1_1(G')=p=5$ and the maximum number of vertices in $G'$ that $\alpha_k \in [p]$ labels for all $\alpha_k \in [p]$ is $V_{\alpha_k}$, then $(p+1)V_{\alpha_k} \geq \left|V(G')\right|$ implies that $p+1 \geq $ $\frac{\left|V(G')\right|}{\left|V_{\alpha_k}\right|}$. This implies that $\frac{\left|V(G')\right|}{\left|V_{\alpha_k}\right|} \leq p+1$. Now, since $p=5$, then $\frac{\left|V(G')\right|}{\left|V_{\alpha_k}\right|} \leq 6$, which is a contradiction since in fact, $\frac{\left|V(G')\right|}{\left|V_{\alpha_k}\right|} \geq 7.$ \\Thus $\lambda_1^1(C_m \times C_8) \geq 6$ for all $m \equiv 2\; \mod \; 6$.
\end{proof}
Next, we consider the second case of $m \equiv 4 \; \mod \; 6$.
\begin{lem} \label{cc25} For $m \equiv 4\; mod\; 6$, $\lambda_1^1(C_m \times C_8) \geq 6$.
\end{lem}
 \begin{proof} Case 1: Let $G'$ be a connected component of $C_m \times C_8$, $m \equiv 4 \mod 6$ and let $\bar{V}$ be a set of $V_i \subset V(G')$  such that for all $i$ there exist $v'_i,v''_i \in V_i$ such that $l(v'_i)=\alpha_k=l(v''_i)$. Now suppose $V_0 \in \bar{V}.$ By Lemma \ref{cc20}, $\alpha_k \notin L(V_1 \cup V_2)$. Since $\bar{V},$ contains all possible $V_i \subset V(G)$ and since $V_0\in \bar{V}$, then for all $i \equiv 0\;\mod\;3,$ $V_i \in \bar{V}$ except for $i=m-1$ since $C_m$ is a cycle and $V_0 \in \bar{V}$. We know that $m=6n'+4$,  $n' \in \mathbb N$ and thus, $m-4=0 \; \mod\;3$, which implies that $V_{m-4} \in \bar{V}$. Set $\bar{V'}=\left\{V_3,\cdots ,V_{m-4}\right\}$. Thus, $\left|V'\right|=\frac{m-4}{3}=\frac{6n'+4-4}{3}=2n'$. Now, $\bar{V}=\bar{V}' \cup V_0$. Thus $\left|\bar{V}\right| 2n'+1$ and $\left|V_{\alpha_k}\right|=2(2n'+1)=4n'+2$. Now, $\left|V(G')\right|=4(6n'+4)=24n'+16.$ Finally, $\frac{\left|V(G)\right|}{\left|V_{\alpha_k}\right|}=\frac{24n'+16}{4n'+2} > 6 $

 Case 2: Suppose that for all triple $V_i,V_{i+1},V_{i+2} \subset V(G)$, $\alpha_k \in L(V_i) \cap L(V_{i+1})$ and $\alpha_k \notin L(V_{i+2})$. We can select the initial triple as $V_0,V_1,V_2$, that is, $\alpha_k \in L(V_0) \cap L(V_1)$ and $\alpha_k \notin L(V_2)$ (and subsequently, $\alpha_k \in L(V_3) \cap L(V_4)$ and $\alpha_k \notin L(V_5) \cdots $). Thus, $\alpha_k \notin V_i$ for all $i$ such that $i+1 \equiv 0\; \mod \; 3.$ Now since $m \equiv 4\; \mod \; 6$ there exists $n' \in \mathbb N$ such that $m\equiv 6n'+4$. Clearly, $m-1=6n'+3=3(2n'+1)\equiv 0\;\mod\; 3.$ However, $\alpha_k \notin L(V_{m-1})$ since $C_m$ is a cycle and by the Lemma \ref{cc22}. Therefore let $\bar{V}=\left\{V_0,V_1,V_2,\cdots,V_{m-2} \right\} \subseteq V(G')$. Then $\left|\bar{V}\right|=m-2+1=m-1$. Clearly $\left|V_{\alpha_k}\right|=2 \frac{\left|\bar{V}\right|}{3} = \frac{2(m-1)}{3}$. The last equation implies that $\frac{2(6n'+4-1)}{3}=$$\frac {2\cdot3(2n'+1)}{3}=2(n'+1)$, $n' \in \mathbb N.$ Now,$\left|V(G')\right|=4m=4(6n'+16)=24n'+16$. Therefore, $\frac{\left|V(G')\right|}{\left|V_{\alpha_k}\right|}=\frac{24n'+16}{4n'+2} > 6$.

Case 3: This follows similar argument as in Case 3, in the proof of Lemma \ref{cc24}.

Therefore for $m$ even, $m\equiv 4\; \mod \;6,$ $\lambda_1^1(G') \geq 6$ follows similar argument as in proof of Lemma \ref{cc24}.
\end{proof}
\begin{cor} \label{cc26}For all $m$ even, $m \not\equiv \ 0 mod 6$, $\lambda_1^1(C_m \times C_8) \geq 6.$
\end{cor}
\begin{proof} It follows from combining the results in Lemmas \ref{cc24} and \ref{cc25}.
\end{proof}
Next we obtain the $\lambda_1^1$ number of a special case of Corollary \ref{cc26}.
\begin{thm} \label{cc27} $\lambda_1^1(C_8 \times C_8)=7$
\end{thm}
\begin{proof} By following the the process in the proof of Lemma \ref{cc24}, we have that $\left|V_{\alpha_k}\right|=4$, for $V_{\alpha_k} \subseteq V(G')$, where $G'$ is a connected component of $C_8 \times C_8$ and therefore $\frac{\left|V(G')\right|}{\left|V_{\alpha_k}\right|}=\frac{32}{4}=8$. Thus $\lambda_1^1(C_8 \times C_8)\geq 7$. From an earlier result,
 %from \cite{AA2},
  $\lambda_1^1(C_4 \times C_8)=7$. By copying re-occurrence of the labeling of $C_4 \times C_8$, we have that $\lambda_1^1(C_8 \times C_8)\leq 7$ and the result follows.
\end{proof}
In what follows, we extend our result to $m \geq 10$.
\begin{thm}\label{cc28} Let $m \in \left\{10, 14\right\}$. Then $\lambda_1^1(C_{m} \times C_8)=6$
\end{thm}
\begin{proof}
By Corollary \ref{cc26}, $\lambda_1^1(C_{m} \times C_8) \geq 6$ for all $m \in \left\{10, 14\right\}$.
Conversely, we show that for $m \in \left\{10, 14\right\}$, $\lambda_1^1(C_m \times C_8) \leq 6$ by labeling their connected component as shown below.
%\newpage
{\tiny{
\begin{center}
\pgfdeclarelayer{nodelayer}
\pgfdeclarelayer{edgelayer}
\pgfsetlayers{nodelayer,edgelayer}
\begin{tikzpicture}
%\centering
	\begin{pgfonlayer}{nodelayer}
%\begin{center}
%\begin{tikzpicture}
%\centering
%	\begin{pgfonlayer}{nodelayer}
	
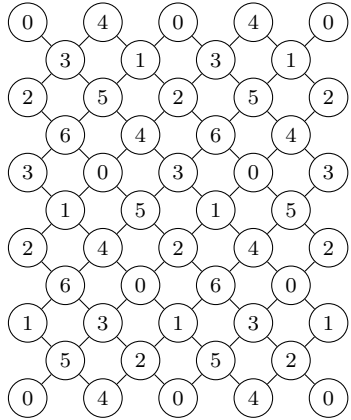
	\node [minimum size=0cm,]  at (-6.0,16) {$Fig. \; 9$: $6-L(1,1)$-Labeling of $C_{10} \times C_8$};
%	{\tiny{
		\node [minimum size=0cm,draw,circle] (0) at (-4,17) {$ 0$};
		\node [minimum size=0cm,draw,circle] (1) at (-5,17) {$ 4$};
		\node [minimum size=0cm,draw,circle] (2) at (-6,17) {$ 0$};
		\node [minimum size=0cm,draw,circle] (3) at (-7,17) {$ 4$};
		\node [minimum size=0cm,draw,circle] (4) at (-8,17) {$ 0$};
	
		\node [minimum size=0cm,draw,circle] (6) at (-4.5,17.5) {$ 2$};
		\node [minimum size=0cm,draw,circle] (7) at (-5.5,17.5) {$ 5$};
		\node [minimum size=0cm,draw,circle] (8) at (-6.5,17.5) {$ 2$};
		\node [minimum size=0cm,draw,circle] (9) at (-7.5,17.5) {$ 5$};

		\node [minimum size=0cm,draw,circle] (11) at (-4,18) {$ 1$};
		\node [minimum size=0cm,draw,circle] (12) at (-5,18) {$ 3$};
		\node [minimum size=0cm,draw,circle] (13) at (-6,18) {$ 1$};
		\node [minimum size=0cm,draw,circle] (14) at (-7,18) {$ 3$};
		\node [minimum size=0cm,draw,circle] (15) at (-8,18) {$ 1$};
	
		\node [minimum size=0cm,draw,circle] (17) at (-4.5,18.5) {$ 0$};
		\node [minimum size=0cm,draw,circle] (18) at (-5.5,18.5) {$ 6$};
		\node [minimum size=0cm,draw,circle] (19) at (-6.5,18.5) {$ 0$};
		\node [minimum size=0cm,draw,circle] (20) at (-7.5,18.5) {$ 6$};
		
        \node [minimum size=0cm,draw,circle] (22) at (-4,19) {$ 2$};
		\node [minimum size=0cm,draw,circle] (23) at (-5,19) {$ 4$};
		\node [minimum size=0cm,draw,circle] (24) at (-6,19) {$ 2$};
		\node [minimum size=0cm,draw,circle] (25) at (-7,19) {$ 4$};
		\node [minimum size=0cm,draw,circle] (26) at (-8,19) {$ 2$};
		
		\node [minimum size=0cm,draw,circle] (28) at (-4.5,19.5) {$ 5$};
		\node [minimum size=0cm,draw,circle] (29) at (-5.5,19.5) {$ 1$};
		\node [minimum size=0cm,draw,circle] (30) at (-6.5,19.5) {$ 5$};
		\node [minimum size=0cm,draw,circle] (31) at (-7.5,19.5) {$ 1$};
		
		\node [minimum size=0cm,draw,circle] (33) at (-4,20) {$ 3$};
		\node [minimum size=0cm,draw,circle] (34) at (-5,20) {$ 0$};
		\node [minimum size=0cm,draw,circle] (35) at (-6,20) {$ 3$};
		\node [minimum size=0cm,draw,circle] (36) at (-7,20) {$ 0$};
		\node [minimum size=0cm,draw,circle] (37) at (-8,20) {$ 3$};
		
		\node [minimum size=0cm,draw,circle] (39) at (-4.5,20.5) {$ 4$};
		\node [minimum size=0cm,draw,circle] (40) at (-5.5,20.5) {$ 6$};
		\node [minimum size=0cm,draw,circle] (41) at (-6.5,20.5) {$ 4$};
		\node [minimum size=0cm,draw,circle] (42) at (-7.5,20.5) {$ 6$};

		\node [minimum size=0cm,draw,circle] (44) at (-4,21) {$ 2$};
		\node [minimum size=0cm,draw,circle] (45) at (-5,21) {$ 5$};
		\node [minimum size=0cm,draw,circle] (46) at (-6,21) {$ 2$};
		\node [minimum size=0cm,draw,circle] (47) at (-7,21) {$ 5$};
		\node [minimum size=0cm,draw,circle] (48) at (-8,21) {$ 2$};
		
		\node [minimum size=0cm,draw,circle] (50) at (-4.5,21.5) {$ 1$};
		\node [minimum size=0cm,draw,circle] (51) at (-5.5,21.5) {$ 3$};
		\node [minimum size=0cm,draw,circle] (52) at (-6.5,21.5) {$ 1$};
		\node [minimum size=0cm,draw,circle] (53) at (-7.5,21.5) {$ 3$};
		
	 \node [minimum size=0cm,draw,circle] (55') at (-4,22) {$ 0$};
		\node [minimum size=0cm,draw,circle] (56') at (-5,22) {$ 4$};
		\node [minimum size=0cm,draw,circle] (57') at (-6,22) {$ 0$};
		\node [minimum size=0cm,draw,circle] (58') at (-7,22) {$ 4$};
		\node [minimum size=0cm,draw,circle] (59') at (-8,22) {$ 0$};
	
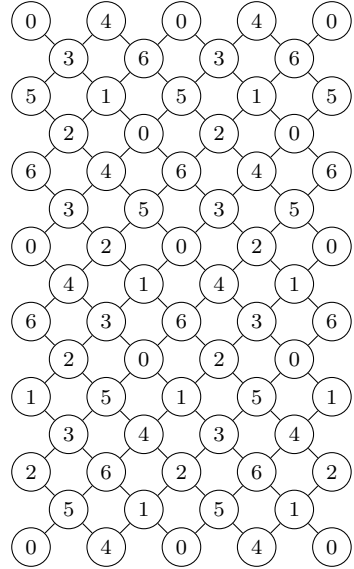
\node [minimum size=0cm,]  at (0,16) {$Fig. \; 10$: $6-L(1,1)$-Labeling of $C_{14} \times C_{8}$};
%{\tiny{ 	
		\node [minimum size=0cm,draw,circle] (70) at (2,17) {$ 0$};
		\node [minimum size=0cm,draw,circle] (71) at (1,17) {$ 4$};
		\node [minimum size=0cm,draw,circle] (72) at (0,17) {$ 0$};
		\node [minimum size=0cm,draw,circle] (73) at (-1,17) {$ 4$};
		\node [minimum size=0cm,draw,circle] (74) at (-2,17){$ 0$};
	
		\node [minimum size=0cm,draw,circle] (76) at (1.5,17.5) {$ 1$};
		\node [minimum size=0cm,draw,circle] (77) at (0.5,17.5) {$ 5$};
		\node [minimum size=0cm,draw,circle] (78) at (-0.5,17.5) {$ 1$};
		\node [minimum size=0cm,draw,circle] (79) at (-1.5,17.5) {$ 5$};

		\node [minimum size=0cm,draw,circle] (81) at (2,18) {$ 2$};
		\node [minimum size=0cm,draw,circle] (82) at (1,18) {$ 6$};
		\node [minimum size=0cm,draw,circle] (83) at (0,18) {$ 2$};
		\node [minimum size=0cm,draw,circle] (84) at (-1,18) {$ 6$};
		\node [minimum size=0cm,draw,circle] (85) at (-2,18){$ 2$};
	
		\node [minimum size=0cm,draw,circle] (87) at (1.5,18.5) {$ 4$};
		\node [minimum size=0cm,draw,circle] (88) at (0.5,18.5) {$ 3$};
		\node [minimum size=0cm,draw,circle] (89) at (-0.5,18.5) {$ 4$};
		\node [minimum size=0cm,draw,circle] (90) at (-1.5,18.5) {$ 3$};
		
    \node [minimum size=0cm,draw,circle] (92) at (2,19) {$ 1$};
		\node [minimum size=0cm,draw,circle] (93) at (1,19) {$ 5$};
		\node [minimum size=0cm,draw,circle] (94) at (0,19) {$ 1$};
		\node [minimum size=0cm,draw,circle] (95) at (-1,19) {$ 5$};
		\node [minimum size=0cm,draw,circle] (96) at (-2,19){$ 1$};
		
		\node [minimum size=0cm,draw,circle] (98) at (1.5,19.5) {$ 0$};
		\node [minimum size=0cm,draw,circle] (99) at (0.5,19.5) {$ 2$};
		\node [minimum size=0cm,draw,circle] (100) at (-0.5,19.5) {$ 0$};
		\node [minimum size=0cm,draw,circle] (101) at (-1.5,19.5) {$ 2$};
		
		\node [minimum size=0cm,draw,circle] (103) at (2,20) {$ 6$};
		\node [minimum size=0cm,draw,circle] (104) at (1,20) {$ 3$};
		\node [minimum size=0cm,draw,circle] (105) at (0,20) {$ 6$};
		\node [minimum size=0cm,draw,circle] (106) at (-1,20) {$ 3$};
		\node [minimum size=0cm,draw,circle] (107) at (-2,20){$ 6$};
		
		\node [minimum size=0cm,draw,circle] (109) at (1.5,20.5) {$ 1$};
		\node [minimum size=0cm,draw,circle] (110) at (0.5,20.5) {$ 4$};
		\node [minimum size=0cm,draw,circle] (111) at (-0.5,20.5){$1$};
		\node [minimum size=0cm,draw,circle] (112) at (-1.5,20.5) {$4$};

		\node [minimum size=0cm,draw,circle] (114) at (2,21) {$ 0$};
		\node [minimum size=0cm,draw,circle] (115) at (1,21) {$ 2$};
		\node [minimum size=0cm,draw,circle] (116) at (0,21) {$ 0$};
		\node [minimum size=0cm,draw,circle] (117) at (-1,21) {$ 2$};
		\node [minimum size=0cm,draw,circle] (118) at (-2,21){$ 0$};
		
		\node [minimum size=0cm,draw,circle] (120) at (1.5,21.5) {$ 5$};
		\node [minimum size=0cm,draw,circle] (121) at (0.5,21.5) {$ 3$};
		\node [minimum size=0cm,draw,circle] (122) at (-0.5,21.5){$ 5$};
		\node [minimum size=0cm,draw,circle] (123) at (-1.5,21.5){$ 3$};
		
	 \node [minimum size=0cm,draw,circle] (125') at (2,22) {$ 6$};
		\node [minimum size=0cm,draw,circle] (126') at (1,22) {$ 4$};
		\node [minimum size=0cm,draw,circle] (127') at (0,22) {$ 6$};
		\node [minimum size=0cm,draw,circle] (128') at (-1,22) {$ 4$};
		\node [minimum size=0cm,draw,circle] (129') at (-2,22) {$ 6$};
		
		\node [minimum size=0cm,draw,circle] (130) at (1.5,22.5) {$ 0$};
		\node [minimum size=0cm,draw,circle] (131) at (0.5,22.5) {$ 2$};
		\node [minimum size=0cm,draw,circle] (132) at (-0.5,22.5){$ 0$};
		\node [minimum size=0cm,draw,circle] (133) at (-1.5,22.5) {$ 2$};

		\node [minimum size=0cm,draw,circle] (134) at (2,23) {$ 5$};
		\node [minimum size=0cm,draw,circle] (135) at (1,23) {$ 1$};
		\node [minimum size=0cm,draw,circle] (136) at (0,23) {$ 5$};
		\node [minimum size=0cm,draw,circle] (137) at (-1,23) {$ 1$};
		\node [minimum size=0cm,draw,circle] (138) at (-2,23) {$ 5$};
		
		\node [minimum size=0cm,draw,circle] (139) at (1.5,23.5) {$ 6$};
		\node [minimum size=0cm,draw,circle] (140) at (0.5,23.5) {$ 3$};
		\node [minimum size=0cm,draw,circle] (141) at (-0.5,23.5) {$ 6$};
		\node [minimum size=0cm,draw,circle] (142) at (-1.5,23.5) {$ 3$};
		
	  \node [minimum size=0cm,draw,circle] (143) at (2,24) {$ 0$};
		\node [minimum size=0cm,draw,circle] (144) at (1,24) {$ 4$};
		\node [minimum size=0cm,draw,circle] (145) at (0,24) {$ 0$};
		\node [minimum size=0cm,draw,circle] (146) at (-1,24) {$ 4$};
		\node [minimum size=0cm,draw,circle] (147) at (-2,24){$ 0$};
			
\end{pgfonlayer}
	\begin{pgfonlayer}{edgelayer}

   \draw [thin=1.00] (0) to (6);
   \draw [thin=1.00] (1) to (6);
   \draw [thin=1.00] (1) to (7);
   \draw [thin=1.00] (2) to (7);
	  \draw [thin=1.00] (2) to (8);
		\draw [thin=1.00] (3) to (8);
		\draw [thin=1.00] (3) to (9);
		\draw [thin=1.00] (4) to (9);

		\draw [thin=1.00] (6) to (11);
		\draw [thin=1.00] (6) to (12);
		\draw [thin=1.00] (7) to (12);
		\draw [thin=1.00] (7) to (13);
		\draw [thin=1.00] (8) to (13);
		\draw [thin=1.00] (8) to (14);
   \draw [thin=1.00] (9) to (14);
   \draw [thin=1.00] (9) to (15);

		\draw [thin=1.00] (11) to (17);
		\draw [thin=1.00] (12) to (17);
		\draw [thin=1.00] (12) to (18);
		\draw [thin=1.00] (13) to (18);
		\draw [thin=1.00] (13) to (19);
		\draw [thin=1.00] (14) to (19);
		\draw [thin=1.00] (14) to (20);
		\draw [thin=1.00] (15) to (20);
		
   \draw [thin=1.00] (17) to (22);
   \draw [thin=1.00] (17) to (23);
   \draw [thin=1.00] (18) to (23);
   \draw [thin=1.00] (18) to (24);
   \draw [thin=1.00] (19) to (24);
		\draw [thin=1.00] (19) to (25);
		\draw [thin=1.00] (20) to (25);
	  \draw [thin=1.00] (20) to (26);

		\draw [thin=1.00] (22) to (28);
		\draw [thin=1.00] (23) to (28);
		\draw [thin=1.00] (23) to (29);
		\draw [thin=1.00] (24) to (29);
		\draw [thin=1.00] (24) to (30);
		\draw [thin=1.00] (25) to (30);
		\draw [thin=1.00] (25) to (31);
		\draw [thin=1.00] (26) to (31);

   \draw [thin=1.00] (28) to (33);
   \draw [thin=1.00] (28) to (34);
   \draw [thin=1.00] (29) to (34);
	  \draw [thin=1.00] (29) to (35);
		\draw [thin=1.00] (30) to (35);
		\draw [thin=1.00] (30) to (36);
		\draw [thin=1.00] (31) to (36);
		\draw [thin=1.00] (31) to (37);

		\draw [thin=1.00] (33) to (39);
		\draw [thin=1.00] (34) to (39);
		\draw [thin=1.00] (34) to (40);
		\draw [thin=1.00] (35) to (40);
		\draw [thin=1.00] (35) to (41);
   \draw [thin=1.00] (36) to (41);
   \draw [thin=1.00] (36) to (42);
   \draw [thin=1.00] (37) to (42);

		\draw [thin=1.00] (39) to (44);
		\draw [thin=1.00] (39) to (45);
		\draw [thin=1.00] (40) to (45);
		\draw [thin=1.00] (40) to (46);
		\draw [thin=1.00] (41) to (46);
		\draw [thin=1.00] (41) to (47);
		\draw [thin=1.00] (42) to (47);
		\draw [thin=1.00] (42) to (48);

   \draw [thin=1.00] (44) to (50);
   \draw [thin=1.00] (45) to (50);
   \draw [thin=1.00] (45) to (51);
   \draw [thin=1.00] (46) to (51);
		\draw [thin=1.00] (46) to (52);
		\draw [thin=1.00] (47) to (52);
	  \draw [thin=1.00] (47) to (53);
		\draw [thin=1.00] (48) to (53);

		\draw [thin=1.00] (50) to (55');
		\draw [thin=1.00] (50) to (56');
	  \draw [thin=1.00] (51) to (56');
		\draw [thin=1.00] (51) to (57');		
		\draw [thin=1.00] (52) to (57');
		\draw [thin=1.00] (52) to (58');
		\draw [thin=1.00] (53) to (58');		
		\draw [thin=1.00] (53) to (59');		

 \draw [thin=1.00] (70) to (76);
   \draw [thin=1.00] (71) to (76);
   \draw [thin=1.00] (71) to (77);
   \draw [thin=1.00] (72) to (77);
	  \draw [thin=1.00] (72) to (78);
		\draw [thin=1.00] (73) to (78);
		\draw [thin=1.00] (73) to (79);
		\draw [thin=1.00] (74) to (79);

		\draw [thin=1.00] (76) to (81);
		\draw [thin=1.00] (76) to (82);
		\draw [thin=1.00] (77) to (82);
		\draw [thin=1.00] (77) to (83);
		\draw [thin=1.00] (78) to (83);
		\draw [thin=1.00] (78) to (84);
   \draw [thin=1.00] (79) to (84);
   \draw [thin=1.00] (79) to (85);

		\draw [thin=1.00] (81) to (87);
		\draw [thin=1.00] (82) to (87);
		\draw [thin=1.00] (82) to (88);
		\draw [thin=1.00] (83) to (88);
		\draw [thin=1.00] (83) to (89);
		\draw [thin=1.00] (84) to (89);
		\draw [thin=1.00] (84) to (90);
		\draw [thin=1.00] (85) to (90);
		
   \draw [thin=1.00] (87) to (92);
   \draw [thin=1.00] (87) to (93);
   \draw [thin=1.00] (88) to (93);
   \draw [thin=1.00] (88) to (94);
   \draw [thin=1.00] (89) to (94);
		\draw [thin=1.00] (89) to (95);
		\draw [thin=1.00] (90) to (95);
	  \draw [thin=1.00] (90) to (96);

		\draw [thin=1.00] (92) to (98);
		\draw [thin=1.00] (93) to (98);
		\draw [thin=1.00] (93) to (99);
		\draw [thin=1.00] (94) to (99);
		\draw [thin=1.00] (94) to (100);
		\draw [thin=1.00] (95) to (100);
		\draw [thin=1.00] (95) to (101);
		\draw [thin=1.00] (96) to (101);

   \draw [thin=1.00] (98) to (103);
   \draw [thin=1.00] (98) to (104);
   \draw [thin=1.00] (99) to (104);
	  \draw [thin=1.00] (99) to (105);
		\draw [thin=1.00] (100) to (105);
		\draw [thin=1.00] (100) to (106);
		\draw [thin=1.00] (101) to (106);
		\draw [thin=1.00] (101) to (107);

		\draw [thin=1.00] (103) to (109);
		\draw [thin=1.00] (104) to (109);
		\draw [thin=1.00] (104) to (110);
		\draw [thin=1.00] (105) to (110);
		\draw [thin=1.00] (105) to (111);
   \draw [thin=1.00] (106) to (111);
   \draw [thin=1.00] (106) to (112);
   \draw [thin=1.00] (107) to (112);

		\draw [thin=1.00] (109) to (114);
		\draw [thin=1.00] (109) to (115);
		\draw [thin=1.00] (110) to (115);
		\draw [thin=1.00] (110) to (116);
		\draw [thin=1.00] (111) to (116);
		\draw [thin=1.00] (111) to (117);
		\draw [thin=1.00] (112) to (117);
		\draw [thin=1.00] (112) to (118);

   \draw [thin=1.00] (114) to (120);
   \draw [thin=1.00] (115) to (120);
   \draw [thin=1.00] (115) to (121);
   \draw [thin=1.00] (116) to (121);
		\draw [thin=1.00] (116) to (122);
		\draw [thin=1.00] (117) to (122);
	  \draw [thin=1.00] (117) to (123);
		\draw [thin=1.00] (118) to (123);

		\draw [thin=1.00] (120) to (125');
		\draw [thin=1.00] (120) to (126');
	  \draw [thin=1.00] (121) to (126');
		\draw [thin=1.00] (121) to (127');		
		\draw [thin=1.00] (122) to (127');
		\draw [thin=1.00] (122) to (128');
		\draw [thin=1.00] (123) to (128');		
		\draw [thin=1.00] (123) to (129');

   	\draw [thin=1.00] (125') to (130);
		\draw [thin=1.00] (126') to (130);
		\draw [thin=1.00] (126') to (131);
		\draw [thin=1.00] (127') to (131);
		\draw [thin=1.00] (127') to (132);
   \draw [thin=1.00] (128') to (132);
   \draw [thin=1.00] (128') to (133);
   \draw [thin=1.00] (129') to (133);

		\draw [thin=1.00] (130) to (134);
		\draw [thin=1.00] (130) to (135);
		\draw [thin=1.00] (131) to (135);
		\draw [thin=1.00] (131) to (136);
		\draw [thin=1.00] (132) to (136);
		\draw [thin=1.00] (132) to (137);
		\draw [thin=1.00] (133) to (137);
		\draw [thin=1.00] (133) to (138);

   \draw [thin=1.00] (134) to (139);
   \draw [thin=1.00] (135) to (139);
   \draw [thin=1.00] (135) to (140);
   \draw [thin=1.00] (136) to (140);
		\draw [thin=1.00] (136) to (141);
		\draw [thin=1.00] (137) to (141);
	  \draw [thin=1.00] (137) to (142);
		\draw [thin=1.00] (138) to (142);

		\draw [thin=1.00] (139) to (143);
		\draw [thin=1.00] (139) to (144);
	  \draw [thin=1.00] (140) to (144);
		\draw [thin=1.00] (140) to (145);		
		\draw [thin=1.00] (141) to (145);
		\draw [thin=1.00] (141) to (146);
		\draw [thin=1.00] (142) to (146);		
		\draw [thin=1.00] (142) to (147);		
\end{pgfonlayer} {edgelayer}
\end{tikzpicture}
\end{center}
}}

\end{proof}
%\newpage
In the next result we show that for all $m \equiv 2\;\mod\;6$ and $m \equiv 4\;\mod\;6$, $m\geq 14$, $\lambda_1^1(C_m \times C_8)=6$
{\tiny{
\begin{center}
\pgfdeclarelayer{nodelayer}
\pgfdeclarelayer{edgelayer}
\pgfsetlayers{nodelayer,edgelayer}
\begin{tikzpicture}
%\centering
	\begin{pgfonlayer}{nodelayer}
	
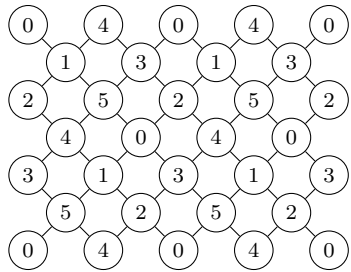
	\node [minimum size=0cm,]  at (-4,12) {$Fig. \; 11$: $5-L(1,1)$-Labeling of $C_{6} \times C_8$};
%{\tiny{ 	
		\node [minimum size=0cm,draw,circle] (0) at (-2,13) {$ 0$};
		\node [minimum size=0cm,draw,circle] (1) at (-3,13) {$ 4$};
		\node [minimum size=0cm,draw,circle] (2) at (-4,13) {$ 0$};
		\node [minimum size=0cm,draw,circle] (3) at (-5,13) {$ 4$};
		\node [minimum size=0cm,draw,circle] (4) at (-6,13) {$ 0$};
	
		\node [minimum size=0cm,draw,circle] (6) at (-2.5,13.5) {$ 2$};
		\node [minimum size=0cm,draw,circle] (7) at (-3.5,13.5) {$ 5$};
		\node [minimum size=0cm,draw,circle] (8) at (-4.5,13.5) {$ 2$};
		\node [minimum size=0cm,draw,circle] (9) at (-5.5,13.5) {$ 5$};

		\node [minimum size=0cm,draw,circle] (11) at (-2,14) {$ 3$};
		\node [minimum size=0cm,draw,circle] (12) at (-3,14) {$ 1$};
		\node [minimum size=0cm,draw,circle] (13) at (-4,14) {$ 3$};
		\node [minimum size=0cm,draw,circle] (14) at (-5,14) {$ 1$};
		\node [minimum size=0cm,draw,circle] (15) at (-6,14) {$ 3$};
	
		\node [minimum size=0cm,draw,circle] (17) at (-2.5,14.5) {$ 0$};
		\node [minimum size=0cm,draw,circle] (18) at (-3.5,14.5) {$ 4$};
		\node [minimum size=0cm,draw,circle] (19) at (-4.5,14.5) {$ 0$};
		\node [minimum size=0cm,draw,circle] (20) at (-5.5,14.5) {$ 4$};
		
    \node [minimum size=0cm,draw,circle] (22) at (-2,15) {$ 2$};
		\node [minimum size=0cm,draw,circle] (23) at (-3,15) {$ 5$};
		\node [minimum size=0cm,draw,circle] (24) at (-4,15) {$ 2$};
		\node [minimum size=0cm,draw,circle] (25) at (-5,15) {$ 5$};
		\node [minimum size=0cm,draw,circle] (26) at (-6,15) {$ 2$};
		
		\node [minimum size=0cm,draw,circle] (28) at (-2.5,15.5) {$ 3$};
		\node [minimum size=0cm,draw,circle] (29) at (-3.5,15.5) {$ 1$};
		\node [minimum size=0cm,draw,circle] (30) at (-4.5,15.5) {$ 3$};
		\node [minimum size=0cm,draw,circle] (31) at (-5.5,15.5) {$ 1$};
		
		\node [minimum size=0cm,draw,circle] (33) at (-2,16) {$ 0$};
		\node [minimum size=0cm,draw,circle] (34) at (-3,16) {$ 4$};
		\node [minimum size=0cm,draw,circle] (35) at (-4,16) {$ 0$};
		\node [minimum size=0cm,draw,circle] (36) at (-5,16) {$ 4$};
		\node [minimum size=0cm,draw,circle] (37) at (-6,16) {$ 0$};
\end{pgfonlayer}
	\begin{pgfonlayer}{edgelayer}

   \draw [thin=1.00] (0) to (6);
   \draw [thin=1.00] (1) to (6);
   \draw [thin=1.00] (1) to (7);
   \draw [thin=1.00] (2) to (7);
	  \draw [thin=1.00] (2) to (8);
		\draw [thin=1.00] (3) to (8);
		\draw [thin=1.00] (3) to (9);
		\draw [thin=1.00] (4) to (9);

		\draw [thin=1.00] (6) to (11);
		\draw [thin=1.00] (6) to (12);
		\draw [thin=1.00] (7) to (12);
		\draw [thin=1.00] (7) to (13);
		\draw [thin=1.00] (8) to (13);
		\draw [thin=1.00] (8) to (14);
    \draw [thin=1.00] (9) to (14);
    \draw [thin=1.00] (9) to (15);

		\draw [thin=1.00] (11) to (17);
		\draw [thin=1.00] (12) to (17);
		\draw [thin=1.00] (12) to (18);
		\draw [thin=1.00] (13) to (18);
		\draw [thin=1.00] (13) to (19);
		\draw [thin=1.00] (14) to (19);
		\draw [thin=1.00] (14) to (20);
		\draw [thin=1.00] (15) to (20);
		
   \draw [thin=1.00] (17) to (22);
   \draw [thin=1.00] (17) to (23);
   \draw [thin=1.00] (18) to (23);
   \draw [thin=1.00] (18) to (24);
   \draw [thin=1.00] (19) to (24);
		\draw [thin=1.00] (19) to (25);
		\draw [thin=1.00] (20) to (25);
	  \draw [thin=1.00] (20) to (26);

		\draw [thin=1.00] (22) to (28);
		\draw [thin=1.00] (23) to (28);
		\draw [thin=1.00] (23) to (29);
		\draw [thin=1.00] (24) to (29);
		\draw [thin=1.00] (24) to (30);
		\draw [thin=1.00] (25) to (30);
		\draw [thin=1.00] (25) to (31);
		\draw [thin=1.00] (26) to (31);

   \draw [thin=1.00] (28) to (33);
   \draw [thin=1.00] (28) to (34);
   \draw [thin=1.00] (29) to (34);
	  \draw [thin=1.00] (29) to (35);
		\draw [thin=1.00] (30) to (35);
		\draw [thin=1.00] (30) to (36);
		\draw [thin=1.00] (31) to (36);
		\draw [thin=1.00] (31) to (37);
	
\end{pgfonlayer} {edgelayer}
\end{tikzpicture}
\end{center}
}}

\begin{thm}\label{cc29} Let $m \not\equiv 0\; \mod\; 6$, $m\geq 10$ and even. Then $\lambda_1^1(C_m \times C_8)=6$.
\end{thm}
\begin{proof}By Corollary \ref{cc26}, we see that for $m \not\equiv 0\; \mod\; 6$, $m\geq 16$, $\lambda_1^1(C_m \times C_8) \geq 6$. Now, by combining the $m'-$copies of the labeling in Figure 1, with the $n'-$copies of labeling in Figure 11,  $m', n' \in \mathbb N$ we have that $\lambda_1^1(C_{10m'+6n'} \times C_8) \leq 6$, with $10m'+6n' \equiv 4\;\mod\;6$. By combining the labeling in Figure 10 with the $n'-$copies labeling Figure 11, $n' \geq 1$, $n' \in \mathbb N$ we have that $\lambda_1^1(C_{14+6n'} \times C_8) \leq 6$, with $14+6n' \equiv 4\;\mod\;6$. Thus, $C_m \times C_8 \leq 6$ for all $m\geq 16, m \not\equiv 0\;\mod\;6.$ Note that if $n', m'=0$, then we have the $L(1,1)$-labeling of $C_m \times C_8$, where $m\in \left\{10,14\right\}$, which are done in Theorem \ref{cc28}
\end{proof}

In what comes next, we obtain the  $\lambda_1^1(C_m \times C_{10})$. Our result will be based on that of $P_m \times C_n$. %(See \cite{AA1}).
 \begin{lem} For all $m\geq 9, \ n\geq 12$, $\lambda_1^1(C_m \times C_n) \geq 5$.
 \end{lem}
 \begin{proof} It is easy to see that $P_m \times C_n \subseteq C_m \times C_n$. Therefore the claim follows from Lemma \ref{f}.
 \end{proof}

 Now that the lower bound has been shown for $C_m \times C_n$, for specific lengths of cycles, we proceed to establish the optimal $L(1,1)$-numbers for various graphs in this class. In the case of $C_m \times C_{10}$, see Theorem \ref{ccc20}.

\begin{thm} For $m\geq 3$, $\lambda_1^1(C_m\times C_{12})=5$
\end{thm}
\begin{proof}
For all $m \equiv 0\mod 4,$ or  $m \equiv 0\mod 6$, and by commutativity of $C_m \times C_n$ the claim follows from Corollary \ref{h}. We now need to show the result for $m\not\equiv 0\mod 4$, $m\not\equiv 0\mod6$. It is easy to see that such number, $m'$, is obtainable from this formula: $m'=p+2,$ where $p\in \mathbb{N},\  p \equiv 0 \mod 4,  0 \mod 6$. The first of such number is 14. We need a $5-$labeling of $C_{14} \times C_{12}$.

{\tiny{
\begin{center}
\pgfdeclarelayer{nodelayer}
\pgfdeclarelayer{edgelayer}
\pgfsetlayers{nodelayer,edgelayer}
\begin{tikzpicture}

	\begin{pgfonlayer}{nodelayer}
	
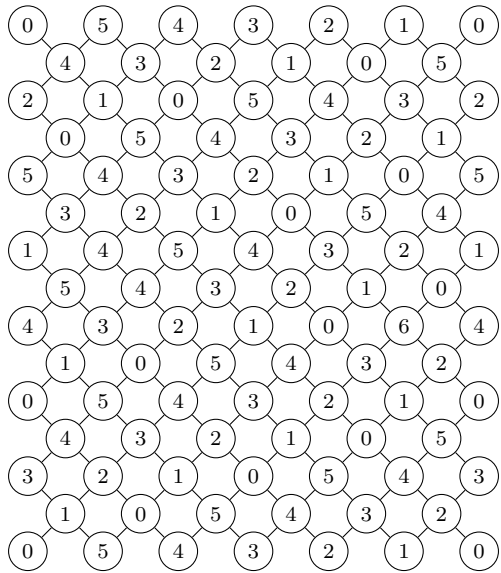
	\node [minimum size=0cm,]  at (-5,16) {$Fig. \; 12$: $5-L(1,1)$-Labeling of $C_{14} \times C_{12}$};
%{\tiny{ 	
		\node [minimum size=0cm,draw,circle] (0) at (-2,17) {$ 0$};
		\node [minimum size=0cm,draw,circle] (1) at (-3,17) {$ 1$};
		\node [minimum size=0cm,draw,circle] (2) at (-4,17) {$ 2$};
		\node [minimum size=0cm,draw,circle] (3) at (-5,17) {$ 3$};
		\node [minimum size=0cm,draw,circle] (4) at (-6,17) {$ 4$};
	  \node [minimum size=0cm,draw,circle] (5) at (-7,17) {$ 5$};
		\node [minimum size=0cm,draw,circle] (6) at (-8,17) {$ 0$};
		
		\node [minimum size=0cm,draw,circle] (7) at (-2.5,17.5) {$ 2$};
		\node [minimum size=0cm,draw,circle] (8) at (-3.5,17.5) {$ 3$};
		\node [minimum size=0cm,draw,circle] (9) at (-4.5,17.5) {$ 4$};
		\node [minimum size=0cm,draw,circle] (10) at (-5.5,17.5){$ 5$};
    \node [minimum size=0cm,draw,circle] (11) at (-6.5,17.5){$ 0$};
		\node [minimum size=0cm,draw,circle] (12) at (-7.5,17.5){$ 1$};
		
		\node [minimum size=0cm,draw,circle] (13) at (-2,18) {$ 3$};
		\node [minimum size=0cm,draw,circle] (14) at (-3,18) {$ 4$};
		\node [minimum size=0cm,draw,circle] (15) at (-4,18) {$ 5$};
		\node [minimum size=0cm,draw,circle] (16) at (-5,18) {$ 0$};
		\node [minimum size=0cm,draw,circle] (17) at (-6,18) {$ 1$};
	  \node [minimum size=0cm,draw,circle] (18) at (-7,18) {$ 2$};
		\node [minimum size=0cm,draw,circle] (19) at (-8,18) {$ 3$};
		
		\node [minimum size=0cm,draw,circle] (20) at (-2.5,18.5) {$ 5$};
		\node [minimum size=0cm,draw,circle] (21) at (-3.5,18.5) {$ 0$};
		\node [minimum size=0cm,draw,circle] (22) at (-4.5,18.5) {$ 1$};
		\node [minimum size=0cm,draw,circle] (23) at (-5.5,18.5) {$ 2$};
  	\node [minimum size=0cm,draw,circle] (24) at (-6.5,18.5) {$ 3$};
		\node [minimum size=0cm,draw,circle] (25) at (-7.5,18.5) {$ 4$};
			
    \node [minimum size=0cm,draw,circle] (26) at (-2,19) {$ 0$};
		\node [minimum size=0cm,draw,circle] (27) at (-3,19) {$ 1$};
		\node [minimum size=0cm,draw,circle] (28) at (-4,19) {$ 2$};
		\node [minimum size=0cm,draw,circle] (29) at (-5,19) {$ 3$};
		\node [minimum size=0cm,draw,circle] (30) at (-6,19) {$ 4$};
		\node [minimum size=0cm,draw,circle] (31) at (-7,19) {$ 5$};
		\node [minimum size=0cm,draw,circle] (32) at (-8,19) {$ 0$};
		
		\node [minimum size=0cm,draw,circle] (33) at (-2.5,19.5) {$ 2$};
		\node [minimum size=0cm,draw,circle] (34) at (-3.5,19.5) {$ 3$};
		\node [minimum size=0cm,draw,circle] (35) at (-4.5,19.5) {$ 4$};
		\node [minimum size=0cm,draw,circle] (36) at (-5.5,19.5) {$ 5$};
		\node [minimum size=0cm,draw,circle] (37) at (-6.5,19.5) {$ 0$};
		\node [minimum size=0cm,draw,circle] (38) at (-7.5,19.5) {$ 1$};

		\node [minimum size=0cm,draw,circle] (39) at (-2,20) {$ 4$};
		\node [minimum size=0cm,draw,circle] (40) at (-3,20) {$ 6$};
		\node [minimum size=0cm,draw,circle] (41) at (-4,20) {$ 0$};
		\node [minimum size=0cm,draw,circle] (42) at (-5,20) {$ 1$};
		\node [minimum size=0cm,draw,circle] (43) at (-6,20) {$ 2$};
		\node [minimum size=0cm,draw,circle] (44) at (-7,20) {$ 3$};
		\node [minimum size=0cm,draw,circle] (45) at (-8,20) {$ 4$};
		
		\node [minimum size=0cm,draw,circle] (46) at (-2.5,20.5) {$ 0$};
		\node [minimum size=0cm,draw,circle] (47) at (-3.5,20.5) {$ 1$};
		\node [minimum size=0cm,draw,circle] (48) at (-4.5,20.5) {$ 2$};
		\node [minimum size=0cm,draw,circle] (49) at (-5.5,20.5) {$ 3$};
    \node [minimum size=0cm,draw,circle] (50) at (-6.5,20.5) {$ 4$};
		\node [minimum size=0cm,draw,circle] (51) at (-7.5,20.5) {$ 5$};
		
		\node [minimum size=0cm,draw,circle] (52) at (-2,21) {$ 1$};
		\node [minimum size=0cm,draw,circle] (53) at (-3,21) {$ 2$};
		\node [minimum size=0cm,draw,circle] (54) at (-4,21) {$ 3$};
		\node [minimum size=0cm,draw,circle] (55) at (-5,21) {$ 4$};
		\node [minimum size=0cm,draw,circle] (56) at (-6,21) {$ 5$};
   	\node [minimum size=0cm,draw,circle] (57) at (-7,21) {$ 4$};
		\node [minimum size=0cm,draw,circle] (58) at (-8,21) {$ 1$};
			
		\node [minimum size=0cm,draw,circle] (59) at (-2.5,21.5) {$ 4$};
		\node [minimum size=0cm,draw,circle] (60) at (-3.5,21.5) {$ 5$};
		\node [minimum size=0cm,draw,circle] (61) at (-4.5,21.5) {$ 0$};
		\node [minimum size=0cm,draw,circle] (62) at (-5.5,21.5) {$ 1$};
	  \node [minimum size=0cm,draw,circle] (63) at (-6.5,21.5) {$ 2$};
		\node [minimum size=0cm,draw,circle] (64) at (-7.5,21.5) {$ 3$};
			
	  \node [minimum size=0cm,draw,circle] (65) at (-2,22) {$ 5$};
		\node [minimum size=0cm,draw,circle] (66) at (-3,22) {$ 0$};
		\node [minimum size=0cm,draw,circle] (67) at (-4,22) {$ 1$};
		\node [minimum size=0cm,draw,circle] (68) at (-5,22) {$ 2$};
		\node [minimum size=0cm,draw,circle] (69) at (-6,22) {$ 3$};
	  \node [minimum size=0cm,draw,circle] (70) at (-7,22) {$ 4$};
		\node [minimum size=0cm,draw,circle] (71) at (-8,22) {$ 5$};
			
		\node [minimum size=0cm,draw,circle] (72) at (-2.5,22.5) {$ 1$};
		\node [minimum size=0cm,draw,circle] (73) at (-3.5,22.5) {$ 2$};
		\node [minimum size=0cm,draw,circle] (74) at (-4.5,22.5) {$ 3$};
		\node [minimum size=0cm,draw,circle] (75) at (-5.5,22.5) {$ 4$};
    \node [minimum size=0cm,draw,circle] (76) at (-6.5,22.5) {$ 5$};
		\node [minimum size=0cm,draw,circle] (77) at (-7.5,22.5) {$ 0$};
		
		\node [minimum size=0cm,draw,circle] (78) at (-2,23) {$ 2$};
		\node [minimum size=0cm,draw,circle] (79) at (-3,23) {$ 3$};
		\node [minimum size=0cm,draw,circle] (80) at (-4,23) {$ 4$};
		\node [minimum size=0cm,draw,circle] (81) at (-5,23) {$ 5$};
		\node [minimum size=0cm,draw,circle] (82) at (-6,23) {$ 0$};
  	\node [minimum size=0cm,draw,circle] (83) at (-7,23) {$ 1$};
		\node [minimum size=0cm,draw,circle] (84) at (-8,23) {$ 2$};
			
		\node [minimum size=0cm,draw,circle] (85) at (-2.5,23.5) {$ 5$};
		\node [minimum size=0cm,draw,circle] (86) at (-3.5,23.5) {$ 0$};
		\node [minimum size=0cm,draw,circle] (87) at (-4.5,23.5) {$ 1$};
		\node [minimum size=0cm,draw,circle] (88) at (-5.5,23.5) {$ 2$};
		\node [minimum size=0cm,draw,circle] (89) at (-6.5,23.5) {$ 3$};
		\node [minimum size=0cm,draw,circle] (90) at (-7.5,23.5) {$ 4$};
		
	  \node [minimum size=0cm,draw,circle] (91) at (-2,24) {$ 0$};
		\node [minimum size=0cm,draw,circle] (92) at (-3,24) {$ 1$};
		\node [minimum size=0cm,draw,circle] (93) at (-4,24) {$ 2$};
		\node [minimum size=0cm,draw,circle] (94) at (-5,24) {$ 3$};
		\node [minimum size=0cm,draw,circle] (95) at (-6,24) {$ 4$};
		\node [minimum size=0cm,draw,circle] (96) at (-7,24) {$ 5$};
		\node [minimum size=0cm,draw,circle] (97) at (-8,24) {$ 0$};	
\end{pgfonlayer}
	\begin{pgfonlayer}{edgelayer}

   \draw [thin=1.00] (0) to (7);
   \draw [thin=1.00] (1) to (7);
   \draw [thin=1.00] (1) to (8);
   \draw [thin=1.00] (2) to (8);
	  \draw [thin=1.00] (2) to (9);
		\draw [thin=1.00] (3) to (9);
		\draw [thin=1.00] (3) to (10);
		\draw [thin=1.00] (4) to (10);
		\draw [thin=1.00] (4) to (11);
		\draw [thin=1.00] (5) to (11);
		\draw [thin=1.00] (5) to (12);
		\draw [thin=1.00] (6) to (12);

		\draw [thin=1.00] (7) to (13);
		\draw [thin=1.00] (7) to (14);
		\draw [thin=1.00] (8) to (14);
		\draw [thin=1.00] (8) to (15);
		\draw [thin=1.00] (9) to (15);
		\draw [thin=1.00] (9) to (16);
   \draw [thin=1.00] (10) to (16);
   \draw [thin=1.00] (10) to (17);
   \draw [thin=1.00] (11) to (17);
		\draw [thin=1.00] (11) to (18);
		\draw [thin=1.00] (12) to (18);
		\draw [thin=1.00] (12) to (19);

		\draw [thin=1.00] (13) to (20);
		\draw [thin=1.00] (14) to (20);
		\draw [thin=1.00] (14) to (21);
		\draw [thin=1.00] (15) to (21);
		\draw [thin=1.00] (15) to (22);
		\draw [thin=1.00] (16) to (22);
		\draw [thin=1.00] (16) to (23);
		\draw [thin=1.00] (17) to (23);
		\draw [thin=1.00] (17) to (24);
		\draw [thin=1.00] (18) to (24);
		\draw [thin=1.00] (18) to (25);
		\draw [thin=1.00] (19) to (25);
		
   \draw [thin=1.00] (20) to (26);
   \draw [thin=1.00] (20) to (27);
   \draw [thin=1.00] (21) to (27);
   \draw [thin=1.00] (21) to (28);
   \draw [thin=1.00] (22) to (28);
		\draw [thin=1.00] (22) to (29);
		\draw [thin=1.00] (23) to (29);
	  \draw [thin=1.00] (23) to (30);
	  \draw [thin=1.00] (24) to (30);
		\draw [thin=1.00] (24) to (31);
		\draw [thin=1.00] (25) to (31);
		\draw [thin=1.00] (25) to (32);

		\draw [thin=1.00] (26) to (33);
		\draw [thin=1.00] (27) to (33);
		\draw [thin=1.00] (27) to (34);
		\draw [thin=1.00] (28) to (34);
		\draw [thin=1.00] (28) to (35);
		\draw [thin=1.00] (29) to (35);
		\draw [thin=1.00] (29) to (36);
		\draw [thin=1.00] (30) to (36);
		\draw [thin=1.00] (30) to (37);
		\draw [thin=1.00] (31) to (37);
		\draw [thin=1.00] (31) to (38);
		\draw [thin=1.00] (32) to (38);

    \draw [thin=1.00] (33) to (39);
    \draw [thin=1.00] (33) to (40);
    \draw [thin=1.00] (34) to (40);
	  \draw [thin=1.00] (34) to (41);
		\draw [thin=1.00] (35) to (41);
		\draw [thin=1.00] (35) to (42);
		\draw [thin=1.00] (36) to (42);
		\draw [thin=1.00] (36) to (43);
		\draw [thin=1.00] (37) to (43);
		\draw [thin=1.00] (37) to (44);
		\draw [thin=1.00] (38) to (44);
		\draw [thin=1.00] (38) to (45);

		\draw [thin=1.00] (39) to (46);
		\draw [thin=1.00] (40) to (46);
		\draw [thin=1.00] (40) to (47);
		\draw [thin=1.00] (41) to (47);
		\draw [thin=1.00] (41) to (48);
   \draw [thin=1.00] (42) to (48);
   \draw [thin=1.00] (42) to (49);
   \draw [thin=1.00] (43) to (49);
   \draw [thin=1.00] (43) to (50);
		\draw [thin=1.00] (44) to (50);
		\draw [thin=1.00] (44) to (51);
		\draw [thin=1.00] (45) to (51);

		\draw [thin=1.00] (46) to (52);
		\draw [thin=1.00] (46) to (53);
		\draw [thin=1.00] (47) to (53);
		\draw [thin=1.00] (47) to (54);
		\draw [thin=1.00] (48) to (54);
		\draw [thin=1.00] (48) to (55);
		\draw [thin=1.00] (49) to (55);
		\draw [thin=1.00] (49) to (56);
		\draw [thin=1.00] (50) to (56);
		\draw [thin=1.00] (50) to (57);
		\draw [thin=1.00] (51) to (57);
		\draw [thin=1.00] (51) to (58);

   \draw [thin=1.00] (52) to (59);
   \draw [thin=1.00] (53) to (59);
   \draw [thin=1.00] (53) to (60);
   \draw [thin=1.00] (54) to (60);
		\draw [thin=1.00] (54) to (61);
		\draw [thin=1.00] (55) to (61);
	  \draw [thin=1.00] (55) to (62);
		\draw [thin=1.00] (56) to (62);
		\draw [thin=1.00] (56) to (63);
		\draw [thin=1.00] (57) to (63);
		\draw [thin=1.00] (57) to (64);
		\draw [thin=1.00] (58) to (64);

		\draw [thin=1.00] (59) to (65);
		\draw [thin=1.00] (59) to (66);
	  \draw [thin=1.00] (60) to (66);
		\draw [thin=1.00] (60) to (67);		
		\draw [thin=1.00] (61) to (67);
		\draw [thin=1.00] (61) to (68);
		\draw [thin=1.00] (62) to (68);		
		\draw [thin=1.00] (62) to (69);
		\draw [thin=1.00] (63) to (69);
		\draw [thin=1.00] (63) to (70);
		\draw [thin=1.00] (64) to (70);
		\draw [thin=1.00] (64) to (71);

   	\draw [thin=1.00] (65) to (72);
		\draw [thin=1.00] (66) to (72);
		\draw [thin=1.00] (66) to (73);
		\draw [thin=1.00] (67) to (73);
		\draw [thin=1.00] (67) to (74);
   \draw [thin=1.00] (68) to (74);
   \draw [thin=1.00] (68) to (75);
   \draw [thin=1.00] (69) to (75);
   \draw [thin=1.00] (69) to (76);
		\draw [thin=1.00] (70) to (76);
		\draw [thin=1.00] (70) to (77);
		\draw [thin=1.00] (71) to (77);

		\draw [thin=1.00] (72) to (78);
		\draw [thin=1.00] (72) to (79);
		\draw [thin=1.00] (73) to (79);
		\draw [thin=1.00] (73) to (80);
		\draw [thin=1.00] (74) to (80);
		\draw [thin=1.00] (74) to (81);
		\draw [thin=1.00] (75) to (81);
		\draw [thin=1.00] (75) to (82);
		\draw [thin=1.00] (76) to (82);
		\draw [thin=1.00] (76) to (83);
		\draw [thin=1.00] (77) to (83);
		\draw [thin=1.00] (77) to (84);

   \draw [thin=1.00] (78) to (85);
   \draw [thin=1.00] (79) to (85);
   \draw [thin=1.00] (79) to (86);
   \draw [thin=1.00] (80) to (86);
		\draw [thin=1.00] (80) to (87);
		\draw [thin=1.00] (81) to (87);
	  \draw [thin=1.00] (81) to (88);
		\draw [thin=1.00] (82) to (88);
		\draw [thin=1.00] (82) to (89);
		\draw [thin=1.00] (83) to (89);
		\draw [thin=1.00] (83) to (90);
		\draw [thin=1.00] (84) to (90);

		\draw [thin=1.00] (85) to (91);
		\draw [thin=1.00] (85) to (92);
	  \draw [thin=1.00] (86) to (92);
		\draw [thin=1.00] (86) to (93);		
		\draw [thin=1.00] (87) to (93);
		\draw [thin=1.00] (87) to (94);
		\draw [thin=1.00] (88) to (94);		
		\draw [thin=1.00] (88) to (95);
		\draw [thin=1.00] (89) to (95);
		\draw [thin=1.00] (89) to (96);
		\draw [thin=1.00] (90) to (96);
		\draw [thin=1.00] (90) to (97);
			\end{pgfonlayer} {edgelayer}
\end{tikzpicture}
\end{center}
}}

Note that $\cup_{i=0}^4\left\{V_i\right\}$ in Figure 12 above  forms a component of a $C_4 \times C_{12}$ and $L(\cup_{i=0}^4\left\{V_i\right\})=[5]$. Therefore there exist an independent $5-L(1,1)$-labeling of $C_4 \times C_{12}$ in the $L(1,1)$-labeling of $C_{14} \times C_{12}$.
Thus $\lambda_1^1(C_{10+4m'} \times C_{12})=5$, $m'\in \mathbb N$. Since for all $p\in \mathbb N$, $p$ can be expressed as $10+4m'$, then the required result holds.
\end{proof}

 \begin{cor}
     For all $m\equiv 0\; mod \; 14$, $n\equiv 0 \; mod\; 12$, $m,n \not\equiv 0\; mod \; 5,$ $\lambda_1^1(C_m \times C_n)=5$.
 \end{cor}
 The next result establishes an optimal $L(1,1)$-labeling of $C_m \times C_n$ of a certain size. This resolves all cases of large enough $m$ and $n$.

\begin{thm}For $m',m'', n', n'' \in \mathbb{Z}_+$ $\lambda_1^1(C_{10m'+14m''} \times C_{10n'+12n''})= 5$.
%$m \notin \left\{16,18,22,26,32,36,46\right\}$, $n \notin %\left\{26,28,38\right\}$, $\lambda^1_1(C_m \times C_n)\leq 5$.
 \end{thm}

 \begin{proof}
From earlier results, $\lambda_1^1(C_m \times C_n) \geq 5$ for $C_m \times C_n$ defined in the statement above. Each of the quadrant in Figure $13$ represents special $5-L(1,1)$-labelings of $C_{10} \times C_{10}$, $C_{10} \times C_{12}$, $C_{10} \times C_{14}$ and $C_{12} \times C_{14}$ respectively. Clearly, these labelings form a $5-L(1,1)$-labeling of $C_{10+14} \times C_{10+12}$. Thus, for $m',m'', n', n'' \in \mathbb{Z}_+$, $\lambda_1^1(C_{10m'+14m''} \times C_{10n'+12n''})\leq 5$.

 \newpage
{\tiny{
%\begin{center}
\pgfdeclarelayer{nodelayer}
\pgfdeclarelayer{edgelayer}
\pgfsetlayers{nodelayer,edgelayer}
\begin{tikzpicture}
%\centering
	\begin{pgfonlayer}{nodelayer}

	  \node [minimum size=0cm,] (0e) at (-7,15.25) {};
		\node [minimum size=0cm,draw,circle] (0) at (-2,16) {$ 0$};
		\node [minimum size=0cm,draw,circle] (1) at (-3,16) {$ 1$};
		\node [minimum size=0cm,draw,circle] (2) at (-4,16) {$ 2$};
		\node [minimum size=0cm,draw,circle] (3) at (-5,16) {$ 4$};
		\node [minimum size=0cm,draw,circle] (4) at (-6,16) {$ 5$};
		\node [minimum size=0cm,draw,circle] (5) at (-7,16) {$ 0$};
		
		\node [minimum size=0cm,draw,circle] (6) at (-2.5,16.5) {$ 3$};
		\node [minimum size=0cm,draw,circle] (7) at (-3.5,16.5) {$ 4$};
		\node [minimum size=0cm,draw,circle] (8) at (-4.5,16.5) {$ 5$};
		\node [minimum size=0cm,draw,circle] (9) at (-5.5,16.5) {$ 0$};
	\node [minimum size=0cm,draw,circle] (10) at (-6.5,16.5) {$ 1$};
	
		\node [minimum size=0cm,draw,circle] (11) at (-2,17) {$ 4$};
		\node [minimum size=0cm,draw,circle] (12) at (-3,17) {$ 5$};
		\node [minimum size=0cm,draw,circle] (13) at (-4,17) {$ 0$};
		\node [minimum size=0cm,draw,circle] (14) at (-5,17) {$ 1$};
		\node [minimum size=0cm,draw,circle] (15) at (-6,17) {$ 3$};
		\node [minimum size=0cm,draw,circle] (16) at (-7,17) {$ 4$};
		
		\node [minimum size=0cm,draw,circle] (17) at (-2.5,17.5) {$ 0$};
		\node [minimum size=0cm,draw,circle] (18) at (-3.5,17.5) {$ 2$};
		\node [minimum size=0cm,draw,circle] (19) at (-4.5,17.5) {$ 3$};
		\node [minimum size=0cm,draw,circle] (20) at (-5.5,17.5) {$ 4$};
		\node [minimum size=0cm,draw,circle] (21) at (-6.5,17.5) {$ 5$};
		
    \node [minimum size=0cm,draw,circle] (22) at (-2,18) {$ 2$};
		\node [minimum size=0cm,draw,circle] (23) at (-3,18) {$ 3$};
		\node [minimum size=0cm,draw,circle] (24) at (-4,18) {$ 4$};
		\node [minimum size=0cm,draw,circle] (25) at (-5,18) {$ 5$};
		\node [minimum size=0cm,draw,circle] (26) at (-6,18) {$ 0$};
		\node [minimum size=0cm,draw,circle] (27) at (-7,18) {$ 2$};
		
		\node [minimum size=0cm,draw,circle] (28) at (-2.5,18.5) {$ 4$};
		\node [minimum size=0cm,draw,circle] (29) at (-3.5,18.5) {$ 5$};
		\node [minimum size=0cm,draw,circle] (30) at (-4.5,18.5) {$ 0$};
		\node [minimum size=0cm,draw,circle] (31) at (-5.5,18.5) {$ 2$};
		\node [minimum size=0cm,draw,circle] (32) at (-6.5,18.5) {$ 3$};
		
		\node [minimum size=0cm,draw,circle] (33) at (-2,19) {$ 5$};
		\node [minimum size=0cm,draw,circle] (34) at (-3,19) {$ 0$};
		\node [minimum size=0cm,draw,circle] (35) at (-4,19) {$ 2$};
		\node [minimum size=0cm,draw,circle] (36) at (-5,19) {$ 3$};
		\node [minimum size=0cm,draw,circle] (37) at (-6,19) {$ 4$};
		\node [minimum size=0cm,draw,circle] (38) at (-7,19) {$ 5$};
		
		\node [minimum size=0cm,draw,circle] (39) at (-2.5,19.5) {$ 2$};
		\node [minimum size=0cm,draw,circle] (40) at (-3.5,19.5) {$ 3$};
		\node [minimum size=0cm,draw,circle] (41) at (-4.5,19.5) {$ 4$};
		\node [minimum size=0cm,draw,circle] (42) at (-5.5,19.5) {$ 5$};
		\node [minimum size=0cm,draw,circle] (43) at (-6.5,19.5) {$ 0$};
		
		\node [minimum size=0cm,draw,circle] (44) at (-2,20) {$ 3$};
		\node [minimum size=0cm,draw,circle] (45) at (-3,20) {$ 4$};
		\node [minimum size=0cm,draw,circle] (46) at (-4,20) {$ 5$};
		\node [minimum size=0cm,draw,circle] (47) at (-5,20) {$ 0$};
		\node [minimum size=0cm,draw,circle] (48) at (-6,20) {$ 2$};
		\node [minimum size=0cm,draw,circle] (49) at (-7,20) {$ 3$};
		
		\node [minimum size=0cm,draw,circle] (50) at (-2.5,20.5) {$ 5$};
		\node [minimum size=0cm,draw,circle] (51) at (-3.5,20.5) {$ 0$};
		\node [minimum size=0cm,draw,circle] (52) at (-4.5,20.5) {$ 1$};
		\node [minimum size=0cm,draw,circle] (53) at (-5.5,20.5) {$ 3$};
		\node [minimum size=0cm,draw,circle] (54) at (-6.5,20.5) {$ 4$};
		
	 \node [minimum size=0cm,draw,circle] (55') at (-2,21) {$ 0$};
		\node [minimum size=0cm,draw,circle] (56') at (-3,21) {$ 1$};
		\node [minimum size=0cm,draw,circle] (57') at (-4,21) {$ 2$};
		\node [minimum size=0cm,draw,circle] (58') at (-5,21) {$ 4$};
		\node [minimum size=0cm,draw,circle] (59') at (-6,21) {$ 5$};
		\node [minimum size=0cm,draw,circle] (60') at (-7,21) {$ 0$};
		\node [minimum size=0cm,] (0g) at (-1,21) {};
		\node [minimum size=0cm,] (0gg) at (-4.5,15.5) {$C_{10} \times C_{10}$};
		
		 \node [minimum size=0cm,draw,circle] (55a) at (-2,14.5) {$ 0$};
		\node [minimum size=0cm,draw,circle] (56a) at (-3,14.5) {$ 1$};
		\node [minimum size=0cm,draw,circle] (57a) at (-4,14.5) {$ 2$};
		\node [minimum size=0cm,draw,circle] (58a) at (-5,14.5) {$ 4$};
		\node [minimum size=0cm,draw,circle] (59a) at (-6,14.5) {$ 5$};
		\node [minimum size=0cm,draw,circle] (60a) at (-7,14.5) {$ 0$};
		
		\node [minimum size=0cm,draw,circle] (61a) at (-2.5,14) {$ 5$};
		\node [minimum size=0cm,draw,circle] (62a) at (-3.5,14) {$ 0$};
		\node [minimum size=0cm,draw,circle] (63a) at (-4.5,14) {$ 1$};
		\node [minimum size=0cm,draw,circle] (64a) at (-5.5,14) {$ 2$};
		\node [minimum size=0cm,draw,circle] (65a) at (-6.5,14) {$ 4$};
		
		\node [minimum size=0cm,draw,circle] (66a) at (-2,13.5) {$ 2$};
		\node [minimum size=0cm,draw,circle] (67a) at (-3,13.5) {$ 4$};
		\node [minimum size=0cm,draw,circle] (68a) at (-4,13.5) {$ 5$};
		\node [minimum size=0cm,draw,circle] (69a) at (-5,13.5) {$ 0$};
		\node [minimum size=0cm,draw,circle] (70a) at (-6,13.5) {$ 1$};
		\node [minimum size=0cm,draw,circle] (71a) at (-7,13.5) {$ 2$};
		
		\node [minimum size=0cm,draw,circle] (72a) at (-2.5,13) {$ 1$};
		\node [minimum size=0cm,draw,circle] (73a) at (-3.5,13) {$ 2$};
		\node [minimum size=0cm,draw,circle] (74a) at (-4.5,13) {$ 4$};
		\node [minimum size=0cm,draw,circle] (75a) at (-5.5,13) {$ 3$};
		\node [minimum size=0cm,draw,circle] (76a) at (-6.5,13) {$ 0$};
		
		\node [minimum size=0cm,draw,circle] (77a) at (-2,12.5) {$ 5$};
		\node [minimum size=0cm,draw,circle] (78a) at (-3,12.5) {$ 0$};
		\node [minimum size=0cm,draw,circle] (79a) at (-4,12.5) {$ 1$};
		\node [minimum size=0cm,draw,circle] (80a) at (-5,12.5) {$ 2$};
		\node [minimum size=0cm,draw,circle] (81a) at (-6,12.5) {$ 4$};
		\node [minimum size=0cm,draw,circle] (82a) at (-7,12.5) {$ 5$};
		
		\node [minimum size=0cm,draw,circle] (83a) at (-2.5,12) {$ 4$};
		\node [minimum size=0cm,draw,circle] (84a) at (-3.5,12) {$ 5$};
		\node [minimum size=0cm,draw,circle] (85a) at (-4.5,12) {$ 0$};
		\node [minimum size=0cm,draw,circle] (86a) at (-5.5,12) {$ 1$};
		\node [minimum size=0cm,draw,circle] (87a) at (-6.5,12) {$ 2$};
		
		\node [minimum size=0cm,draw,circle] (88a) at (-2,11.5) {$ 1$};
		\node [minimum size=0cm,draw,circle] (89a) at (-3,11.5) {$ 2$};
		\node [minimum size=0cm,draw,circle] (90a) at (-4,11.5) {$ 4$};
		\node [minimum size=0cm,draw,circle] (91a) at (-5,11.5) {$ 5$};
		\node [minimum size=0cm,draw,circle] (92a) at (-6,11.5) {$ 0$};
		\node [minimum size=0cm,draw,circle] (93a) at (-7,11.5) {$ 1$};
		
		\node [minimum size=0cm,draw,circle] (94a) at (-2.5,11) {$ 0$};
		\node [minimum size=0cm,draw,circle] (95a) at (-3.5,11) {$ 1$};
		\node [minimum size=0cm,draw,circle] (96a) at (-4.5,11) {$ 2$};
		\node [minimum size=0cm,draw,circle] (97a) at (-5.5,11) {$ 4$};
		\node [minimum size=0cm,draw,circle] (98a) at (-6.5,11) {$ 5$};
		
		\node [minimum size=0cm,draw,circle] (99a) at (-2,10.5)  {$ 4$};
		\node [minimum size=0cm,draw,circle] (100a) at (-3,10.5) {$ 5$};
		\node [minimum size=0cm,draw,circle] (101a) at (-4,10.5) {$ 0$};
		\node [minimum size=0cm,draw,circle] (102a) at (-5,10.5) {$ 1$};
		\node [minimum size=0cm,draw,circle] (103a) at (-6,10.5) {$ 2$};
		\node [minimum size=0cm,draw,circle] (104a) at (-7,10.5) {$ 4$};
		
		\node [minimum size=0cm,draw,circle] (105a) at (-2.5,10) {$ 3$};
		\node [minimum size=0cm,draw,circle] (106a) at (-3.5,10) {$ 4$};
		\node [minimum size=0cm,draw,circle] (107a) at (-4.5,10) {$ 5$};
		\node [minimum size=0cm,draw,circle] (108a) at (-5.5,10) {$ 0$};
		\node [minimum size=0cm,draw,circle] (109a) at (-6.5,10) {$ 1$};
		
    \node [minimum size=0cm,draw,circle] (110a) at (-2,9.5) {$ 0$};
		\node [minimum size=0cm,draw,circle] (111a) at (-3,9.5) {$ 2$};
		\node [minimum size=0cm,draw,circle] (112a) at (-4,9.5) {$ 3$};
		\node [minimum size=0cm,draw,circle] (113a) at (-5,9.5) {$ 4$};
		\node [minimum size=0cm,draw,circle] (114a) at (-6,9.5) {$ 5$};
		\node [minimum size=0cm,draw,circle] (115a) at (-7,9.5) {$ 0$};
		
		\node [minimum size=0cm,draw,circle] (116a) at (-2.5,9) {$ 5$};
		\node [minimum size=0cm,draw,circle] (117a) at (-3.5,9) {$ 1$};
		\node [minimum size=0cm,draw,circle] (118a) at (-4.5,9) {$ 2$};
		\node [minimum size=0cm,draw,circle] (119a) at (-5.5,9) {$ 3$};
		\node [minimum size=0cm,draw,circle] (120a) at (-6.5,9) {$ 4$};
		
		\node [minimum size=0cm,draw,circle] (121a) at (-2,8.5) {$ 3$};
		\node [minimum size=0cm,draw,circle] (122a) at (-3,8.5) {$ 4$};
		\node [minimum size=0cm,draw,circle] (123a) at (-4,8.5) {$ 0$};
		\node [minimum size=0cm,draw,circle] (124a) at (-5,8.5) {$ 1$};
		\node [minimum size=0cm,draw,circle] (125a) at (-6,8.5) {$ 2$};
		\node [minimum size=0cm,draw,circle] (126a) at (-7,8.5) {$ 3$};
		
		\node [minimum size=0cm,draw,circle] (127a) at (-2.5,8) {$ 2$};
		\node [minimum size=0cm,draw,circle] (128a) at (-3.5,8) {$ 3$};
		\node [minimum size=0cm,draw,circle] (129a) at (-4.5,8) {$ 5$};
		\node [minimum size=0cm,draw,circle] (130a) at (-5.5,8) {$ 0$};
		\node [minimum size=0cm,draw,circle] (131a) at (-6.5,8) {$ 1$};
		
		\node [minimum size=0cm,] (0d) at (-1,7.5) {};
		\node [minimum size=0cm,draw,circle] (132a) at (-2,7.5) {$0$};
		\node [minimum size=0cm,draw,circle] (133a) at (-3,7.5) {$ 1$};
		\node [minimum size=0cm,draw,circle] (134a) at (-4,7.5) {$ 2$};
		\node [minimum size=0cm,draw,circle] (135a) at (-5,7.5) {$ 4$};
		\node [minimum size=0cm,draw,circle] (136a) at (-6,7.5) {$ 5$};
		\node [minimum size=0cm,draw,circle] (137a) at (-7,7.5) {$ 0$};

%\node [minimum size=0cm,] at (2,15.5) {$Figure \; 6$: $5-L(1,1)$-Labeling of $C_{10} \times C_{12}$};
    \node [minimum size=0cm,draw,circle] (55) at (0,16) {$ 0$};
		\node [minimum size=0cm,draw,circle] (56) at (1,16) {$ 5$};
		\node [minimum size=0cm,draw,circle] (57) at (2,16) {$ 4$};
		\node [minimum size=0cm,draw,circle] (58) at (3,16) {$ 3$};
		\node [minimum size=0cm,draw,circle] (59) at (4,16) {$ 2$};
		\node [minimum size=0cm,draw,circle] (60) at (5,16) {$ 1$};
		\node [minimum size=0cm,draw,circle] (60b) at (6,16) {$ 0$};
		
		\node [minimum size=0cm,draw,circle] (61) at (0.5,16.5) {$ 2$};
		\node [minimum size=0cm,draw,circle] (62) at (1.5,16.5) {$ 1$};
		\node [minimum size=0cm,draw,circle] (63) at (2.5,16.5) {$ 0$};
		\node [minimum size=0cm,draw,circle] (64) at (3.5,16.5) {$ 5$};
		\node [minimum size=0cm,draw,circle] (65) at (4.5,16.5) {$ 4$};
		\node [minimum size=0cm,draw,circle] (65b) at (5.5,16.5) {$ 3$};
		
		\node [minimum size=0cm,draw,circle] (66) at (0,17) {$ 4$};
		\node [minimum size=0cm,draw,circle] (67) at (1,17) {$ 3$};
		\node [minimum size=0cm,draw,circle] (68) at (2,17) {$ 2$};
		\node [minimum size=0cm,draw,circle] (69) at (3,17) {$ 1$};
		\node [minimum size=0cm,draw,circle] (70) at (4,17) {$ 0$};
		\node [minimum size=0cm,draw,circle] (71) at (5,17) {$ 5$};
		\node [minimum size=0cm,draw,circle] (71b) at (6,17){$ 4$};
		
		\node [minimum size=0cm,draw,circle] (72) at (0.5,17.5) {$ 5$};
		\node [minimum size=0cm,draw,circle] (73) at (1.5,17.5) {$ 4$};
		\node [minimum size=0cm,draw,circle] (74) at (2.5,17.5) {$ 3$};
		\node [minimum size=0cm,draw,circle] (75) at (3.5,17.5) {$ 2$};
		\node [minimum size=0cm,draw,circle] (76) at (4.5,17.5) {$ 1$};
		\node [minimum size=0cm,draw,circle] (76b) at (5.5,17.5) {$ 0$};
		
		\node [minimum size=0cm,draw,circle] (77) at (0,18) {$ 2$};
		\node [minimum size=0cm,draw,circle] (78) at (1,18) {$ 1$};
		\node [minimum size=0cm,draw,circle] (79) at (2,18) {$ 0$};
		\node [minimum size=0cm,draw,circle] (80) at (3,18) {$ 5$};
		\node [minimum size=0cm,draw,circle] (81) at (4,18) {$ 4$};
		\node [minimum size=0cm,draw,circle] (82) at (5,18) {$ 3$};
		\node [minimum size=0cm,draw,circle] (82b) at (6,18){$ 2$};
		
		\node [minimum size=0cm,draw,circle] (83) at (0.5,18.5) {$ 3$};
		\node [minimum size=0cm,draw,circle] (84) at (1.5,18.5) {$ 2$};
		\node [minimum size=0cm,draw,circle] (85) at (2.5,18.5) {$ 1$};
		\node [minimum size=0cm,draw,circle] (86) at (3.5,18.5) {$ 0$};
		\node [minimum size=0cm,draw,circle] (87) at (4.5,18.5) {$ 5$};
		\node [minimum size=0cm,draw,circle] (87b) at (5.5,18.5) {$ 4$};
		
		\node [minimum size=0cm,draw,circle] (88) at (0,19) {$ 5$};
		\node [minimum size=0cm,draw,circle] (89) at (1,19) {$ 4$};
		\node [minimum size=0cm,draw,circle] (90) at (2,19) {$ 3$};
		\node [minimum size=0cm,draw,circle] (91) at (3,19) {$ 2$};
		\node [minimum size=0cm,draw,circle] (92) at (4,19) {$ 1$};
		\node [minimum size=0cm,draw,circle] (93) at (5,19) {$ 0$};
		\node [minimum size=0cm,draw,circle] (93b) at (6,19) {$ 5$};
		
		\node [minimum size=0cm,draw,circle] (94) at (0.5,19.5) {$ 1$};
		\node [minimum size=0cm,draw,circle] (95) at (1.5,19.5) {$ 0$};
		\node [minimum size=0cm,draw,circle] (96) at (2.5,19.5) {$ 5$};
		\node [minimum size=0cm,draw,circle] (97) at (3.5,19.5) {$ 4$};
		\node [minimum size=0cm,draw,circle] (98) at (4.5,19.5) {$ 3$};
		\node [minimum size=0cm,draw,circle] (98b) at (5.5,19.5) {$ 2$};
		
		\node [minimum size=0cm,draw,circle] (99) at (0,20) {$ 3$};
		\node [minimum size=0cm,draw,circle] (100) at (1,20) {$ 2$};
		\node [minimum size=0cm,draw,circle] (101) at (2,20) {$ 1$};
		\node [minimum size=0cm,draw,circle] (102) at (3,20) {$ 0$};
		\node [minimum size=0cm,draw,circle] (103) at (4,20) {$ 5$};
		\node [minimum size=0cm,draw,circle] (104) at (5,20) {$ 4$};
		\node [minimum size=0cm,draw,circle] (104b) at (6,20) {$ 3$};
		
		\node [minimum size=0cm,draw,circle] (105) at (0.5,20.5) {$ 4$};
		\node [minimum size=0cm,draw,circle] (106) at (1.5,20.5) {$ 3$};
		\node [minimum size=0cm,draw,circle] (107) at (2.5,20.5) {$ 2$};
		\node [minimum size=0cm,draw,circle] (108) at (3.5,20.5) {$ 1$};
		\node [minimum size=0cm,draw,circle] (109) at (4.5,20.5) {$ 0$};
		\node [minimum size=0cm,draw,circle] (109b) at (5.5,20.5) {$ 5$};
		
   \node [minimum size=0cm,draw,circle] (110) at (0,21) {$ 0$};
	 \node [minimum size=0cm,draw,circle] (111) at (1,21) {$ 5$};
	 \node [minimum size=0cm,draw,circle] (112) at (2,21) {$ 4$};
	 \node [minimum size=0cm,draw,circle] (113) at (3,21) {$ 3$};
	 \node [minimum size=0cm,draw,circle] (114) at (4,21) {$ 2$};
	 \node [minimum size=0cm,draw,circle] (115) at (5,21) {$ 1$};
	 \node [minimum size=0cm,draw,circle] (115b) at (6,21) {$ 0$};
		\node [minimum size=0cm,] (115cc) at (3,15.5) {$C_{10}\times C_{12}$};
			
		\node [minimum size=0cm,] (0f) at (6,15.25) {};
		\node [minimum size=0cm,draw,circle] (0c) at (0,14.5) {$ 0$};
		\node [minimum size=0cm,draw,circle] (1c) at (1,14.5) {$ 5$};
		\node [minimum size=0cm,draw,circle] (2c) at (2,14.5) {$ 4$};
		\node [minimum size=0cm,draw,circle] (3c) at (3,14.5) {$ 3$};
		\node [minimum size=0cm,draw,circle] (4c) at (4,14.5) {$ 2$};
	  \node [minimum size=0cm,draw,circle] (5c) at (5,14.5) {$ 1$};
		\node [minimum size=0cm,draw,circle] (6c) at (6,14.5) {$ 0$};
		
		\node [minimum size=0cm,draw,circle] (7c) at (0.5,14) {$ 4$};
		\node [minimum size=0cm,draw,circle] (8c) at (1.5,14) {$ 3$};
		\node [minimum size=0cm,draw,circle] (9c) at (2.5,14) {$ 2$};
		\node [minimum size=0cm,draw,circle] (10c) at (3.5,14) {$ 1$};
    \node [minimum size=0cm,draw,circle] (11c) at (4.5,14) {$ 0$};
		\node [minimum size=0cm,draw,circle] (12c) at (5.5,14){$ 5$};
		
		\node [minimum size=0cm,draw,circle] (13c) at (0,13.5) {$ 2$};
		\node [minimum size=0cm,draw,circle] (14c) at (1,13.5) {$ 1$};
		\node [minimum size=0cm,draw,circle] (15c) at (2,13.5) {$ 0$};
		\node [minimum size=0cm,draw,circle] (16c) at (3,13.5) {$ 5$};
		\node [minimum size=0cm,draw,circle] (17c) at (4,13.5) {$ 4$};
	  \node [minimum size=0cm,draw,circle] (18c) at (5,13.5) {$ 3$};
		\node [minimum size=0cm,draw,circle] (19c) at (6,13.5) {$ 2$};
		
		\node [minimum size=0cm,draw,circle] (20c) at (0.5,13) {$0 $};
		\node [minimum size=0cm,draw,circle] (21c) at (1.5,13) {$ 5$};
		\node [minimum size=0cm,draw,circle] (22c) at (2.5,13) {$ 4$};
		\node [minimum size=0cm,draw,circle] (23c) at (3.5,13) {$ 3$};
  	\node [minimum size=0cm,draw,circle] (24c) at (4.5,13) {$ 2$};
		\node [minimum size=0cm,draw,circle] (25c) at (5.5,13){$ $1};
			
    \node [minimum size=0cm,draw,circle] (26c) at (0,12.5) {$ 5$};
		\node [minimum size=0cm,draw,circle] (27c) at (1,12.5) {$ 4$};
		\node [minimum size=0cm,draw,circle] (28c) at (2,12.5) {$ 3$};
		\node [minimum size=0cm,draw,circle] (29c) at (3,12.5) {$ 2$};
		\node [minimum size=0cm,draw,circle] (30c) at (4,12.5) {$ 1$};
		\node [minimum size=0cm,draw,circle] (31c) at (5,12.5) {$ 0$};
		\node [minimum size=0cm,draw,circle] (32c) at (6,12.5) {$ 5$};
		
		\node [minimum size=0cm,draw,circle] (33c) at (0.5,12) {$ 3$};
		\node [minimum size=0cm,draw,circle] (34c) at (1.5,12) {$ 2$};
		\node [minimum size=0cm,draw,circle] (35c) at (2.5,12) {$ 1$};
		\node [minimum size=0cm,draw,circle] (36c) at (3.5,12) {$ 0$};
		\node [minimum size=0cm,draw,circle] (37c) at (4.5,12) {$ 5$};
		\node [minimum size=0cm,draw,circle] (38c) at (5.5,12) {$ 4$};

		\node [minimum size=0cm,draw,circle] (39c) at (0,11.5) {$ 1$};
		\node [minimum size=0cm,draw,circle] (40c) at (1,11.5) {$ 0$};
		\node [minimum size=0cm,draw,circle] (41c) at (2,11.5) {$ 5$};
		\node [minimum size=0cm,draw,circle] (42c) at (3,11.5) {$ 4$};
		\node [minimum size=0cm,draw,circle] (43c) at (4,11.5) {$ 3$};
		\node [minimum size=0cm,draw,circle] (44c) at (5,11.5) {$ 2$};
		\node [minimum size=0cm,draw,circle] (45c) at (6,11.5) {$ 1$};
		
		\node [minimum size=0cm,draw,circle] (46c) at (0.5,11) {$ 5$};
		\node [minimum size=0cm,draw,circle] (47c) at (1.5,11) {$ 4$};
		\node [minimum size=0cm,draw,circle] (48c) at (2.5,11) {$ 3$};
		\node [minimum size=0cm,draw,circle] (49c) at (3.5,11) {$ 2$};
    \node [minimum size=0cm,draw,circle] (50c) at (4.5,11) {$ 1$};
		\node [minimum size=0cm,draw,circle] (51c) at (5.5,11) {$ 0$};
		
		\node [minimum size=0cm,draw,circle] (52c) at (0,10.5) {$ 4$};
		\node [minimum size=0cm,draw,circle] (53c) at (1,10.5) {$ 3$};
		\node [minimum size=0cm,draw,circle] (54c) at (2,10.5) {$ 2$};
		\node [minimum size=0cm,draw,circle] (55c) at (3,10.5) {$ 1$};
		\node [minimum size=0cm,draw,circle] (56c) at (4,10.5) {$ 0$};
   	\node [minimum size=0cm,draw,circle] (57c) at (5,10.5) {$ 6$};
		\node [minimum size=0cm,draw,circle] (58c) at (6,10.5) {$ 4$};
			
		\node [minimum size=0cm,draw,circle] (59c) at (0.5,10) {$ 1$};
		\node [minimum size=0cm,draw,circle] (60c) at (1.5,10) {$ 0$};
		\node [minimum size=0cm,draw,circle] (61c) at (2.5,10) {$ 5$};
		\node [minimum size=0cm,draw,circle] (62c) at (3.5,10) {$ 4$};
	  \node [minimum size=0cm,draw,circle] (63c) at (4.5,10) {$ 3$};
		\node [minimum size=0cm,draw,circle] (64c) at (5.5,10) {$ 2$};
			
	  \node [minimum size=0cm,draw,circle] (65c) at (0,9.5) {$ 0$};
		\node [minimum size=0cm,draw,circle] (66c) at (1,9.5) {$ 5$};
		\node [minimum size=0cm,draw,circle] (67c) at (2,9.5) {$ 4$};
		\node [minimum size=0cm,draw,circle] (68c) at (3,9.5) {$ 3$};
		\node [minimum size=0cm,draw,circle] (69c) at (4,9.5) {$ 2$};
	  \node [minimum size=0cm,draw,circle] (70c) at (5,9.5) {$ 1$};
		\node [minimum size=0cm,draw,circle] (71c) at (6,9.5) {$ 0$};
			
		\node [minimum size=0cm,draw,circle] (72c) at (0.5,9) {$ 4$};
		\node [minimum size=0cm,draw,circle] (73c) at (1.5,9) {$ 3$};
		\node [minimum size=0cm,draw,circle] (74c) at (2.5,9) {$ 2$};
		\node [minimum size=0cm,draw,circle] (75c) at (3.5,9) {$ 1$};
    \node [minimum size=0cm,draw,circle] (76c) at (4.5,9) {$ 0$};
		\node [minimum size=0cm,draw,circle] (77c) at (5.5,9) {$ 5$};
		
		\node [minimum size=0cm,draw,circle] (78c) at (0,8.5) {$ 3$};
		\node [minimum size=0cm,draw,circle] (79c) at (1,8.5) {$ 2$};
		\node [minimum size=0cm,draw,circle] (80c) at (2,8.5) {$ 1$};
		\node [minimum size=0cm,draw,circle] (81c) at (3,8.5) {$ 0$};
		\node [minimum size=0cm,draw,circle] (82c) at (4,8.5) {$ 5$};
  	\node [minimum size=0cm,draw,circle] (83c) at (5,8.5) {$ 4$};
		\node [minimum size=0cm,draw,circle] (84c) at (6,8.5){$ 3$};
			
		\node [minimum size=0cm,draw,circle] (85c) at (0.5,8) {$ 1$};
		\node [minimum size=0cm,draw,circle] (86c) at (1.5,8) {$ 0$};
		\node [minimum size=0cm,draw,circle] (87c) at (2.5,8) {$ 5$};
		\node [minimum size=0cm,draw,circle] (88c) at (3.5,8) {$ 4$};
		\node [minimum size=0cm,draw,circle] (89c) at (4.5,8) {$ 3$};
		\node [minimum size=0cm,draw,circle] (90c) at (5.5,8) {$ 2$};
		
	  \node [minimum size=0cm,draw,circle] (91c) at (0,7.5) {$ 0$};
		\node [minimum size=0cm,draw,circle] (92c) at (1,7.5) {$ 5$};
		\node [minimum size=0cm,draw,circle] (93c) at (2,7.5) {$ 4$};
		\node [minimum size=0cm,draw,circle] (94c) at (3,7.5) {$ 3$};
		\node [minimum size=0cm,draw,circle] (95c) at (4,7.5) {$ 2$};
		\node [minimum size=0cm,draw,circle] (96c) at (5,7.5) {$ 1$};
		\node [minimum size=0cm,draw,circle] (97c) at (6,7.5) {$ 0$};
		\node [minimum size=0cm,] (97cc) at (3,7.0) { $C_{14}\times C_{12}$};
		\node [minimum size=0cm,] (97cc) at (-4.5,7.0) {$C_{14}\times C_{10}$};
		
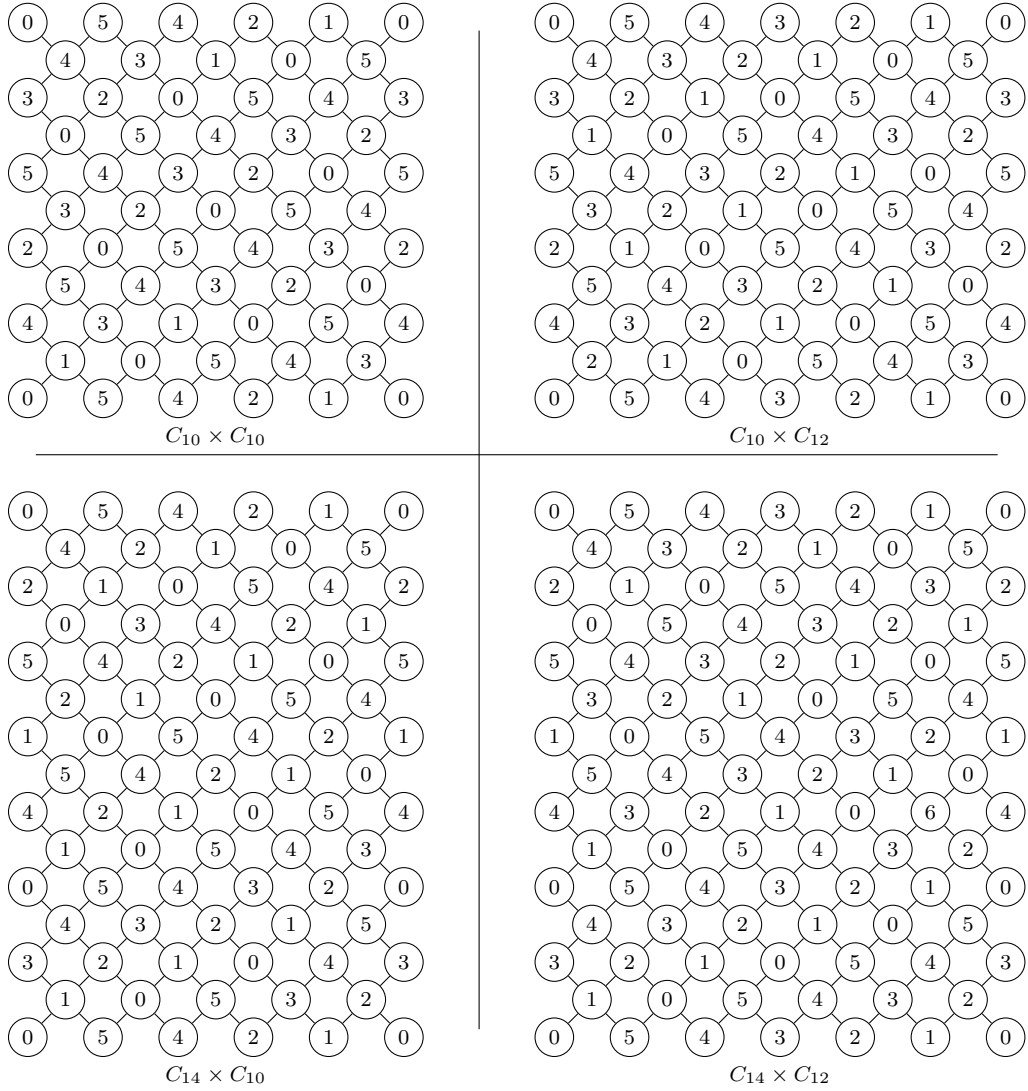
		\node [minimum size=0cm,] (97h) at (-0.5,6) {Fig. $13$: $5-L(1,1)-$labeling of $C_{10m'+14m''} \times C_{10n'+14n''}$, for all $m',m'', n', n'' \in \left\{0,1\right\}$.};

\end{pgfonlayer}
	\begin{pgfonlayer}{edgelayer}

   \draw [thin=1.00] (0) to (6);
   \draw [thin=1.00] (1) to (6);
   \draw [thin=1.00] (1) to (7);
   \draw [thin=1.00] (2) to (7);
	  \draw [thin=1.00] (2) to (8);
		\draw [thin=1.00] (3) to (8);
		\draw [thin=1.00] (3) to (9);
		\draw [thin=1.00] (4) to (9);
		\draw [thin=1.00] (4) to (10);
		\draw [thin=1.00] (5) to (10);
		%\draw [thin=1.00] (5) to (11);
		\draw [thin=1.00] (6) to (11);
		\draw [thin=1.00] (6) to (12);
		\draw [thin=1.00] (7) to (12);
		\draw [thin=1.00] (7) to (13);
		\draw [thin=1.00] (8) to (13);
		\draw [thin=1.00] (8) to (14);
   \draw [thin=1.00] (9) to (14);
   \draw [thin=1.00] (9) to (15);
   \draw [thin=1.00] (10) to (15);
   \draw [thin=1.00] (10) to (16);
   %\draw [thin=1.00] (11) to (16);
		\draw [thin=1.00] (11) to (17);
		\draw [thin=1.00] (12) to (17);
		\draw [thin=1.00] (12) to (18);
		\draw [thin=1.00] (13) to (18);
		\draw [thin=1.00] (13) to (19);
		\draw [thin=1.00] (14) to (19);
		\draw [thin=1.00] (14) to (20);
		\draw [thin=1.00] (15) to (20);
		\draw [thin=1.00] (15) to (21);
		\draw [thin=1.00] (16) to (21);
	%\draw [thin=1.00] (16) to (22);
   \draw [thin=1.00] (17) to (22);
   \draw [thin=1.00] (17) to (23);
   \draw [thin=1.00] (18) to (23);
   \draw [thin=1.00] (18) to (24);
   \draw [thin=1.00] (19) to (24);
		\draw [thin=1.00] (19) to (25);
		\draw [thin=1.00] (20) to (25);
	  \draw [thin=1.00] (20) to (26);
		\draw [thin=1.00] (21) to (26);		
		\draw [thin=1.00] (21) to (27);
		%\draw [thin=1.00] (22) to (27);
		\draw [thin=1.00] (22) to (28);
		\draw [thin=1.00] (23) to (28);
		\draw [thin=1.00] (23) to (29);
		\draw [thin=1.00] (24) to (29);
		\draw [thin=1.00] (24) to (30);
		\draw [thin=1.00] (25) to (30);
		\draw [thin=1.00] (25) to (31);
		\draw [thin=1.00] (26) to (31);
		\draw [thin=1.00] (26) to (32);
   \draw [thin=1.00] (27) to (32);
   %\draw [thin=1.00] (27) to (33);
   \draw [thin=1.00] (28) to (33);
   \draw [thin=1.00] (28) to (34);
   \draw [thin=1.00] (29) to (34);
	  \draw [thin=1.00] (29) to (35);
		\draw [thin=1.00] (30) to (35);
		\draw [thin=1.00] (30) to (36);
		\draw [thin=1.00] (31) to (36);
		\draw [thin=1.00] (31) to (37);
		\draw [thin=1.00] (32) to (37);
		\draw [thin=1.00] (32) to (38);
		%\draw [thin=1.00] (33) to (38);
		\draw [thin=1.00] (33) to (39);
		\draw [thin=1.00] (34) to (39);
		\draw [thin=1.00] (34) to (40);
		\draw [thin=1.00] (35) to (40);
		\draw [thin=1.00] (35) to (41);
   \draw [thin=1.00] (36) to (41);
   \draw [thin=1.00] (36) to (42);
   \draw [thin=1.00] (37) to (42);
   \draw [thin=1.00] (37) to (43);
   \draw [thin=1.00] (38) to (43);
		%\draw [thin=1.00] (38) to (44);
		\draw [thin=1.00] (39) to (44);
		\draw [thin=1.00] (39) to (45);
		\draw [thin=1.00] (40) to (45);
		\draw [thin=1.00] (40) to (46);
		\draw [thin=1.00] (41) to (46);
		\draw [thin=1.00] (41) to (47);
		\draw [thin=1.00] (42) to (47);
		\draw [thin=1.00] (42) to (48);
		\draw [thin=1.00] (43) to (48);
	\draw [thin=1.00] (43) to (49);
   %\draw [thin=1.00] (44) to (49);
   \draw [thin=1.00] (44) to (50);
   \draw [thin=1.00] (45) to (50);
   \draw [thin=1.00] (45) to (51);
   \draw [thin=1.00] (46) to (51);
		\draw [thin=1.00] (46) to (52);
		\draw [thin=1.00] (47) to (52);
	  \draw [thin=1.00] (47) to (53);
		\draw [thin=1.00] (48) to (53);		
		\draw [thin=1.00] (48) to (54);
		\draw [thin=1.00] (49) to (54);
		%\draw [thin=1.00] (49) to (55');
		\draw [thin=1.00] (50) to (55');
		\draw [thin=1.00] (50) to (56');
	  \draw [thin=1.00] (51) to (56');
		\draw [thin=1.00] (51) to (57');		
		\draw [thin=1.00] (52) to (57');
		\draw [thin=1.00] (52) to (58');
		\draw [thin=1.00] (53) to (58');		
		\draw [thin=1.00] (53) to (59');
		\draw [thin=1.00] (54) to (59');
		\draw [thin=1.00] (54) to (60');

   \draw [thin=1.00] (55) to (61);
   \draw [thin=1.00] (56) to (61);
   \draw [thin=1.00] (56) to (62);
   \draw [thin=1.00] (57) to (62);
	  \draw [thin=1.00] (57) to (63);
		\draw [thin=1.00] (58) to (63);
		\draw [thin=1.00] (58) to (64);
		\draw [thin=1.00] (59) to (64);
		\draw [thin=1.00] (59) to (65);
		\draw [thin=1.00] (60) to (65);
		\draw [thin=1.00] (60) to (65b);
		\draw [thin=1.00] (60b) to (65b);
		
		%\draw [thin=1.00] (60) to (66);
		\draw [thin=1.00] (61) to (66);
		\draw [thin=1.00] (61) to (67);
		\draw [thin=1.00] (62) to (67);
		\draw [thin=1.00] (62) to (68);
		\draw [thin=1.00] (63) to (68);
		\draw [thin=1.00] (63) to (69);
   \draw [thin=1.00] (64) to (69);
   \draw [thin=1.00] (64) to (70);
   \draw [thin=1.00] (65) to (70);
   \draw [thin=1.00] (65) to (71);
   \draw [thin=1.00] (65b) to (71);
	\draw [thin=1.00] (65b) to (71b);

   %\draw [thin=1.00] (66) to (71);
		\draw [thin=1.00] (66) to (72);
		\draw [thin=1.00] (67) to (72);
		\draw [thin=1.00] (67) to (73);
		\draw [thin=1.00] (68) to (73);
		\draw [thin=1.00] (68) to (74);
		\draw [thin=1.00] (69) to (74);
		\draw [thin=1.00] (69) to (75);
		\draw [thin=1.00] (70) to (75);
		\draw [thin=1.00] (70) to (76);
		\draw [thin=1.00] (71) to (76);
		\draw [thin=1.00] (71) to (76b);
		\draw [thin=1.00] (71b) to (76b);
		
	%\draw [thin=1.00] (71) to (77);
   \draw [thin=1.00] (72) to (77);
   \draw [thin=1.00] (72) to (78);
   \draw [thin=1.00] (73) to (78);
   \draw [thin=1.00] (73) to (79);
   \draw [thin=1.00] (74) to (79);
		\draw [thin=1.00] (74) to (80);
		\draw [thin=1.00] (75) to (80);
	  \draw [thin=1.00] (75) to (81);
		\draw [thin=1.00] (76) to (81);		
		\draw [thin=1.00] (76) to (82);
		\draw [thin=1.00] (76b) to (82);
		\draw [thin=1.00] (76b) to (82b);
		
		%\draw [thin=1.00] (77) to (82);
		\draw [thin=1.00] (77) to (83);
		\draw [thin=1.00] (78) to (83);
		\draw [thin=1.00] (78) to (84);
		\draw [thin=1.00] (79) to (84);
		\draw [thin=1.00] (79) to (85);
		\draw [thin=1.00] (80) to (85);
		\draw [thin=1.00] (80) to (86);
		\draw [thin=1.00] (81) to (86);
		\draw [thin=1.00] (81) to (87);
   \draw [thin=1.00] (82) to (87);
   \draw [thin=1.00] (82) to (87b);
		\draw [thin=1.00] (82b) to (87b);

   %\draw [thin=1.00] (82) to (88);
   \draw [thin=1.00] (83) to (88);
   \draw [thin=1.00] (83) to (89);
   \draw [thin=1.00] (84) to (89);
	  \draw [thin=1.00] (84) to (90);
		\draw [thin=1.00] (85) to (90);
		\draw [thin=1.00] (85) to (91);
		\draw [thin=1.00] (86) to (91);
		\draw [thin=1.00] (86) to (92);
		\draw [thin=1.00] (87) to (92);
		\draw [thin=1.00] (87) to (93);
		\draw [thin=1.00] (87b) to (93);
		\draw [thin=1.00] (87b) to (93b);
		
		%\draw [thin=1.00] (88) to (93);
		\draw [thin=1.00] (88) to (94);
		\draw [thin=1.00] (89) to (94);
		\draw [thin=1.00] (89) to (95);
		\draw [thin=1.00] (90) to (95);
		\draw [thin=1.00] (90) to (96);
   \draw [thin=1.00] (91) to (96);
   \draw [thin=1.00] (91) to (97);
   \draw [thin=1.00] (92) to (97);
   \draw [thin=1.00] (92) to (98);
   \draw [thin=1.00] (93) to (98);
   \draw [thin=1.00] (93) to (98b);
	 \draw [thin=1.00] (93b) to (98b);

		%\draw [thin=1.00] (93) to (99);
		\draw [thin=1.00] (94) to (99);
		\draw [thin=1.00] (94) to (100);
		\draw [thin=1.00] (95) to (100);
		\draw [thin=1.00] (95) to (101);
		\draw [thin=1.00] (96) to (101);
		\draw [thin=1.00] (96) to (102);
		\draw [thin=1.00] (97) to (102);
		\draw [thin=1.00] (97) to (103);
		\draw [thin=1.00] (98) to (103);
	  \draw [thin=1.00] (98) to (104);
	  \draw [thin=1.00] (98b) to (104);
		\draw [thin=1.00] (98b) to (104b);
	
   %\draw [thin=1.00] (99) to (104);
   \draw [thin=1.00] (99) to (105);
   \draw [thin=1.00] (100) to (105);
   \draw [thin=1.00] (100) to (106);
   \draw [thin=1.00] (101) to (106);
		\draw [thin=1.00] (101) to (107);
		\draw [thin=1.00] (102) to (107);
	  \draw [thin=1.00] (102) to (108);
		\draw [thin=1.00] (103) to (108);		
		\draw [thin=1.00] (103) to (109);
		\draw [thin=1.00] (104) to (109);
		\draw [thin=1.00] (104) to (109b);
		\draw [thin=1.00] (104b) to (109b);
		
		%\draw [thin=1.00] (104) to (110);
   \draw [thin=1.00] (105) to (110);
   \draw [thin=1.00] (105) to (111);
   \draw [thin=1.00] (106) to (111);
		\draw [thin=1.00] (106) to (112);
		\draw [thin=1.00] (107) to (112);
	  \draw [thin=1.00] (107) to (113);
		\draw [thin=1.00] (108) to (113);		
		\draw [thin=1.00] (108) to (114);
		\draw [thin=1.00] (109) to (114);
   \draw [thin=1.00] (109) to (115);
   \draw [thin=1.00] (109b) to (115);
		\draw [thin=1.00] (109b) to (115b);

   	\draw [thin=1.00] (55a) to (61a);
   \draw [thin=1.00] (56a) to (61a);
   \draw [thin=1.00] (56a) to (62a);
   \draw [thin=1.00] (57a) to (62a);
	  \draw [thin=1.00] (57a) to (63a);
		\draw [thin=1.00] (58a) to (63a);
		\draw [thin=1.00] (58a) to (64a);
		\draw [thin=1.00] (59a) to (64a);
		\draw [thin=1.00] (59a) to (65a);
		\draw [thin=1.00] (60a) to (65a);
			%\draw [thin=1.00] (60a) to (66a);
		\draw [thin=1.00] (61a) to (66a);
		\draw [thin=1.00] (61a) to (67a);
		\draw [thin=1.00] (62a) to (67a);
		\draw [thin=1.00] (62a) to (68a);
		\draw [thin=1.00] (63a) to (68a);
		\draw [thin=1.00] (63a) to (69a);
   \draw [thin=1.00] (64a) to (69a);
   \draw [thin=1.00] (64a) to (70a);
   \draw [thin=1.00] (65a) to (70a);
   \draw [thin=1.00] (65a) to (71a);
   %\draw [thin=1.00] (66a) to (71a);
		\draw [thin=1.00] (66a) to (72a);
		\draw [thin=1.00] (67a) to (72a);
		\draw [thin=1.00] (67a) to (73a);
		\draw [thin=1.00] (68a) to (73a);
		\draw [thin=1.00] (68a) to (74a);
		\draw [thin=1.00] (69a) to (74a);
		\draw [thin=1.00] (69a) to (75a);
		\draw [thin=1.00] (70a) to (75a);
		\draw [thin=1.00] (70a) to (76a);
		\draw [thin=1.00] (71a) to (76a);
	%\draw [thin=1.00] (71a) to (77a);
   \draw [thin=1.00] (72a) to (77a);
   \draw [thin=1.00] (72a) to (78a);
   \draw [thin=1.00] (73a) to (78a);
   \draw [thin=1.00] (73a) to (79a);
   \draw [thin=1.00] (74a) to (79a);
		\draw [thin=1.00] (74a) to (80a);
		\draw [thin=1.00] (75a) to (80a);
	  \draw [thin=1.00] (75a) to (81a);
		\draw [thin=1.00] (76a) to (81a);		
		\draw [thin=1.00] (76a) to (82a);
		%\draw [thin=1.00] (77a) to (82a);
		\draw [thin=1.00] (77a) to (83a);
		\draw [thin=1.00] (78a) to (83a);
		\draw [thin=1.00] (78a) to (84a);
		\draw [thin=1.00] (79a) to (84a);
		\draw [thin=1.00] (79a) to (85a);
		\draw [thin=1.00] (80a) to (85a);
		\draw [thin=1.00] (80a) to (86a);
		\draw [thin=1.00] (81a) to (86a);
		\draw [thin=1.00] (81a) to (87a);
   \draw [thin=1.00] (82a) to (87a);
   %\draw [thin=1.00] (82a) to (88a);
   \draw [thin=1.00] (83a) to (88a);
   \draw [thin=1.00] (83a) to (89a);
   \draw [thin=1.00] (84a) to (89a);
	  \draw [thin=1.00] (84a) to (90a);
		\draw [thin=1.00] (85a) to (90a);
		\draw [thin=1.00] (85a) to (91a);
		\draw [thin=1.00] (86a) to (91a);
		\draw [thin=1.00] (86a) to (92a);
		\draw [thin=1.00] (87a) to (92a);
		\draw [thin=1.00] (87a) to (93a);
		%\draw [thin=1.00] (88a) to (93a);
		\draw [thin=1.00] (88a) to (94a);
		\draw [thin=1.00] (89a) to (94a);
		\draw [thin=1.00] (89a) to (95a);
		\draw [thin=1.00] (90a) to (95a);
		\draw [thin=1.00] (90a) to (96a);
   \draw [thin=1.00] (91a) to (96a);
   \draw [thin=1.00] (91a) to (97a);
   \draw [thin=1.00] (92a) to (97a);
   \draw [thin=1.00] (92a) to (98a);
   \draw [thin=1.00] (93a) to (98a);
		%\draw [thin=1.00] (93a) to (99a);
		\draw [thin=1.00] (94a) to (99a);
		\draw [thin=1.00] (94a) to (100a);
		\draw [thin=1.00] (95a) to (100a);
		\draw [thin=1.00] (95a) to (101a);
		\draw [thin=1.00] (96a) to (101a);
		\draw [thin=1.00] (96a) to (102a);
		\draw [thin=1.00] (97a) to (102a);
		\draw [thin=1.00] (97a) to (103a);
		\draw [thin=1.00] (98a) to (103a);
	\draw [thin=1.00] (98a) to (104a);
   %\draw [thin=1.00] (99a) to (104a);
   \draw [thin=1.00] (99a) to (105a);
   \draw [thin=1.00] (100a) to (105a);
   \draw [thin=1.00] (100a) to (106a);
   \draw [thin=1.00] (101a) to (106a);
		\draw [thin=1.00] (101a) to (107a);
		\draw [thin=1.00] (102a) to (107a);
	  \draw [thin=1.00] (102a) to (108a);
		\draw [thin=1.00] (103a) to (108a);		
		\draw [thin=1.00] (103a) to (109a);
		\draw [thin=1.00] (104a) to (109a);
		%\draw [thin=1.00] (104a) to (110a);
   \draw [thin=1.00] (105a) to (110a);
   \draw [thin=1.00] (105a) to (111a);
   \draw [thin=1.00] (106a) to (111a);
		\draw [thin=1.00] (106a) to (112a);
		\draw [thin=1.00] (107a) to (112a);
	  \draw [thin=1.00] (107a) to (113a);
		\draw [thin=1.00] (108a) to (113a);		
		\draw [thin=1.00] (108a) to (114a);
		\draw [thin=1.00] (109a) to (114a);
   \draw [thin=1.00] (109a) to (115a);
   %\draw [thin=1.00] (110a) to (115a);
   \draw [thin=1.00] (110a) to (116a);
		\draw [thin=1.00] (111a) to (116a);
		\draw [thin=1.00] (111a) to (117a);
	  \draw [thin=1.00] (112a) to (117a);
		\draw [thin=1.00] (112a) to (118a);		
		\draw [thin=1.00] (113a) to (118a);
		\draw [thin=1.00] (113a) to (119a);
		\draw [thin=1.00] (114a) to (119a);		
		\draw [thin=1.00] (114a) to (120a);
		\draw [thin=1.00] (115a) to (120a);
		
		%\draw [thin=1.00] (115a) to (121a);
		\draw [thin=1.00] (116a) to (121a);
		\draw [thin=1.00] (116a) to (122a);
	  \draw [thin=1.00] (117a) to (122a);
		\draw [thin=1.00] (117a) to (123a);		
		\draw [thin=1.00] (118a) to (123a);
		\draw [thin=1.00] (118a) to (124a);
		\draw [thin=1.00] (119a) to (124a);		
		\draw [thin=1.00] (119a) to (125a);
		\draw [thin=1.00] (120a) to (125a);
		\draw [thin=1.00] (120a) to (126a);
		
		%\draw [thin=1.00] (121a) to (126a);
   \draw [thin=1.00] (121a) to (127a);
		\draw [thin=1.00] (122a) to (127a);
		\draw [thin=1.00] (122a) to (128a);
	  \draw [thin=1.00] (123a) to (128a);
		\draw [thin=1.00] (123a) to (129a);		
		\draw [thin=1.00] (124a) to (129a);
		\draw [thin=1.00] (124a) to (130a);
		\draw [thin=1.00] (125a) to (130a);		
		\draw [thin=1.00] (125a) to (131a);
		\draw [thin=1.00] (126a) to (131a);
		
		%\draw [thin=1.00] (126a) to (132a);
		\draw [thin=1.00] (127a) to (132a);
		\draw [thin=1.00] (127a) to (133a);
	  \draw [thin=1.00] (128a) to (133a);
		\draw [thin=1.00] (128a) to (134a);		
		\draw [thin=1.00] (129a) to (134a);
		\draw [thin=1.00] (129a) to (135a);
		\draw [thin=1.00] (130a) to (135a);		
		\draw [thin=1.00] (130a) to (136a);
		\draw [thin=1.00] (131a) to (136a);
		\draw [thin=1.00] (131a) to (137a);	
		
		\draw [thin=1.00] (0c) to (7c);
   \draw [thin=1.00] (1c) to (7c);
   \draw [thin=1.00] (1c) to (8c);
   \draw [thin=1.00] (2c) to (8c);
	  \draw [thin=1.00] (2c) to (9c);
		\draw [thin=1.00] (3c) to (9c);
		\draw [thin=1.00] (3c) to (10c);
		\draw [thin=1.00] (4c) to (10c);
		\draw [thin=1.00] (4c) to (11c);
		\draw [thin=1.00] (5c) to (11c);
		\draw [thin=1.00] (5c) to (12c);
		\draw [thin=1.00] (6c) to (12c);

		\draw [thin=1.00] (7c) to (13c);
		\draw [thin=1.00] (7c) to (14c);
		\draw [thin=1.00] (8c) to (14c);
		\draw [thin=1.00] (8c) to (15c);
		\draw [thin=1.00] (9c) to (15c);
		\draw [thin=1.00] (9c) to (16c);
   \draw [thin=1.00] (10c) to (16c);
   \draw [thin=1.00] (10c) to (17c);
   \draw [thin=1.00] (11c) to (17c);
		\draw [thin=1.00] (11c) to (18c);
		\draw [thin=1.00] (12c) to (18c);
		\draw [thin=1.00] (12c) to (19c);

		\draw [thin=1.00] (13c) to (20c);
		\draw [thin=1.00] (14c) to (20c);
		\draw [thin=1.00] (14c) to (21c);
		\draw [thin=1.00] (15c) to (21c);
		\draw [thin=1.00] (15c) to (22c);
		\draw [thin=1.00] (16c) to (22c);
		\draw [thin=1.00] (16c) to (23c);
		\draw [thin=1.00] (17c) to (23c);
		\draw [thin=1.00] (17c) to (24c);
		\draw [thin=1.00] (18c) to (24c);
		\draw [thin=1.00] (18c) to (25c);
		\draw [thin=1.00] (19c) to (25c);
		
   \draw [thin=1.00] (20c) to (26c);
   \draw [thin=1.00] (20c) to (27c);
   \draw [thin=1.00] (21c) to (27c);
   \draw [thin=1.00] (21c) to (28c);
   \draw [thin=1.00] (22c) to (28c);
		\draw [thin=1.00] (22c) to (29c);
		\draw [thin=1.00] (23c) to (29c);
	  \draw [thin=1.00] (23c) to (30c);
	  \draw [thin=1.00] (24c) to (30c);
		\draw [thin=1.00] (24c) to (31c);
		\draw [thin=1.00] (25c) to (31c);
		\draw [thin=1.00] (25c) to (32c);

		\draw [thin=1.00] (26c) to (33c);
		\draw [thin=1.00] (27c) to (33c);
		\draw [thin=1.00] (27c) to (34c);
		\draw [thin=1.00] (28c) to (34c);
		\draw [thin=1.00] (28c) to (35c);
		\draw [thin=1.00] (29c) to (35c);
		\draw [thin=1.00] (29c) to (36c);
		\draw [thin=1.00] (30c) to (36c);
		\draw [thin=1.00] (30c) to (37c);
		\draw [thin=1.00] (31c) to (37c);
		\draw [thin=1.00] (31c) to (38c);
		\draw [thin=1.00] (32c) to (38c);

    \draw [thin=1.00] (33c) to (39c);
    \draw [thin=1.00] (33c) to (40c);
    \draw [thin=1.00] (34c) to (40c);
	  \draw [thin=1.00] (34c) to (41c);
		\draw [thin=1.00] (35c) to (41c);
		\draw [thin=1.00] (35c) to (42c);
		\draw [thin=1.00] (36c) to (42c);
		\draw [thin=1.00] (36c) to (43c);
		\draw [thin=1.00] (37c) to (43c);
		\draw [thin=1.00] (37c) to (44c);
		\draw [thin=1.00] (38c) to (44c);
		\draw [thin=1.00] (38c) to (45c);

		\draw [thin=1.00] (39c) to (46c);
		\draw [thin=1.00] (40c) to (46c);
		\draw [thin=1.00] (40c) to (47c);
		\draw [thin=1.00] (41c) to (47c);
		\draw [thin=1.00] (41c) to (48c);
   \draw [thin=1.00] (42c) to (48c);
   \draw [thin=1.00] (42c) to (49c);
   \draw [thin=1.00] (43c) to (49c);
   \draw [thin=1.00] (43c) to (50c);
		\draw [thin=1.00] (44c) to (50c);
		\draw [thin=1.00] (44c) to (51c);
		\draw [thin=1.00] (45c) to (51c);

		\draw [thin=1.00] (46c) to (52c);
		\draw [thin=1.00] (46c) to (53c);
		\draw [thin=1.00] (47c) to (53c);
		\draw [thin=1.00] (47c) to (54c);
		\draw [thin=1.00] (48c) to (54c);
		\draw [thin=1.00] (48c) to (55c);
		\draw [thin=1.00] (49c) to (55c);
		\draw [thin=1.00] (49c) to (56c);
		\draw [thin=1.00] (50c) to (56c);
		\draw [thin=1.00] (50c) to (57c);
		\draw [thin=1.00] (51c) to (57c);
		\draw [thin=1.00] (51c) to (58c);

   \draw [thin=1.00] (52c) to (59c);
   \draw [thin=1.00] (53c) to (59c);
   \draw [thin=1.00] (53c) to (60c);
   \draw [thin=1.00] (54c) to (60c);
		\draw [thin=1.00] (54c) to (61c);
		\draw [thin=1.00] (55c) to (61c);
	  \draw [thin=1.00] (55c) to (62c);
		\draw [thin=1.00] (56c) to (62c);
		\draw [thin=1.00] (56c) to (63c);
		\draw [thin=1.00] (57c) to (63c);
		\draw [thin=1.00] (57c) to (64c);
		\draw [thin=1.00] (58c) to (64c);

		\draw [thin=1.00] (59c) to (65c);
		\draw [thin=1.00] (59c) to (66c);
	  \draw [thin=1.00] (60c) to (66c);
		\draw [thin=1.00] (60c) to (67c);		
		\draw [thin=1.00] (61c) to (67c);
		\draw [thin=1.00] (61c) to (68c);
		\draw [thin=1.00] (62c) to (68c);		
		\draw [thin=1.00] (62c) to (69c);
		\draw [thin=1.00] (63c) to (69c);
		\draw [thin=1.00] (63c) to (70c);
		\draw [thin=1.00] (64c) to (70c);
		\draw [thin=1.00] (64c) to (71c);

   	\draw [thin=1.00] (65c) to (72c);
		\draw [thin=1.00] (66c) to (72c);
		\draw [thin=1.00] (66c) to (73c);
		\draw [thin=1.00] (67c) to (73c);
		\draw [thin=1.00] (67c) to (74c);
   \draw [thin=1.00] (68c) to (74c);
   \draw [thin=1.00] (68c) to (75c);
   \draw [thin=1.00] (69c) to (75c);
   \draw [thin=1.00] (69c) to (76c);
		\draw [thin=1.00] (70c) to (76c);
		\draw [thin=1.00] (70c) to (77c);
		\draw [thin=1.00] (71c) to (77c);

		\draw [thin=1.00] (72c) to (78c);
		\draw [thin=1.00] (72c) to (79c);
		\draw [thin=1.00] (73c) to (79c);
		\draw [thin=1.00] (73c) to (80c);
		\draw [thin=1.00] (74c) to (80c);
		\draw [thin=1.00] (74c) to (81c);
		\draw [thin=1.00] (75c) to (81c);
		\draw [thin=1.00] (75c) to (82c);
		\draw [thin=1.00] (76c) to (82c);
		\draw [thin=1.00] (76c) to (83c);
		\draw [thin=1.00] (77c) to (83c);
		\draw [thin=1.00] (77c) to (84c);

   \draw [thin=1.00] (78c) to (85c);
   \draw [thin=1.00] (79c) to (85c);
   \draw [thin=1.00] (79c) to (86c);
   \draw [thin=1.00] (80c) to (86c);
		\draw [thin=1.00] (80c) to (87c);
		\draw [thin=1.00] (81c) to (87c);
	  \draw [thin=1.00] (81c) to (88c);
		\draw [thin=1.00] (82c) to (88c);
		\draw [thin=1.00] (82c) to (89c);
		\draw [thin=1.00] (83c) to (89c);
		\draw [thin=1.00] (83c) to (90c);
		\draw [thin=1.00] (84c) to (90c);

		\draw [thin=1.00] (85c) to (91c);
		\draw [thin=1.00] (85c) to (92c);
	  \draw [thin=1.00] (86c) to (92c);
		\draw [thin=1.00] (86c) to (93c);		
		\draw [thin=1.00] (87c) to (93c);
		\draw [thin=1.00] (87c) to (94c);
		\draw [thin=1.00] (88c) to (94c);		
		\draw [thin=1.00] (88c) to (95c);
		\draw [thin=1.00] (89c) to (95c);
		\draw [thin=1.00] (89c) to (96c);
		\draw [thin=1.00] (90c) to (96c);
		\draw [thin=1.00] (90c) to (97c);
		
		\draw [thin=1.00] (0g) to (0d);
		\draw [thin=1.00] (0e) to (0f);

		\end{pgfonlayer}

\end{tikzpicture}
				
}}
	
	\end{proof}
	
%\newpage	
\begin{cor} For $m \geq 48$ and $n \geq 40$, $\lambda_1^1(C_m \times C_n)=5$.
\end{cor}
 The last corollary gave the values of $m,n$ beyond which $\lambda_1^1(C_m \times C_n)=5$. However, there are smaller product graphs whose $L(1,1)$-number is $5$ as demonstrated in the next corollary.
\begin{cor} \label{cor1} For all $m,n \geq 14$, $m \notin \left\{14,16,18,22,26,32,36,46\right\}$, $n \notin \left\{14,16,18,26,28,34\right\}$, $\lambda_1^1(C_m \times C_n)=5$.
\end{cor}
For some of pairs  $\left\{m',n'\right\}$, in the two sets  defined above, namely $\left\{16,18\right\}$,\\$\left\{18,28\right\}$,$\left\{32,18\right\}$,$\left\{36,16\right\}$,$\left\{36,18\right\}$,$\left\{36,2  8\right\},$ $\lambda_1^1(C_m \times C_n)=5$. This is obvious from earlier results. For the remaining pairs, it can easily be confirmed, by manual labeling, that  $\lambda_1^1(C_m \times C_n) \leq 6$. However, we observe that there could be a better upper bound and therefore, present the following conjecture:

\begin{conj}\label{conj1} For $m,n \geq 12$, $m,n \not\equiv 0\mod 5$, $\lambda_1^1(C_m \times C_n)=5$.
\end{conj}
Solving this conjecture only requires confirming the $5-L(1,1)$-labeling for $C_m \times C_n$, where $m,n$ are the remaining pair yet to be confirmed in the sets in Corollary \ref{cor1}.

The results obtained is summarized in the table below:

\begin{center}

 $\begin{array}{|c|c|c|} \hline
 m & n & \lambda_1^1(C_m \times C_n)\\ \hline
 4 & 4,5,8,10 & 7\\ \hline
 4 & n\not\equiv 0 \mod3 \ \&  \ n \geq 11 & 6\\ \hline
 m \equiv 0 \mod 4 & n\equiv 0\mod 3 & 5\\ \hline
 m\equiv 0 \mod 5 & n \equiv 0 \mod 5 & 4 \\ \hline
 6 & 3, 6 & 8\\ \hline
 6 & 5, 10 & 7 \\ \hline
 6 & 7, 9, 11, 14+4n', n' \geq 0 & 6\\ \hline
 8 & 8 & 7 \\ \hline
 8 & n \geq 10 ,  n \not\equiv 0 \mod 3 & 6 \\ \hline
 m \equiv 0\mod 10 & n\geq 11, n \not\equiv 0 \mod 5 & 5\\ \hline
 12 & n\geq 12 & 5 \\ \hline
 m \geq 40, even & n \geq 48 & 5 \\ \hline

\end{array}$

\end{center}

\end{document}